\documentclass{amsart}
\usepackage{amssymb}
\usepackage{mathrsfs}
\usepackage{mathabx}
\usepackage{amscd} 
 \usepackage[normalem]{ulem}
 \usepackage{color}\long\def\red#1{{\color{red}#1}}

\long\def\red#1{#1}

\usepackage{hyperref}
\usepackage{graphicx}
\usepackage[usenames,dvipsnames,x11names]{xcolor}
\usepackage{tikz}
\usetikzlibrary{arrows,snakes,shapes,shadows}
\usepackage{verbatim}

\newcommand{\grassman}{\mathbf G}

\newcommand\carrousel{carrousel decomposition}
\newcommand\carrousels{carrousel decompositions}
\newcommand\denom{{\operatorname{denom}}}
\let\cal\mathcal
\renewcommand{\setminus}{\smallsetminus}
\newcommand\Q{{\mathbb Q}}
\newcommand\R{{\mathbb R}}
\newcommand\C{{\mathbb C}}
\newcommand\Z{{\mathbb Z}}
\newcommand\N{{\mathbb N}}
\renewcommand\L{{$\cal L$}}
\renewcommand\P{{$\cal P$}}

\newcommand\Nn{{\mathcal N}}

\newtheorem{theorem}{Theorem}[section]
\newtheorem{proposition}[theorem]{Proposition}
\newtheorem*{theorem*}{Theorem}
\newtheorem{corollary}[theorem]{Corollary}
\newtheorem{lemma}[theorem]{Lemma}

\theoremstyle{definition}
\newtheorem{amalgamation}[theorem]{Amalgamation}
\newtheorem*{amalgamation*}{Amalgamation}
\newtheorem{example}[theorem]{Example}
\newtheorem{remark}[theorem]{Remark}
\newtheorem*{remark*}{Remark}
\newtheorem{definition}[theorem]{Definition}

\renewcommand{\int}{\operatorname{int}}
\newcommand{\lcm}{\operatorname{lcm}}

\begin{document}
\title[Lipschitz geometry, analytic invariants and equisingularity]{Lipschitz geometry of
  complex surfaces: \\analytic invariants and equisingularity.}
\author{Walter D Neumann}\address{Department of Mathematics, Barnard
  College, Columbia University, 2009 Broadway MC4429, New York, NY
  10027, USA} \email{neumann@math.columbia.edu} 
  \author{Anne Pichon}
\address{Aix Marseille Universit\'e, CNRS, Centrale Marseille, I2M, UMR 7373, 13453 Marseille, FRANCE} \email{anne.pichon@univ-amu.fr}
\begin{abstract} We prove that the outer Lipschitz geometry of the germ
   of a normal complex surface singularity determines a large
  amount of its analytic structure.
  In particular, it follows that any analytic
  family of normal surface singularities with constant  Lipschitz
  geometry is Zariski equisingular. We also prove a strong converse
  for families of normal complex hypersurface singularities in $\C^3$:
  Zariski equisingularity implies  Lipschitz triviality. So for  such
  a family Lipschitz triviality, constant Lipschitz geometry and
  Zariski equisingularity are equivalent to each other.
 \end{abstract}
 \maketitle
 
\section{Introduction}

This paper has two aims. One is to prove the equivalence of
Zariski and bilipschitz equisingularity for families of normal complex
surface singularities. The other, on which the first partly depends,
is to describe which analytic invariants are determined by the
Lipschitz geometry of a normal complex surface singularity.

In \cite{BNP} the richness of the Lipschitz geometry of a normal
surface singularity was demonstrated in a classification in terms of
discrete invariants associated to a refined JSJ decomposition of the
singularity link. That paper addressed the inner metric. The present
paper concerns the outer metric, and we show it is even richer.

\subsection*{Equisingularity}
The question of defining a good notion of equisingularity of a reduced
hypersurface $\mathfrak{X} \subset \C^n$ along a non-singular complex
subspace $Y \subset \mathfrak{X}$ in a neighbourhood of a point $0 \in
\mathfrak{X}$ started in 1965 with two papers of Zariski (\cite{Z1,
  Z2}). This problem has been extensively studied with different
approaches and by many authors such as Zariski himself, Abhyankar,
Brian\c{c}on, Gaffney, Hironaka, L\^e, Lejeune-Jalabert, Lipman,
Mostowski, Parusi\'nski, Pham, Speder, Teissier, Thom, Trotman,
Varchenko, Wahl, Whitney and many others.

One of the central concepts introduced by Zariski is the
algebro-geometric equisingularity, called nowadays Zariski
equisingularity. The idea is that the equisingularity of
$\mathfrak{X}$ along $Y$ is defined inductively on the codimension of
$Y$ in $\mathfrak{X}$ by requiring that the reduced discriminant locus
of a suitably general projection $p \colon \mathfrak{X} \to \C^{n-1}$
be itself equisingular along
$p(Y)$. The induction starts by requiring that in codimension one the
discriminant locus is nonsingular. (This is essentially Speder's
  definition in \cite{Sp}; Zariski later gave a variant in \cite{Z3},
  but  for surface singularities the variants
  are equivalent to each other.)

When $Y$ has codimension one in $\mathfrak{X}$, it is well known that
Zariski equisingularity is equivalent to all main notions of
equisingularity, such as Whitney conditions for the pair
$(\mathfrak{X} \setminus Y,Y)$ and topological triviality.  However
these properties fail to be equivalent in higher codimension: Zariski
equisingularity still implies topological triviality (\cite{Va1,Va2})
and Whitney conditions (\cite{Sp}), but the converse statements are
false (\cite{BS1,BS2, T}) and a global theory of equisingularity is
still far from being established. For good surveys of equisingularity
questions, see \cite{LT, T1}.

The first main result of this paper states the equivalence between
Zariski equisingularity, constancy of Lipschitz geometry and
triviality of Lipschitz geometry in the case $Y$ is the singular locus
of $\mathfrak{X}$ and has codimension $2$ in $\mathfrak{X}$.  We must
say what we mean by ``Lipschitz geometry".  If $(X,0)$ is a germ of a
complex variety, then any embedding $\phi\colon(X,0)\hookrightarrow
(\C^n,0)$ determines two metrics on $(X,0)$: the outer metric
$$d_{out}(x_1,x_2):=|\phi(x_1)-\phi(x_2)|\quad(\text{
    i.e., distance in }\C^n)$$ and the inner metric
$$d_{inn}(x_1,x_2):=\inf\{\operatorname{length}(\phi\circ\gamma):
\gamma\text{ is a rectifyable path in }X\text{ from }x_1\text{ to
}x_2\}\,.$$ The outer metric determines the inner metric, and up to
bilipschitz equivalence both these metrics are independent of the
choice of complex embedding.  We speak of the (inner or outer)
\emph{Lipschitz geometry} of $(X,0)$ when considering these metrics up
to bilipschitz equivalence. If we work up to semi-algebraic or semi-analytic
bilipschitz equivalence, we speak of \emph{semi-algebraic} or
\emph{semi-analytic
  Lipschitz geometry}.

As already mentioned, the present paper is devoted to outer Lipschitz
geometry. We consider here the case of normal complex surface
singularities.

Let $(\mathfrak{X},0)\subset (\C^{n},0)$ be a germ of hypersurface at
the origin of $\C^{n}$ with smooth codimension $2$ singular set $(Y,0)
\subset (\mathfrak{X},0)$. 

The germ $(\mathfrak{X},0)$ \emph{has constant
  Lipschitz geometry} along $Y$ if there exists a smooth
retraction $r \colon (\mathfrak{X},0) \to (Y,0)$
whose fibers are transverse to $Y$ and there is a neighbourhood $U$ of $0$
in $Y$ such that for all $y \in U$ there exists a
bilipschitz homeomorphism $h_y \colon (r^{-1}(y),y) \to (r^{-1}(0)
\cap \mathfrak{X},0)$.

 The germ $(\mathfrak{X},0)$ is
  \emph{Lipschitz trivial} along $Y$ if there exists a germ at 0 of a
bilipschitz homeomorphism $\Phi \colon(\mathfrak
  X,Y)\to (X,0)\times Y$ with $\Phi|_Y=id_Y$, where $(X,0)$ is a
  normal complex surface germ.

  We say $(\mathfrak{X},0)$ has constant \emph{semi-analytic}
    Lipschitz geometry or is \emph{semi-analytic} Lipschitz trivial
  if the maps $r$ and $h_y$, resp.\ $\Phi$ above are semi-analytic.

\begin{theorem} \label{th:zariski vs bilipschitz} The following are
  equivalent:
  \begin{enumerate}
  \item \label{it:zariski} $(\mathfrak{X},0)$ is Zariski equisingular
    along $Y$;
  \item \label{it:constancy} $(\mathfrak{X},0)$ has constant Lipschitz
    geometry along $Y$;
  \item \label{it:constancySA} $(\mathfrak{X},0)$ has constant
    semi-analytic Lipschitz geometry along $Y$;
  \item \label{it:lipschitz} $(\mathfrak{X},0)$ is semi-analytic
    Lipschitz trivial along $Y$ \end{enumerate}
\end{theorem}
The equivalence between \eqref{it:zariski} and \eqref{it:lipschitz}
has been conjectured by Mostowski in a talk for which written notes
are available (\cite{M}).  He also gave there brief hints to prove the
result using his theory of Lipschitz stratifications. Our approach is
different and we construct a decomposition of the pair
$(\mathfrak{X},\mathfrak{X} \setminus Y)$ using the theory of
\carrousels, introduced by L\^e in \cite{Le}, on the family of
discriminant curves. In both approaches, polar curves play a central
role.

The implications \eqref{it:lipschitz}$\Rightarrow$\eqref{it:constancySA}$\Rightarrow$\eqref{it:constancy}
are trivial and the implication
\eqref{it:constancy}$\Rightarrow$\eqref{it:zariski} will be a
consequence 
of one of our other main results, part \eqref{it:discriminant} of Theorem
\ref{th:invariants from geometry} below. The implication 
\eqref{it:zariski}$\Rightarrow$\eqref{it:lipschitz} will be proved in the final sections of the paper.
 
It was known already that the inner Lipschitz geometry is
not sufficient to understand Zariski equisingularity. Indeed, the
family of hypersurfaces $(X_t,0)\subset (\C^3,0)$ with equation $z^3 +
tx^4z+x^6+y^6=0 $ is not Zariski equisingular (see \cite{BS3}; at
$t=0$, the discriminant curve has $6$ branches, while it has $12$
branches when $t \neq 0$). But it follows from \cite{BNP} that it
has constant inner geometry (in fact $(X_t,0)$ is metrically conical
for the inner metric for all $t$).

\subsection*{Invariants from Lipschitz geometry} 
The other main results
of this paper are stated in the next theorem.

\begin{theorem}\label{th:invariants from geometry}
  If $(X,0)$ is a normal complex surface singularity then the outer
  Lipschitz geometry on $X$ determines:
  \begin{enumerate}
  \item\label{it:resolution of pencils} the decorated resolution graph
    of the minimal good resolution of $(X,0)$ which resolves the
    basepoints of a general linear system of hyperplane
    sections\footnote{ This is the minimal good resolution
      which factors through the  blow-up of the maximal ideal.};
  \item\label{it:multiplicity and Zmax} the multiplicity of $(X,0)$
    and the maximal ideal cycle in its resolution;
  \item\label{it:hyperplane section} for a generic hyperplane $H$, the
    outer  Lipschitz geometry of the curve $(X\cap H,0)$;
   \item\label{it:resolution of pencils1} the decorated resolution
     graph of the minimal good resolution of $(X,0)$ which resolves
     the basepoints of the family of polar curves of  plane
     projections\footnote{ This is the minimal good resolution
      which factors through the Nash modification.};
   \item\label{it:discriminant} the topology of the discriminant curve
     of a generic plane projection;
  \item\label{it:curves}  the outer  Lipschitz 
    geometry of the polar curve of a generic plane projection.
      \end{enumerate}
\end{theorem}
By ``decorated resolution graph'' we mean the resolution graph
decorated with arrows corresponding to the components of the strict
transform of the resolved curve.
As a consequence of this theorem,
the outer Lipschitz geometry of a hypersurface
$(X,0)\subset (\C^3,0)$ determines other invariants such as its
extended Milnor number $\mu^*(X,0)$ and the invariants $k(X,0)$ and
$\phi(X,0)$ (number of vanishing double folds resp.\ cusp-folds, see
\cite{BH}).

The analytic invariants determined by the Lipschitz geometry in
Theorem \ref{th:invariants from geometry} are of two natures:
\eqref{it:resolution of pencils}--\eqref{it:hyperplane section} are
related to the general hyperplane sections and the maximal ideal of
$\mathcal O_{X,0}$ while \eqref{it:resolution of
  pencils1}--\eqref{it:curves} are related to polar curves and the
jacobian ideal.  The techniques we use for these two sets of
invariants differ. We prove \eqref{it:resolution of
  pencils}--\eqref{it:hyperplane section} in section \ref{sec:resolve
  hyperplane pencil}, building in part from the classification theorem for inner bilipschitz geometry of \cite{BNP}. The more difficult part is
\eqref{it:resolution of pencils1}--\eqref{it:curves}. These
polar invariants are determined in Sections \ref{sec:detecting1} and
\ref{sec:finding remaining polars} from a geometric decomposition of
$(X,0)$ defined in Sections \ref{sec:carrousel1}  and \ref{sec:decomposition vs resolution} which refines the 
decomposition used in \cite{BNP}.

This paper has four parts, of which Parts 1 and 2 (Sections \ref{sec:2} to \ref{sec:decomposition vs resolution}) introduce
concepts and techniques that are needed later, Part 3  (Sections \ref{sec:resolve hyperplane
  pencil} to \ref{sec:finding remaining polars}) proves Theorem
\ref{th:invariants from geometry} and Part 4  (Sections \ref{sec:notations}
to \ref{sec:proofZtoL}) proves
Theorem \ref{th:zariski vs bilipschitz}.

\smallskip\noindent{\bf Acknowledgments.}  We are especially grateful
to Lev Birbrair for many conversations which have contributed to this
work.  We are also grateful to Terry Gaffney and Bernard Teissier for
useful conversations and to Joseph Lipman for helpful comments.
Neumann was supported by NSF grants DMS-0905770 and
DMS-1206760. Pichon was supported by ANR-12-JS01-0002-01 SUSI and
FP7-Irses 230844 DynEurBraz. We are also grateful for the
hospitality and support of the following institutions: Columbia University
(P), Institut de Math\'ematiques de Marseille and Aix Marseille
Universit\'e (N), Universidad Federal de Ceara (N,P), CIRM recherche
en bin\^ome (N,P) and IAS Princeton (N,P).

\section*{\bf Part 1: Carrousel and geometry of curves}\label{Part 1}

\section{Preliminary: geometric pieces and rates}\label{sec:2}

  In \cite{BNP} we defined some metric spaces (with inner metric) called $A$-, $B$- and
  $D$-pieces, which will be used extensively in the present paper. We
  give here the definition and basic facts about these pieces (see
  \cite[Sections 11 and 13]{BNP} for more details). First examples
  will appear in the next section. The pieces are topologically
  conical, but usually with metrics that make them shrink non-linearly
  towards the cone point.  We will consider these pieces as germs at
  their cone-points, but for the moment, to simplify notation, we
  suppress this.
  
  $D^2$ denotes the standard unit disc in $\C$, $S^1$ is its boundary, and $I$ denotes the interval $[0,1]$. 
  
\begin{definition}[\bf$\boldsymbol{A(q,q')}$-pieces]\label{def:p2}
  Let $q,q'$ be rational numbers such that $1\le q  < q'$. Let $A$ be
  the euclidean annulus $\{(\rho,\psi):1\le \rho\le 2,\, 0\le \psi\le
  2\pi\}$ in polar coordinates and for $0<r\le 1$ let $g^{(r)}_{q,q'}$
  be the metric on $A$:
$$g^{(r)}_{q,q'}:=(r^q-r^{q'})^2d\rho^2+((\rho-1)r^q+(2-\rho)r^{q'})^2d\psi^2\,.
$$ 
So $A$ with this metric is isometric to the euclidean annulus with
inner and outer radii $r^{q'}$ and $r^q$. The metric completion of
$(0,1]\times S^1\times A$ with the metric
$$dr^2+r^2d\theta^2+g^{(r)}_{q,q'}$$ compactifies it by adding a single point at
$r=0$.  We call a metric space which is bilipschitz homeomorphic to
this completion an \emph{$A(q,q')$-piece} or simply \emph{$A$-piece}.
 \end{definition}
 
 \begin{definition}[\bf$\boldsymbol{B(q)}$-pieces]\label{def:p1}  
   Let $F$ be a compact oriented $2$-manifold, $\phi\colon F\to F$ an
   orientation preserving diffeomorphism, and $M_\phi$ the mapping
   torus of $\phi$, defined as:
$$M_\phi:=([0,2\pi]\times F)/((2\pi,x)\sim(0,\phi(x)))\,.$$
Given a rational number $q>1$, we will define a metric space
$B(F,\phi,q)$ which is topologically the cone on the mapping torus
$M_\phi$.
 
For each $0\le \theta\le 2\pi$ choose a Riemannian metric $g_\theta$
on $F$, varying smoothly with $\theta$, such that for some small
$\delta>0$:
$$
g_\theta=
\begin{cases}
g_0&\text{ for } \theta\in[0,\delta]\,,\\
\phi^*g_{0}&\text{ for }\theta\in[2\pi-\delta,2\pi]\,.
\end{cases}
$$
Then for any $r\in(0,1]$ the metric $r^2d\theta^2+r^{2q}g_\theta$ on
$[0,2\pi]\times F$ induces a smooth metric on $M_\phi$. Thus
$$dr^2+r^2d\theta^2+r^{2q}g_\theta$$ defines a smooth metric on
$(0,1]\times M_\phi$. The metric completion of $(0,1]\times M_\phi$
adds a single point at $r=0$.  Denote this completion by $B(F,\phi,q)$. We call a metric space which is bilipschitz homeomorphic to $B(F,\phi,q)$ a  
     \emph{$B(q)$-piece} or simply  \emph{$B$-piece}. 

\red{If $F$ is an annulus we may 
speak of a \emph{$B$-piece of type $A(q,q)$}.} 
A $B(q)$-piece such that $F$ is a disc and
  $q\ge 1$ is called a \emph{$D(q)$-piece} or simply
  \emph{$D$-piece}.
\end{definition}

\begin{definition}[\bf Conical pieces]\label{def:p3}
 Given a compact smooth $3$-manifold $M$,
  choose a Riemannian metric $g$ on $M$ and consider the metric
  $dr^2+r^2g$ on $(0,1]\times M$. The completion of this adds a point
  at $r=0$, giving a \emph{metric cone on $M$}.  We call a metric space which is bilipschitz homeomorphic to a metric cone a \emph{conical piece}.  Notice that $B(1)$- and $A(1,1)$-pieces are
  conical. For short we will call  any conical piece a \emph{$B(1)$-piece}, even if its  $3$-dimensional link is not a mapping-torus. 
  (Conical-pieces were called $CM$-pieces in \cite{BNP}.)
\end{definition}

The diameter of the image in $B(F,\phi,q)$ of a fiber $F_{r,\theta} =
\{r\}\times \{\theta\}\times F$ is $O(r^q)$. Therefore $q$ describes
a rate of shrink of the surfaces $F_{r,\theta}$ in $B(F,\phi,q)$ with
respect to the distance $r$ to the point at $r=0$. Similarly, the
inner and outer boundary components of any $\{r\} \times \{ t \} \times A$ in an $A(q,q')$ have rate of shrink
respectively $q'$ and $q$ with respect to $r$.

\begin{definition}[\bf Rate]
  The rational number $q$ is called the \emph{rate} of $B(q)$ or
  $D(q)$. The rationals $q$ and $q'$ are the two \emph{rates} of $A(q,q')$.
\end{definition}

In \cite[Section 11]{BNP} we explain how to glue two  $A$- or $B$-pieces by an isometry along cone-boundary components having the
same rate. Some of these gluings
give again $A$- or $B$-pieces and then could simplify by some rules as
described in \cite[Section 13]{BNP}.  We recall here this result,
which will be used later in amalgamation processes.
 
 \begin{lemma}\label{le:rules} 
In this lemma $\cong$ means bilipschitz equivalence
and $\cup$ represents gluing along appropriate boundary components
by an isometry. 
\begin{enumerate}
\item\label{rule:D2} $B(D^2,\phi,q)\cong B(D^2,id,q)$; $B(S^1\times
  I,\phi,q)\cong B(S^1\times I,id,q)$.
\item\label{rule:AA} $A(q,q')\cup A(q',q'')\cong A(q,q'')$.
\item\label{rule:FF} If $F$ is the result of gluing a surface
  $F'$ to a disk $D^2$ along boundary components then
  $B(F',\phi|_{F'},q)\cup B(D^2,\phi|_{D^2},q)\cong B(F,\phi,q)$.
\item\label{rule:AD} $A(q,q')\cup B(D^2,id,q')\cong B(D^2,id,q)$.
\item\label{rule:CM} A union of conical 
    pieces glued along boundary components is a conical piece.\qed
\end{enumerate}
\end{lemma}

\section{The plain \carrousel\ for a plane curve germ} \label{sec:carrousel}

A \emph{\carrousel} for a reduced curve germ $(C,0)\subset
(\C^2,0)$ is constructed in two steps. The first one consists in
truncating the Puiseux series expansions of the branches of $C$ at
suitable exponents and then constructing a decomposition of $(\C^2,0)$
into $A$-, $B$- and $D$-pieces with respect to the truncated Puiseux series. We call this
first decomposition the \emph{unamalgamated \carrousel}.  It is the
\carrousel\ used in \cite{BNP}, based on ideas of L\^e in \cite{Le}.

The second step consists in amalgamating certain pieces of the
unamalgamated \carrousel\ using the techniques of Lemma
\ref{le:rules}. 

There are two choices in the construction: where we truncate the
Puiseux series expansions to form a \carrousel, and how
we amalgamate pieces of the \carrousel\ to form the
amalgamated \carrousel. Amalgamated \carrousels\
will be a key tool in the present paper and we will in
fact use several different ones based on different choices of
{truncation} and {amalgamation}.

\medskip In this section, we will construct an amalgamated \carrousel\
which we call the \emph{plain \carrousel} of $C$. We start by
constructing the appropriate unamalgamated \carrousel.

The tangent cone of $C$ at $0$ is a union $\bigcup_{j=1}^mL^{(j)}$ of
lines. For each $j$ we denote the union of components of $C$ which are
tangent to $L^{(j)}$ by $C^{(j)}$.  We can assume our coordinates
$(x,y)$ in $\C^2$ are chosen so that the $y$-axis is transverse to
each $L^{(j)}$.

We choose $\epsilon_0>0$ sufficiently small that the set $\{(x,y):
|x|=\epsilon\}$ is transverse to $C$ for all $\epsilon\le \epsilon_0$.
We define conical sets $V^{(j)}$ of the form
$$V^{(j)}:=\{(x,y):|y-a_1^{(j)}x|\le \eta |x|, |x|\le
\epsilon_0\}\subset \C^2\,,$$ where the equation of the line $L^{(j)}$
is $y=a_1^{(j)}x$ and $\eta>0$ is small enough that the cones are
disjoint except at $0$. If $\epsilon_0$ is small enough
$C^{(j)}\cap \{|x|\le\epsilon_0\}$ will lie completely in
$V^{(j)}$.  

There is then an $R>0$ such that for any $\epsilon\le \epsilon_0$ the
sets $V^{(j)}$ meet the boundary of the ``square ball''
$$B_\epsilon:=\{(x,y)\in \C^2: |x|\le \epsilon, |y|\le R\epsilon\}$$
only in the part $|x|=\epsilon$ of the boundary. We will use these
balls as a system of Milnor balls.

We first define the \carrousel\ inside each $V^{(j)}$ with respect to
the branches of $C^{(j)}$.  It will consist of closures of regions
between successively smaller neighbourhoods of the
successive 
Puiseux approximations of the branches of $C^{(j)}$. As such, it is
finer than the one of \cite{Le}, which only needed the first Puiseux
exponents of the branches of $C^{(j)}$.

We will fix $j$ for the moment and therefore drop the superscripts, so
our tangent line $L$ has equation $y=a_1 x$. The collection of
coefficients and exponents appearing in the following description
depends, of course, on $j=1,\dots,m$.

\subsection*{Truncation} We first truncate the Puiseux
series for each component of $C$ at a point where truncation does not
affect the topology of $C$. Then for each pair $\kappa=(f, p_k)$
consisting of a Puiseux polynomial $f=\sum_{i=1}^{k-1}a_ix^{p_i}$ and
an exponent $p_k$ for which there is a Puiseux series
$y=\sum_{i=1}^{k}a_ix^{p_i}+\dots$ describing some component of $C$,
we consider all components of $C$ which fit this data. If
$a_{k1},\dots,a_{k{m_\kappa}}$ are the coefficients of $x^{p_k}$ which
occur in these Puiseux polynomials we define
\begin{align*}
  B_\kappa:=\Bigl\{(x,y):~&\alpha_\kappa|x^{p_k}|\le
  |y-\sum_{i=1}^{k-1}a_ix^{p_i}|\le
  \beta_\kappa|x^{p_k}|\\
  & |y-(\sum_{i=1}^{k-1}a_ix^{p_i}+a_{kj}x^{p_k})|\ge
  \gamma_\kappa|x^{p_k}|\text{ for }j=1,\dots,{m_\kappa}\Bigr\}\,.
\end{align*}
Here $\alpha_\kappa,\beta_\kappa,\gamma_\kappa$ are chosen so that
$\alpha_\kappa<|a_{kj}|-\gamma_\kappa<|a_{kj}|+\gamma_\kappa<\beta_\kappa$
for each $j=1,\dots,{m_\kappa}$.  If $\epsilon$ is
small enough, the sets $ B_\kappa$ will be disjoint for different
$\kappa$. 

It is easy to see that $ B_\kappa$ is a $B(p_k)$-piece:
  the intersection $B_{\kappa}\cap \{x=t\}$ is a finite collection of
  disks with some smaller disks removed. The diameter of each of them
  is $O(t^{p_k})$. The closure of the complement in $V=V^{(j)}$ of the
  union of the $B_\kappa$'s is a union of $A$- and
  $D$-pieces. Finally, $\overline{B_\epsilon \setminus \bigcup
    V^{(j)}}$ is a $B(1)$-piece.
  We have then decomposed each cone $V^{(j)}$ and the whole of
  $B_{\epsilon}$ as a union of $B$-, $A$- and $D$-pieces.

\begin{definition}[\bf Carrousel sections]\label{def:carrousel}
  A \emph{carrousel section} is the picture of the intersection of a
  carrousel decomposition with a line $x=t$.
\end{definition}

\begin{amalgamation*}\label{amalg:plain}
  For the study of plane curves a much simpler decomposition suffices,
  which we obtain by amalgamating any $D$-piece that does not contain
  part of the curve $C$ \red{and any $A$-piece} with the piece just
  outside it.
\end{amalgamation*}
\begin{definition}
  We call this amalgamated \carrousel\ the \emph{plain \carrousel} of
  the curve $C$. 
  \end{definition}
   
\begin{example}\label{ex:3.2} 
We give examples for the carrousel section and plain carrousel section in Figure \ref{fig:carrousel} below.
On the left the figure shows a carrousel section
    for a curve $C$ having two branches with Puiseux expansions
    respectively $y=ax^{4/3}+bx^{13/6}+\ldots$ and $y=cx^{7/4} +
    \ldots$ truncated after the terms $bx^{13/6}$ and $cx^{7/4}$
    respectively.  The $D$-pieces are gray. Note that the intersection
    of a piece of the decomposition of $V$ with the disk $V\cap
    \{x=\epsilon\}$ will usually have several components. Note also
    that the rates in $A$- and $D$-pieces are determined by the rates
    in the neighbouring $B$-pieces.

\red{On the right is the corresponding plain carrousel section.}

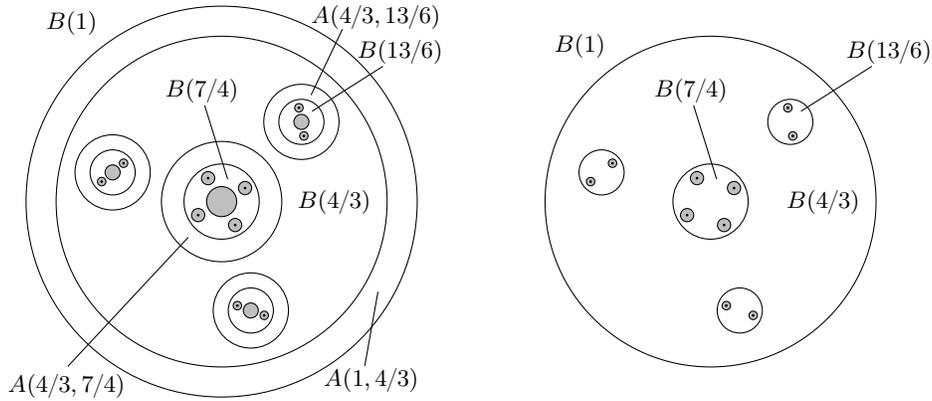
\begin{figure}[ht]
  \centering
\begin{tikzpicture}

\draw  (0,0)circle(2.6cm);
\draw  (0,0)circle(2.2cm);
\draw  (0,0)circle(.8cm);   
\draw  (0,0)circle(0.5cm);

\draw[fill=lightgray](0,0)circle(0.2cm);

\draw[fill=lightgray] (0:0)+(30:0.36)circle(2.5pt);
\draw[fill] (0:0)+(30:0.36)circle(0.3pt);
\draw[fill=lightgray] (0:0)+(120:0.36)circle(2.5pt);
\draw[fill] (0:0)+(120:0.36)circle(0.3pt);
\draw[fill=lightgray] (0:0)+(210:0.36)circle(2.5pt);
\draw[fill] (0:0)+(210:0.36)circle(0.3pt);
\draw[fill=lightgray] (0:0)+(300:0.36)circle(2.5pt);
\draw[fill] (0:0)+(300:0.36)circle(0.3pt);

\draw  (45:1.5)circle(0.5);
\draw  (45:1.5)circle(0.3);
\draw[fill=lightgray] (45:1.5)+(100:0.19)circle(1.6pt);
\draw[fill=lightgray] (45:1.5)+(280:0.19)circle(1.6pt);
\draw[fill] (45:1.5)+(100:0.19)circle(0.3pt);
\draw[fill] (45:1.5)+(280:0.19)circle(0.3pt);
\draw[fill=lightgray]   (45:1.5)circle(0.1);
\draw  (165:1.5)circle(0.5);
\draw  (165:1.5)circle(0.3);
\draw[fill=lightgray] (165:1.5)circle(0.1);
\draw[fill=lightgray] (165:1.5)+(40:0.19)circle(1.6pt);
\draw[fill] (165:1.5)+(40:0.19)circle(0.3pt);
\draw[fill=lightgray] (165:1.5)+(220:0.19)circle(1.6pt);
\draw[fill] (165:1.5)+(220:0.19)circle(0.3pt);
          
\draw  (285:1.5)circle(0.5);
\draw  (285:1.5)circle(0.3);
\draw[fill=lightgray]   (285:1.5)circle(0.1);
\draw[fill=lightgray] (285:1.5)+(-20:0.19)circle(1.6pt);
\draw[fill] (285:1.5)+(-20:0.19)circle(0.3pt);
\draw[fill=lightgray] (285:1.5)+(160:0.19)circle(1.6pt);
\draw[fill] (285:1.5)+(160:0.19)circle(0.3pt);

\draw[thin](100:1.3) --(75:0.31);
\draw[thin](55:2.8) --(50:1.85);
\draw[thin](40:2.8) --(45:1.7);
\draw[thin](-30:2.4)--(-49:2.9);
\draw[thin](-133:.65)--(-133:3);

\node(a)at(130:3.1){\small$B{(1)}$};
\node(a)at(230:3.2){\small$A(4/3,7/4)$};
\node(a)at(100:1.5){\small$B{(7/4)}$};
\node(a)at(0:1.5){\small $B{(4/3)}$};
\node(a)at(40:3.1){\small$B{(13/6)}$};
\node(a)at(50:3.2){\small$A(4/3,13/6)$};
\node(a)at(-50:3.1){\small$A(1,4/3)$};

\begin{scope}[xshift=6.5cm]
\draw  (0,0)circle(2.2cm);
\draw  (0,0)circle(0.5cm);


\draw[fill=lightgray] (0:0)+(30:0.36)circle(2.5pt);
\draw[fill] (0:0)+(30:0.36)circle(0.3pt);
\draw[fill=lightgray] (0:0)+(120:0.36)circle(2.5pt);
\draw[fill] (0:0)+(120:0.36)circle(0.3pt);
\draw[fill=lightgray] (0:0)+(210:0.36)circle(2.5pt);
\draw[fill] (0:0)+(210:0.36)circle(0.3pt);
\draw[fill=lightgray] (0:0)+(300:0.36)circle(2.5pt);
\draw[fill] (0:0)+(300:0.36)circle(0.3pt);

\draw  (45:1.5)circle(0.3);
\draw[fill=lightgray] (45:1.5)+(100:0.19)circle(1.6pt);
\draw[fill=lightgray] (45:1.5)+(280:0.19)circle(1.6pt);
\draw[fill] (45:1.5)+(100:0.19)circle(0.3pt);
\draw[fill] (45:1.5)+(280:0.19)circle(0.3pt);

\draw  (165:1.5)circle(0.3);
\draw[fill=lightgray] (165:1.5)+(40:0.19)circle(1.6pt);
\draw[fill] (165:1.5)+(40:0.19)circle(0.3pt);
\draw[fill=lightgray] (165:1.5)+(220:0.19)circle(1.6pt);
\draw[fill] (165:1.5)+(220:0.19)circle(0.3pt);
          
\draw  (285:1.5)circle(0.3);
\draw[fill=lightgray] (285:1.5)+(-20:0.19)circle(1.6pt);
\draw[fill] (285:1.5)+(-20:0.19)circle(0.3pt);
\draw[fill=lightgray] (285:1.5)+(160:0.19)circle(1.6pt);
\draw[fill] (285:1.5)+(160:0.19)circle(0.3pt);

\draw[thin](100:1.3) --(75:0.31);
\draw[thin](40:2.8) --(45:1.7);

\node(a)at(130:2.7){\small$B(1)$};
\node(a)at(100:1.5){\small$B{(7/4)}$};
\node(a)at(0:1.5){\small $B{(4/3)}$};
\node(a)at(40:3.1){\small$B{(13/6)}$};
\end{scope}
  \end{tikzpicture} 
  \caption{Unamalgamated and plain carrousel sections for
    $C=\{y=ax^{4/3}+bx^{13/6}+\ldots\}\cup\{y=cx^{7/4} +
    \ldots\}$.}
  \label{fig:carrousel}
\end{figure}
\end{example}

The combinatorics of the plain carrousel section can be encoded by a
rooted tree, with vertices corresponding to pieces, edges
corresponding to pieces which intersect along a circle, and with
rational weights (the rates of the pieces) associated to
the nodes of the tree.  It is easy to recover a picture of
  the carrousel section, and hence the embedded topology of the plane
curve, from this weighted tree. For a careful description of how
to do this in terms of either the Eggers tree or the Eisenbud-Neumann
splice diagram of the curve see 
\cite{NP}. 
\begin{proposition}\label{prop:curve from carrousel}
  The combinatorics of the plain carrousel section for a plane curve germ
  $C$ determines the embedded topology of $C$.\qed
\end{proposition}

\section{Lipschitz geometry and topology of a plane curve}\label{sec:curve}

In \cite{NP} we proved the following strong version of a result of
Pham and Teissier \cite{PT} and Fernandes \cite{F} about plane curve
germs. 
\begin{proposition} \label{prop:plane curve} The outer Lipschitz
  geometry of a plane curve germ
  $(C,0)\subset (\C^2,0)$ determines its embedded topology.

  More generally, if $(C,0)\subset (\C^n,0)$ is a curve germ and
  $\ell\colon\C^n\to \C^2$ is a generic plane projection then the
  outer Lipschitz geometry of $(C,0)$ determines the embedded topology
  of the plane projection $(\ell(C),0)\subset (\C^2,0)$.
  \end{proposition}

  What was new in this proposition is that the germ is considered just
  as a metric space germ up to bilipschitz equivalence, without the
  analytic restrictions of \cite{PT} and \cite{F}. The converse
  result, that the embedded topology of a plane curve determines the
  outer Lipschitz geometry, is easier and is proved in \cite{PT}.

 In this section we briefly recall the proof of Proposition
  \ref{prop:plane curve}, since we use the last part of it in Section
  \ref{sec:resolve hyperplane pencil} (proof of Theorem \ref{th:invariants from
    geometry}\eqref{it:hyperplane section}) and we also mildly adapt its proof (as given in \cite{NP}) to illustrate a technique we will use 
  in Section \ref{sec:finding remaining polars}.

  \begin{proof}
  The case $n>2$ of Proposition \ref{prop:plane curve} follows
  immediately from the case $n=2$ since Pham and
  Teissier prove in \cite{PT} that for a generic plane projection
  $\ell$ the restriction $\ell|_C\colon (C,0)\to(\ell(C),0)$ is
  bilipschitz for the outer geometry. So we assume from now on
  that $n=2$, so $(C,0)\subset (\C^2,0)$ is a plane curve.

  We first describe how to recover the combinatorics of the
  plain carrousel section of $C$ using the analytic structure and the
  outer geometry. 

We assume, as in the previous section, that the tangent cone of $C$
at $0$ is a union of lines 
transverse to the $y$-axis. We  use again the family of Milnor balls 
$B_\epsilon$, $\epsilon\le
\epsilon_0$, of the previous
section.  We put $S_\epsilon=\partial B_\epsilon$. 
Let $\mu$ be the multiplicity of $C$. The lines $x=t$ for $t\in
(0,\epsilon_0]$ intersect $C$ in a finite set of points
$p_1(t),\dots,p_{\mu}(t)$ which depends continuously on $t$. For each
$0<j<k\le \mu$ the distance $d(p_j(t),p_k(t))$ has the form
$O(t^{q_{jk}})$, where $q_{jk}$ is either an essential Puiseux
exponent for a branch of the plane curve $C$ or a coincidence exponent
between two branches of $C$.

\begin{lemma} \label{le:curve geometry}The map $ \{(j,k)~|~1\le j<k\le
  \mu\}\mapsto q_{jk}$ determines the embedded topology of $C$.
\end{lemma}

\begin{proof} By Proposition \ref{prop:curve from carrousel} it
  suffices to prove that the map $\{(j,k)~|~1\le j<k\le \mu\}\mapsto
  q_{jk}$ determines the combinatorics of the plain carrousel section
  of the curve $C$.

  Let $q_1>q_2>\ldots > q_s$ be the images of the map $(j,k)
  \mapsto q_{jk}$.
  The proof consists in reconstructing a topological version of the
  carrousel section of $C$ from the innermost pieces to the
  outermost ones by an inductive process starting with $q_1$ and
  ending with $q_s$. 

  We start with $\mu$ discs $D^{(0)}_1,\ldots,D^{(0)}_{\mu}$, which
  will be the innermost pieces of the \carrousel.  We consider the
  graph $G^{(1)}$ whose vertices are in bijection with these $\mu$
  disks and with an edge between vertices $v_j$ and $v_k$ if and only
  if $q_{jk}=q_1$. Let $G_1^{(1)},\ldots,G_{\nu_1}^{(1)}$ be the
  connected components of $G^{(1)}$ and denote by $\alpha_{m}^{(1)}$
  the number of vertices of $G_m^{(1)}$.  For each $G_m^{(1)}$ with
  $\alpha_m^{(1)}>1$ we consider a disc $B^{(1)}_m$ with
  $\alpha_m^{(1)}$ holes, and we glue the discs $D_j^{(0)}$,
  $v_j\in\operatorname{vert}G_m^{(1)}$, into the inner boundary
  components of $B^{(1)}_m$. We call the resulting disc
  $D^{(1)}_m$. For a $G_m^{(1)}$ with just one vertex, $v_{j_m}$ say,
  we rename $D_{j_m}^{(0)}$ as $D_m^{(1)}$.

 The numbers $q_{jk}$ have the property that $q_{jl} \ge min
 (q_{jk},q_{kl})$ for any triple $j,k,l$. So for each distinct $m,n$
 the number $q_{j_m k_n}$ does not depend on the choice of vertices
 $v_{j_m}$ in $G_m^{(1)}$ and $v_{k_n}$ in $G_n^{(1)}$.

 We iterate the above process as follows: we consider the graph
 $G^{(2)}$ whose vertices are in bijection with the connected
 components $G_1^{(1)},\ldots,G_{\nu_1}^{(1)}$ and with an edge
 between the vertices $(G_m^{(1)})$ and $(G_n^{(1)})$ if and only if
 $q_{j_m k_n}$ equals $q_2$ (with vertices $v_{j_m}$ and $v_{j_n}$ in
 $G_m^{(1)}$ and $G_n^{(1)}$ respectively). Let
 $G_1^{(2)},\ldots,G_{\nu_2}^{(2)}$ be the connected components of
 $G_2$. For each $G_m^{(2)}$ let $\alpha_m^{(2)}$ be the number of its
 vertices. If $\alpha_m^{(2)}>1$ we take a disc $B^{(2)}_m$ with
 $\alpha_m^{(2)}$ holes and glue the corresponding disks $D^{(1)}_l$
 into these holes. We call the resulting piece $D^{(2)}_m$. If
 $\alpha_m^{(2)}=1$ we rename the corresponding $D^{(1)}_m$ disk to
$D^{(2)}_m$.

 By construction, repeating this process for $s$ steps gives a
 topological version of the carrousel section of the curve $C$, and
 hence, by Proposition \ref{prop:curve from carrousel}, its embedded topology.
\end{proof}

To complete the proof of Proposition \ref{prop:plane curve} we must show that we can find
the numbers $q_{jk}$ without using the complex structure, and after
a bilipschitz change to the outer metric.

The tangent cone of $C$ at $0$ is a union of lines $L^{(i)}$,
$i=1,\dots,m$, transverse to the $y$-axis.  Denote by $C^{(i)}$ the
part of $C$ tangent to the line $L^{(i)}$.  It suffices to discover
the $q_{jk}$'s belonging to each $C^{(i)}$ independently, since the
$C^{(i)}$'s are distinguished by the fact that the distance between
any two of them outside a ball of radius $\epsilon$ around $0$ is
$O(\epsilon)$, even after bilipschitz change to the
metric. We will therefore assume from now on that the tangent cone of $C$
is a single complex line.

Our points $p_1(t),\dots,p_\mu(t)$ that we used to find the numbers
$q_{jk}$ were obtained by intersecting $C$ with the line $x=t$. The arc
$p_1(t)$, $t\in [0,\epsilon_0]$ satisfies $d(0,p_1(t))=O(t)$. Moreover,
for any small $\delta>0$ we can choose $\epsilon_0$ sufficiently small
that the other points $p_2(t),\dots,p_\mu(t)$ are always in the
transverse disk of radius $\delta t$ centered at $p_1(t)$ in the plane
$x=t$.

Instead of a transverse disk of radius $\delta t$, we now use a ball
$B(p_1(t),\delta t)$ of radius $\delta t$ centered at $p_1(t)$. This
$B(p_1(t),\delta t)$ intersects $C$ in $\mu$ disks
$D_1(t),\dots,D_\mu(t)$, and we have $d(D_j(t),D_k(t))=O(t^{q_{jk}})$,
so we still recover the numbers $q_{jk}$.  

We now replace the arc $p_1(t)$ by any continuous arc $p'_1(t)$ on $C$
with the property that $d(0,p'_1(t))=O(t)$. If $\delta $ is
sufficiently small, the intersection  $B_C(p'_1(t),\delta t):=C\cap B(p'_1(t),\delta t)$ still consists of $\mu$
disks $D'_1(t),\dots,D'_\mu(t)$ with
$d(D'_j(t),D'_k(t))=O(t^{q_{jk}})$.  So we have gotten rid of the
dependence on analytic structure in discovering the topology. But we
must consider what a $K$-bilipschitz change to the metric does. 

Such a
change may make the components of $B_C(p'_1(t),\delta t)$ disintegrate
into many pieces, so we can no longer simply use distance between
pieces. To resolve this, we consider both $B_C'(p'_1(t),\delta t)$ and
$B_C'(p'_1(t),\frac \delta {K^4}t)$ where $B'$ means we are using the
modified metric. Then only $\mu$ components of $B_C'(p_1(t),\delta t)$
will intersect $B_C'(p_1(t),\frac \delta {K^4}t)$.  Naming these
components $D'_1(t),\dots,D'_\mu(t)$ again, we still have
$d(D'_j(t),D'_k(t))=O(t^{q_{jk}})$ so the $q_{jk}$ are determined as
before.
\end{proof}

We end this section with a remark which introduces a key concept for
the proof of Proposition \ref{prop:find carrousel}.
\begin{remark}\label{rk:beyond}
  Assume first that $C$ is irreducible.  Fix $q \in \Bbb Q$ and
  replace in the arguments above the balls $B_C(p_1(t), \delta t)$ of
  radius $\delta t$ by balls of radius $\delta t^q$. If $\delta$
  is big enough then $B_C(p_1(t), \delta t^q)$ consists of $\eta$ discs
  $D''_1(t),\ldots,D''_{\eta}(t)$ for some $\eta \leq \mu$ and the
  rates $q''_{j,k}$ given by $d(D''_j(t),D''_k(t)) = O(t^{q''_{j,k}})$
  coincide with the above rates $q_{j,k}$ such that $q_{j,k}\geq q$.
  These rates determine the carrousel sections of $C$ inside pieces of
  the carrousel with rates $\ge q$.

\begin{definition} \label{def:beyond} We say that the rates
    $q''_{j,k}$ with $1\le j<k\le \eta$ determine the outer Lipschitz geometry
    (resp.\ the carrousel section) of $C$ \emph{beyond rate $q$.}
\end{definition} 
If $C$ is not irreducible, take $p_1$ inside a component $C_0$ of
$C$. Then the rates $q''_{j,k}$ recover the outer Lipschitz geometry beyond $q$
of the union $C'$ of components of $C$ having exponent of coincidence
with $C_0$ greater than or equal to $q$. If $C$ contains a component
which is not in $C'$, we iterate the process by choosing an arc in
it. After iterating this process to catch all the components of $C$,
we say we have determined the outer Lipschitz geometry of the whole $C$ (resp.\
its carrousel section) \emph{beyond rate $q$}.

If $C$ is the generic projection of a curve $C' \subset \C^n$, we
speak of \emph{the outer Lipschitz geometry of $C'$ beyond rate $q$}. 
\end{remark}

\section*{\bf Part 2: Geometric decompositions of a normal complex
  surface singularity}\label{Part 2}

\section{Introduction to geometric decomposition}
Birbrair and the authors proved in \cite{BNP}  the existence and
unicity of a decomposition $(X,0) = (Y,0) \cup (Z,0)$ of a normal
complex surface singularity $(X,0)$ into a thick part $(Y,0)$ and a thin part
$(Z,0)$. The thick part is essentially the metrically conical part of
$(X,0)$ with respect to the inner metric while the thin part shrinks
faster than linearly in size as it approaches the origin. The
thick-thin decomposition was then refined further by dividing
the thin part into $A$- and $B$-pieces, giving a
classification of the inner Lipschitz geometry in terms of discrete data
recording rates and directions of shrink of the refined pieces (see
\cite[Theorem 1.9]{BNP}).

This ``classifying decomposition'' decomposes the 3-manifold link
$X^{(\epsilon)}:=X\cap S_\epsilon$ of $X$ into Seifert fibered pieces
glued along torus boundary components. In general this decomposition
is a refinement of the JSJ decomposition of the $3$-manifold
$X^{(\epsilon)}$ (minimal decomposition into Seifert fibered pieces).

In this Part 2 of the paper we refine the decomposition further, in a
way that implicitly sees the influence of the outer Lipschitz geometry, by
taking the position of polar curves of generic plane projections into
account. We call this decomposition the \emph{geometric decomposition} (see Definition \ref{def:geometric decomposition})
and write it as 
$$(X,0) = \bigcup_{i=1}^\nu (X_{q_i},0)\cup\bigcup_{i>j} (A_{q_i,q_j},0)\,, \quad
q_1>\dots>q_\nu= 1\,.$$ Each $X_{q_i}$ is a union of pieces of type
$B(q_i)$ and each $A_{q_i,q_j}$ is a (possibly empty) union of $A(q_i,q_j)$-pieces glued to $X_{q_i}$ and  $X_{q_j}$ along boundary components. This refinement may also
decompose the thick part by excising from it some $D$-pieces (which
may themselves be further decomposed) which are ``thin for the outer
metric''.

The importance of the geometric decomposition is that it can be
recovered using the  outer Lipschitz geometry of $(X,0)$. We will
show this in Part 3 of the paper.
We will need two different constructions of the geometric
decomposition, one in terms of a carrousel, and one in terms of
resolution. The first is done in Sections
\ref{sec:polarwedges} to \ref{sec:carrousel1} and the second in Section
\ref{sec:decomposition vs resolution}.

It is worth remarking that the geometric decomposition has a
topological description in terms of a relative JSJ decomposition of
the link $X^{(\epsilon)}$. Consider the link $K\subset X^{(\epsilon)}$
consisting of the intersection with $X^{(\epsilon)}$ of the polar
curves of a generic pair of plane projections $X\to\C^2$ union a
generic pair of hyperplane sections of $C$. The decomposition of
$X^{(\epsilon)}$ as the union of the components of the
$X^{(\epsilon)}_{q_i}$'s and the intermediate $A(q_i,q_j)$-pieces is
topologically the relative JSJ decomposition for the pair
$(X^{(\epsilon)},K)$, i.e., the minimal decomposition of
$X^{(\epsilon)}$ into Seifert fibered pieces separated by annular
pieces such that the components of $K$ are Seifert fibers in the
decomposition.  
\red{Actually, the geometric decomposition that we will use is
  slightly stronger, in that it may refine an annular piece by cutting
  it into a sequence of annular pieces. We come back to this in
  Definition \ref{def:gdp}.}

\section{Polar wedges}\label{sec:polarwedges}

We denote by $\grassman(k,n)$ the grassmannian of $k$-dimensional
subspaces of $\C^n$.

\begin{definition}[\bf Generic linear projection]\label{def:generic linear
    projection} 
  Let $(X,0)\subset (\C^n,0)$ be a normal surface germ.  For
  $\cal D\in \grassman(n-2,n)$ let $\ell_{\cal D} \colon \C^n \to
  \C^2$ be the linear projection $\C^n \to \C^2$ with kernel $\cal D$.
  Let $\Pi_{\cal D}\subset X$ be the polar of this projection,
  i.e., the closure in $(X,0)$ of the singular locus of the
  restriction of $\ell_{\cal D}$
to $X \setminus    \{0\}$, 
and let
  $\Delta_{\cal D}=\ell_{\cal D}(\Pi_{\cal D})$ be the discriminant
  curve. There exists an open dense subset $\Omega \subset
  \grassman(n-2,n)$ such that $\{(\Pi_{\cal D},0): \cal D \in
  \Omega\}$ forms an equisingular family of curve germs in terms of
  strong simultaneous resolution and such that the discriminant curves
  $\Delta_{\cal D}$ are reduced and no tangent line to $\Pi_{\cal D}$
  at $0$ is contained in $\cal D$ (\cite[(2.2.2)]{LT0} and
  \cite[V. (1.2.2)]{T3}). The projection $\ell_{\cal D} \colon \C^n
  \to \C^2$ is \emph{generic} for $(X,0)$ if $\cal D \in \Omega$.
\end{definition}

The condition $\Delta_{\cal D}$ reduced means that any
$p\in\Delta_{\cal D}\setminus\{0\}$ has a neighbourhood $U$ in $\C^2$
such that one component of $(\ell_{\cal D}|_X)^{-1}(U)$ maps by a
two-fold branched cover to $U$ and the other components map
bijectively.

\medskip Let $B_\epsilon$ be a Milnor ball for $(X,0)$ (in Section
\ref{sec:carrousel1} we will specify a family of Milnor balls). Fix a
$\cal D \in \Omega$ and a component $\Pi_0$ of the polar curve of
$\ell = \ell_{\cal D}$.  We now recall the main result of
\cite[Section 3]{BNP}, which defines a suitable region $A_0$
containing $\Pi_0$ in $X\cap B_\epsilon$, outside of which $\ell$ is a
local bilipschitz homeomorphism.
  
Let us consider the branch $\Delta_0 = \ell(\Pi_0)$ of the
discriminant curve of $\ell$. Let $V$ be a small
  neighbourhood of $\cal D$ in $\Omega$.  For each $\cal D'$ in $V$ let
  $\Pi_{\cal D',0}$ be the component of $\Pi_{\cal D'}$ close to
  $\Pi_0$. Then  the curve $\ell(\Pi_{\cal
    D',0})$ has Puiseux expansion
$$y=\sum_{j\ge 1} a_j (\cal D') x^{{p_j}}\in \C\{x^{\frac1N}\},\quad \text{ with
  }p_j\in \Q,\quad
1\le p_1<p_2<\cdots\,,$$ 
where $a_j (\cal D') \in \C$.  Here $N=\lcm_{j\ge 1}\denom(p_j)$, where ``denom'' means denominator. (In fact, if $V$ is sufficiently small, the family of curves
$\ell(\Pi_{\cal D',0})$,  for $\cal D' \in V$ is equisingular, see \cite[p. 462]{T3}.)

\begin{definition}[\bf Wedges, contact and polar rates] \label{def:wedges} Let $s$ be the first exponent $p_j$ for
  which the coefficient $a_j(\cal D')$ is non-constant, i.e.,
  it depends on $\cal D'\in V$.
For $\alpha>0$, define
 $$B_0 := \bigl\{(x,y):\Bigl|y-\sum_{j\ge1}a_jx^{{p_j}}
 \Bigr|\le\alpha|x|^{s}\bigr\}\,.$$ We call $B_0$ a
 \emph{$\Delta$-wedge} (about $\Delta_0$). Let $A_0$ be the germ of
 the closure of the connected component of $\ell^{-1}(B_0) \setminus
 \{0\}$ which contains $\Pi_0$.  We call $A_0$ a \emph{polar wedge}
 (about $\Pi_0$).  We call $s$ the \emph{polar rate}  of
   $A_0$ (resp.\ $B_0$).
As described in \cite{BNP}, instead of $B_0$ one can use $B'_0 := \bigl\{(x,y):\Bigl|y-\sum_{j\ge1,p_j\le s}a_jx^{{p_j}}
 \Bigr|\le\alpha|x|^{s}\bigr\}$, since truncating higher-order terms does not change the bilipschitz geometry. 

\red{We say two irreducible complex curves $C_1$ and $C_2$ through the origin in $\C^n$ \emph{have contact $q$} if $d(C_1\cap S_\epsilon, C_2\cap S_\epsilon)=O(\epsilon^q)$.  The contact between $\Pi_{\cal D',0}$ and $\Pi_0$ is the polar rate of $A_0$ and was called ``contact exponent'' in \cite{BNP}.}
\end{definition} 
Clearly $B_0$ is a $D(s)$-piece. Moreover:
\begin{proposition}[Proposition 3.4(2) of \cite{BNP}] \label{prop:polar wedge} A polar wedge $A_0$ with rate $s$ is a $D(s)$-piece.\qed
\end{proposition}

The proof in \cite{BNP} shows that $A_0$ can be approximated to high
order by a set $A_0'=\bigcup_{|t|\le \beta}\Pi_{\cal D_t,0}$ for some
$\beta>0$, where $t\in \C$, $|t|\le \beta$, parametrizes a piece of a
suitable line through $\cal D$ in $\grassman(n-2,n)$. We
will need some of the details later, so we describe this here.

We can choose coordinates $(z_1,\dots,z_n)$ in $\C^n$ and $(x,y)$ in
$\C^2$ so that $\ell$ is the projection $(x,y)=(z_1,z_2)$ and that the
family of plane projections $\ell_{\cal D_t}$ used in \cite{BNP} is
$\ell_{\cal D_t}\colon (z_1,\dots,z_n)\to (z_1, z_2-tz_3)$.

In \cite{BNP} it is shown that $A'_0$ can be parametrized as a union
of the curves $\Pi_{\cal D_t,0}$ in terms of parameters $(u,v)$ as follows:
\begin{align*}
  z_1&=u^N\\
z_2&=u^Nf_{2,0}(u)+v^2u^{Ns}h_2(u,v)\\
z_j &=
 u^Nf_{j,0}(u)+vu^{Ns}h_j(u,v)\,,\quad j=3,\dots,n\\
\Pi_{\cal D_t,0}&=\{(z_1,\dots,z_n):v=t\}\,,
\end{align*}
where $h_2(u,v)$ is a unit. 

We then have $A_0=\{(z_1(u,v),\dots,z_n(u,v)):|v^2h_2(u,v)|\le \alpha\}$,
which agrees up to order $>s$ with $A'_0=\{(z_1(u,v),\dots,z_n(u,v)):|v|\le \beta\}$ with
$\beta=\sqrt{\alpha/{|h_2(0,0)|}}$.

In \cite{BNP} it was also pointed out that at 
least one $h_j(u,v)$ with $j\ge 3$ is a unit, by a modification of
the argument of Teissier \cite[p.~464, lines 7--11]{T3}.  We
will show this using our explicit choice of coordinates above.
\begin{lemma}\label{le:h3}
  $h_3(u,v)$ is a unit.  More specifically, writing
  $h_2(u,v)=\sum_{i\ge0}v^{i}f_{i+2}(u)$, we have
  $h_3(u,v)=\sum_{i\ge0}v^i\,\frac{i+2}{i+1}f_{i+2}(u)$.
\end{lemma}
\begin{proof}
  Make the change of coordinates $z'_2:=z_2-tz_3$, so $\ell_{\cal
    D_t}$ is the projection to the $(z_1,z_2')$-plane.  Since
  $\Pi_{\cal D_t,0}$ is given by $v=t$ in our $(u,v)$ coordinates, we
  change coordinates to $(u,w)$ with $w:=v-t$, so that $\Pi_{\cal
    D_t,0}$ has equation $w=0$.  We then know from above that $z'_2$
  must have the form
$$z'_2=u^Nf'_{2,0}(u)+w^2u^{Ns}h'_2(u,w)$$
with $h'_2(u,w)$ a unit.
On the other hand, we rewrite the expressions for $z_2,z_3$ explicitly as
\begin{align*}
z_2&=u^Nf_{2,0}(u)+u^{Ns}\sum_{i\ge 2}v^if_i(u)\\
z_3&= u^Nf_{3,0}(u)+u^{Ns}\sum_{i\ge 1}v^ig_i(u)\,.
\end{align*}
Then direct calculation from
\begin{align*}
z_2&=u^Nf_{2,0}(u)+u^{Ns}\sum_{i\ge 2}(w+t)^if_i(u)\\
z_3&= u^Nf_{3,0}(u)+u^{Ns}\sum_{i\ge 1}(w+t)^ig_i(u)
\end{align*}
gives:
\begin{align*}
  z'_2=&z_2-tz_3\\= &u^N(f_{2,0}(u)-tf_{3,0}(u))+u^{Ns}\big[\sum_{i\ge 1}t^{i+1}(f_{i+1}(u)-g_i(u))  \\&+w\sum_{i\ge1}t^i((i+1)f_{i+1}(u)-ig_i(u)) \\
&+w^2(f_2(u)+t(3f_3(u)-g_2(u))+t^2(6f_4(u)-3g_3(u)))+\dots\big]\,.
\end{align*}
The part of this expression of degree
$1$ in $w$ must be zero, so
$g_i(u)=\frac{i+1}if_{i+1}(u)$ for each $i$. Thus
\begin{align*}
  z_3&=u^Nf_{3,0}(u)+u^{Ns}\sum_{i\ge 1}v^i\,\text{\Small$\frac{i+1}i$}f_{i+1}(u)\\
&=u^Nf_{3,0}(u)+vu^{Ns}\sum_{i\ge 0}v^i\,\text{\Small$\frac{i+2}{i+1}$}f_{i+2}(u)\,,
\end{align*}
completing the proof.
\end{proof}

For $x \in X$, we define the \emph{local bilipschitz constant} $K(x)$
of the projection $\ell_{\cal D} \colon X \rightarrow \C^2$ as
follows: $K(x)$ is infinite if $x$ belongs to the polar curve
$\Pi_{\cal D}$, and at a point $x \in X \setminus \Pi_{\cal D}$ it is
the reciprocal of the shortest length among images of unit vectors in
$T_{x} X$ under the projection $d\ell_{\cal D} \colon T_{x} X\to
\C^2$.

For $K_0 \geq 1$, set ${\mathcal B}_{K_0}:=\bigl\{ p\in X\cap
(B_\epsilon\setminus \{0\}):K(p)\ge K_0\bigr\}\,,$ and let ${\mathcal
  B}_{K_0}(\Pi_0)$ denote the closure of the connected component of $
{\mathcal B}_{K_0}\setminus \{0\}$ which contains $\Pi_0
\setminus\{0\}$.  As a consequence of \cite[3.3, 3.4(1)]{BNP}, we obtain
that for $K_0>1$ sufficiently large, the set ${{\mathcal B}}_{K_0}$
can be approximated by a polar wedge about $\Pi_{0}$. Precisely:
  
\begin{proposition}\label{prop:thinzones} 
There exist $K_0,K_1\in\R$ with $1<K_1<K_0$ such that
\[
\pushQED{\qed}
{\mathcal B}_{K_0}(\Pi_0) \subset A_0\cap B_\epsilon \subset
  {\mathcal B}_{K_1}(\Pi_0)
\qedhere\popQED
\]
\end{proposition}

Let us consider again a component $\Pi_0$ of $\Pi$ and a polar wedge
$A_0$ around it, the corresponding component $\Delta_0$ of $\Delta$
and $\Delta$-wedge $B_0$, and let $s$ be their polar rate.  Let
$\gamma$ be an irreducible curve having contact $r\ge s$ with
$\Delta_0$. The following Lemma establishes a relation between $r$ and
the geometry of the components of the lifting $\ell^{-1}(\gamma)$ in
the polar wedge $A_0$ about $\Pi_0$.  This will be a key argument in
Part 4 of the paper (see proof of Lemma \ref{prop:polar
  constancy}). It will also be used in section \ref{sec:explicit
  computation} to explicitly compute some polar rates.

We choose coordinates in $\Bbb C^n$ as before, with $\ell=(z_1,z_2)$
our generic projection. 
Recall the Puiseux expansion of a component
$\Delta_0$ of the discriminant:
$$z_2= \sum_{i\geq N} a_i z_1^{i/N} \in \Bbb C\{z_1^{1/N}\}\,.$$
\begin{lemma} \label{lemma:constant} Let $(\gamma,0)$ be an
  irreducible germ of curve in $(\C^2,0)$ with Puiseux expansion:
$$  z_2=\sum_{i\geq N} a_i z_1^{i/N} + \lambda z_1^{r}\,,$$
with $r \geq s$, $rN \in \Bbb Z$ and $\lambda \in \C$, $0<|\lambda|
<1$. Let $L'_{\gamma}$ be the intersection of the lifting
$L_{\gamma}=\ell^{-1}(\gamma)$ with a polar wedge $A_0$ about $\Pi_0$
and let $\ell'$ another generic plane projection for $X$ which is also
generic for the curve $\ell^{-1}(\gamma)$. Then the rational number
$q(r) := \frac{s+r}{2}$ is determined by the topological type of the
curve $\ell'(L'_\gamma)$.
\end{lemma}

\begin{proof} We first consider the case $r>s$. Using again the
  parametrization of $A'_0$ given after the statement of Proposition
  \ref{prop:polar wedge}, the curve $L'_\gamma$ has $z_2$ coordinate
  satisfying
  $z_2=u^Nf_{2,0}(u)+v^2u^{Ns}h_2(u,v)=u^Nf_{2,0}(u)+\lambda u^{Nr}$,
  so $v^2u^{Ns}h_2(u,v)=\lambda u^{Nr}$. Write
  $g(u,v)=v^2h_2(u,v)-\lambda u^{N(r-s)}$. Since $h_2(u,v)$ is a unit,
  we have $g(0,v)\ne 0$, so we can write $v$ as a Puiseux expansion in
  terms of $u$ as follows:
  \begin{enumerate}
   \item If $N(r-s)$ is odd then we have one branch
$$v=\sqrt{\text{\Small$\frac\lambda{h_2(0,0)}$}} u^{\frac{N(r-s)}2}+\sum_{i>N(r-s)}a_iu^{i/2}\in\C\{u^{1/2}\}\,.$$
\item If $N(r-s)$ is even we have two branches
$$v=\pm\sqrt{\text{\Small$\frac\lambda{h_2(0,0)}$}} u^{\frac{N(r-s)}2}+\sum_{i>\frac{N(r-s)}2}b^{\pm}_iu^i\in\C\{u\}\,.$$
\end{enumerate}
Inserting into the parametrization of $A'_0$ we get that $L'_\gamma$ is given by:
\begin{align*}
  z_1&=u^N\\
z_2&=u^Nf_{2,0}(u)+\lambda u^{Nr}\\
z_j &=u^Nf_{j,0}(u)+\sqrt{\text{\Small$\frac\lambda{h_2(0,0)}$}}h_j(0,0)u^{\frac{N(r+s)}2}+h.o.\,,\quad j=3,\dots,n\,,
\end{align*}
where ``h.o.'' means higher order terms in $u$. Notice that $L'_\gamma$ has 
one or two branches depending on the parity of $N(r+s)$.

Taking  $\ell'=\ell_{\cal D_t}$ for some $t$, the curve $\ell'(L'_\gamma)$ is given by (since $r>\frac{r+s}2$):
\begin{align*}
  z_1&=u^N\\
z_2'&=u^N(f_{2,0}(u)-tf_{3,0}(u)) - t\sqrt{\text{\Small$\frac\lambda{h_2(0,0)}$}}h_j(0,0)u^{\frac{N(r+s)}2}+h.o.\,.
\end{align*}
If $N(r+s)$ is odd, then $\ell'(L'_\gamma)$ has one
component and $q(r) = \frac{s+r}{2}$ is an essential Puiseux exponent
of its Puiseux expansion. Otherwise, $\ell'(L'_\gamma)$ has
two components and $q(r)$ is the coincidence exponent between their
Puiseux expansions. In both cases, $q(r)$ is determined by the
topological type of $\ell'(L'_\gamma)$. 

Finally if $r=s$ then  $L'_\gamma$ consists to high order of two fibers of the saturation of $A_0$ by polars, so they have contact $q(r)=s$ with each other.
 \end{proof}

\section{Intermediate and complete \carrousels\ of the discriminant
  curve}\label{sec:intermediate}

In this section we define two \carrousels\ of $(\C^2,0)$
with respect to the discriminant curve $\Delta$ of a generic plane
projection $\ell \colon (X,0) \to (\C^2,0)$.

\subsection*{Intermediate \carrousel} This is obtained by truncating
the Puiseux series expansions of the branches of $\Delta$ as follows:
if $\Delta_0$ is a branch of $\Delta$ with Puiseux expansion $y=
\sum_{i\geq1} a_ix^{p_i}$ and if $s=p_k$ is the rate of a
$\Delta$-wedge about $\Delta_0$, then we consider the truncated
Puiseux series $y= \sum_{i=1}^k a_ix^{p_i}$ and form the carrousel
decomposition for these truncations (see the construction in Section
\ref{sec:carrousel}).  We call this the \emph{unamalgamated
  intermediate \carrousel}.  Notice that a piece of this carrousel
decomposition which contains a branch of $\Delta$ is in fact a
$\Delta$-wedge and is also a $D(s)$-piece, where $s$ is the rate of
this $\Delta$-wedge.

Using Lemma \ref{le:rules}, we then amalgamate according to the following rules:

\begin{amalgamation}[Intermediate \carrousel]
\label{amalg:intermediate}~
\begin{enumerate}
\item We amalgamate any $\Delta$-wedge piece with the piece outside
  it. We call the resulting pieces \emph{$\Delta$-pieces},
\item We then amalgamate any $D$-piece which is not a $\Delta$-piece with the piece outside it. This may create new $D$-pieces and we repeat this amalgamation iteratively until no further amalgamation is possible. 
\end{enumerate}
We call the result the \emph{intermediate \carrousel} of
$\Delta$.
\end{amalgamation}

It may happen that the rate $s$ of the  $\Delta$-wedge piece about a branch $\Delta_0$ of $\Delta$ is strictly
less that the last characteristic exponent of $\Delta_0$. We give an
example in \cite{BNP}. This is why we call this \carrousel\
``intermediate".

\subsection*{Complete \carrousel} This is obtained by truncating each
Puiseux series expansion at the first term which has exponent greater
than or equal to the polar rate and where truncation does not affect
the topology of $\Delta$. By definition, \red{$\Delta$-wedge pieces are not amalgamated}, so this \carrousel\ refines the
$\Delta$-pieces of the intermediate \carrousel. We then amalgamate iteratively any $D$-pieces which do not
contain components of $\Delta$.
  
We call the result the \emph{complete \carrousel} of
$\Delta$.
  
The complete \carrousel\ is a refinement of both the plain and the
intermediate \carrousels.  In particular, according to Proposition \ref{prop:curve from carrousel}, its combinatorics determine
the topology of the curve $\Delta$.  In Section \ref{sec:finding
  remaining polars}, we complete the proof of part \eqref{it:discriminant} of Theorem
\ref{th:invariants from geometry} by proving that the outer
Lipschitz geometry of $(X,0)$ determines the combinatorics of a
complete carrousel section of $\Delta$.

\smallskip We close this section with two examples illustrating carrousels.

\begin{example}\label{ex:D5 carrousel}
  Let $(X,0)$ be the $D_5$ singularity with equation $x^2y + y^4
  +z^2=0$. The discriminant curve of the generic projection
  $\ell=(x,y)$ has two branches $y=0$ and $x^2+y^3=0$, giving us plain
  carrousel rates of $1$ and $3/2$.  In Example \ref{ex:VTrates}, we
  will compute the corresponding polar rates, which equal $2$ and
  $5/2$ respectively, giving the additional rates which show up in the
  intermediate and complete carrousels. See also Example
  \ref{ex:D5-resolution}.

  Figure \ref{fig:carrousel D5} shows the sections of three different
  carrousels for the discriminant of the generic plane projection of
  the singularity $D_5 $: the plain carrousel and the \red{complete and intermediate} \carrousels.   The
  $\Delta$-pieces of the intermediate carrousel section
  are in gray, and the pieces they contain  in the complete carrousel are $D$-pieces which are in
  fact $\Delta$-wedge pieces since for both branches of $\Delta$ the
  polar rate is greater than the last characteristic exponent.
\end{example}

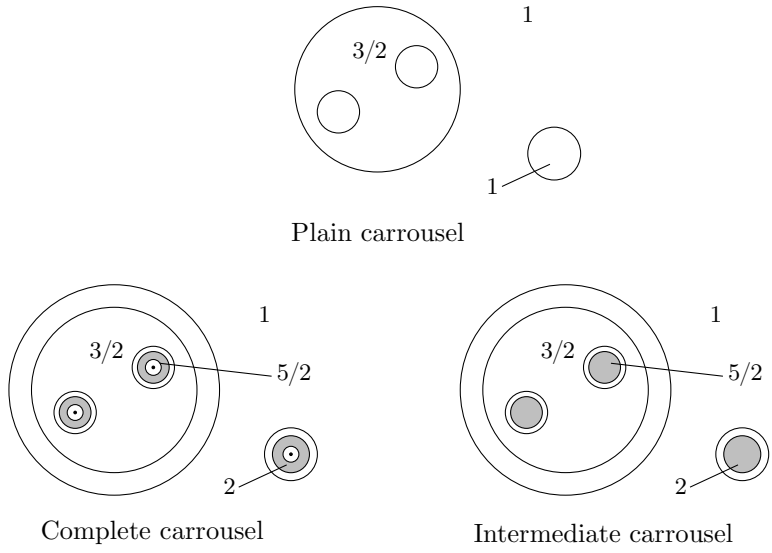
\begin{figure}
\centering
\begin{tikzpicture}
   
\begin{scope}[xshift=3.5cm]
  \draw[ fill=white] (0:0)circle(1.1);

    \node(a)at(-0.1,0.5){\small$3/2$};
      \node(a)at(2,1){\small$1$};
    
    \draw[  ] ( 0:0)+(210:0.6)circle(8pt);
     
     

     \draw[  ] ( 0:0)+(30:0.6)circle(8pt);
  
  
   \draw[ ] ( -20:2.5) circle(10pt);

  \draw[thin](-23.5:2.5)--(-38:2.1);
 \node(a)at(-40:2){\small$1$};
\node(a)at(0,-1.9){Plain carrousel};
\end{scope}

 \begin{scope}[yshift=-4cm]
 
    \draw[  ] (0:0)circle(1.4);
   \draw[fill=white] (0:0)circle(1.1);

    \node(a)at(-0.1,0.5){\small$3/2$};
      \node(a)at(2,1){\small$1$};
    

         \draw[ ] ( 0:0)+(210:0.6)circle(8pt);
          \draw[  fill=lightgray] ( 0:0)+(210:0.6)circle(6pt);
           \draw[ fill=white] ( 0:0)+(210:0.6)circle(3pt);
     \draw[fill=black] ( 0:0)+(210:0.6)circle(0.5pt);


     
        \draw[ fill=lightgray ] ( 0:0)+(30:0.6)circle(6pt);
        \draw[ fill=white] ( 0:0)+(30:0.6)circle(3pt);
         \draw[  ] ( 0:0)+(30:0.6)circle(8pt);

  \draw[fill=black] (0:0)+(30:0.6)circle(0.5pt);
  
  
   \draw[ ] ( -20:2.5) circle(10pt);
     \draw[  fill=lightgray] ( -20:2.5) circle(7pt);
      \draw[ fill=white] ( -20:2.5) circle(3pt);
  \draw[fill=black] ( -20:2.5) circle(0.5pt);
 
    \draw[thin](30.5:0.7)--(5:2.1);
 \node(a)at(5:2.4){\small   $ 5/2$};

 \draw[thin](-23.5:2.5)--(-38:2.1);
 \node(a)at(-40:2){\small$2$};
 \node(a)at(0.5,-1.9){Complete carrousel};
 
  \begin{scope}[xshift=6cm]
   \draw[  ] ( 0:0)circle(1.4);
   \draw[ fill=white] (0:0)circle(1.1);

    \node(a)at(-0.1,0.5){\small$3/2$};
      \node(a)at(2,1){\small$1$};
    

         \draw[ ] ( 0:0)+(210:0.6)circle(8pt);
          \draw[  fill=lightgray] ( 0:0)+(210:0.6)circle(6pt);


     
        \draw[ fill=lightgray ] ( 0:0)+(30:0.6)circle(6pt);
         \draw[  ] ( 0:0)+(30:0.6)circle(8pt);

  
  
   \draw[ ] ( -20:2.5) circle(10pt);
     \draw[  fill=lightgray] ( -20:2.5) circle(7pt);
 
 \draw[thin](30.5:0.7)--(5:2.1);
 \node(a)at(5:2.4){\small   $ 5/2$};

 \draw[thin](-23.5:2.5)--(-38:2.1);
 \node(a)at(-40:2){\small$2$};
 \node(a)at(0.5,-1.9){Intermediate carrousel};
\end{scope}
\end{scope}
\end{tikzpicture} 
\caption{Carrousel sections for $D_5$}
\label{fig:carrousel D5}
\end{figure}

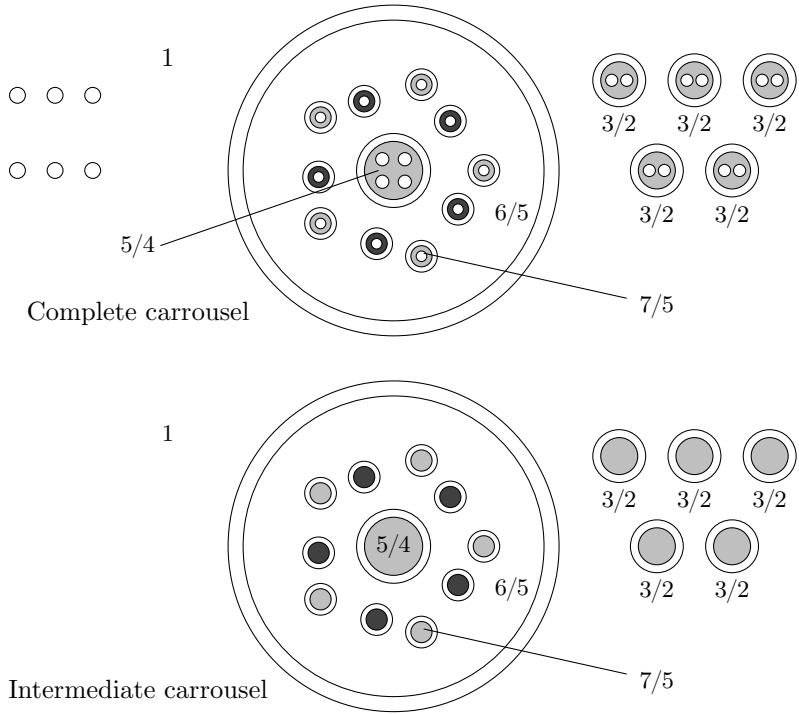
\begin{figure}
\begin{center}

\vbox to 12pt{}

\begin{tikzpicture} 
\begin{scope}[xshift=6cm, yshift=-1.5cm]
\draw[  ] (2,0)circle(10pt);
\draw[fill=lightgray] (2,0)circle(7pt);
\draw[fill=white] ( 2.1,0)circle(2.3pt);
\draw[fill=white] ( 1.9,0)circle(2.3pt);
\node(a)at(2,-0.6){\small$3/2$};
  
 \begin{scope} [xshift=1cm] 
 \draw[  ] (2,0)circle(10pt);
 \draw[ fill=lightgray] (2,0)circle(7pt);
 \draw[fill=white] ( 2.1,0)circle(2.3pt);
 \draw[fill=white] ( 1.9,0)circle(2.3pt);
 \node(a)at(2,-0.6){\small$3/2$};
 \end{scope}
 
 \begin{scope} [xshift=2cm] 
 \draw[ ] (2,0)circle(10pt);
 \draw[ fill=lightgray] (2,0)circle(7pt);
 \draw[fill=white] ( 2.1,0)circle(2.3pt);
 \draw[fill=white] ( 1.9,0)circle(2.3pt);
 \node(a)at(2,-0.6){\small$3/2$};
 \end{scope}
 
 \begin{scope} [xshift=0.5cm, yshift=-1.2cm] 
 \draw[  ] (2,0)circle(10pt);
 \draw[ fill=lightgray] (2,0)circle(7pt);
 \draw[fill=white] ( 2.1,0)circle(2.3pt);
 \draw[fill=white] ( 1.9,0)circle(2.3pt);
 \node(a)at(2,-0.6){\small$3/2$};
  
   \begin{scope} [xshift=1cm] 
   \draw[  ] (2,0)circle(10pt);
   \draw[ fill=lightgray] (2,0)circle(7pt);
   \draw[fill=white] ( 2.1,0)circle(2.3pt);
   \draw[fill=white] ( 1.9,0)circle(2.3pt);
   \node(a)at(2,-0.6){\small$3/2$};
   \end{scope}
 \end{scope}
\end{scope}

\begin{scope}[yshift=-2.7cm]
\draw[ ] (0:5) circle(2cm);
\draw[ ] (0:5) circle(2.2cm);
  
\draw[  ] (5,0)circle(14pt);
\draw[ fill=lightgray] (5,0)circle(11pt);
     
\draw[fill=white] ( 4.85,-0.15)circle(2.5pt);
\draw[fill=white] ( 4.85,0.15)circle(2.5pt);
\draw[fill=white] ( 5.15,-0.15)circle(2.5pt);
\draw[fill=white] ( 5.15,0.15)circle(2.5pt);
   
\draw[fill=lightgray] (0:5)+(0:1.2)circle(4pt);
\draw[fill=white] (0:5)+(0:1.2)circle(2pt);
\draw[  ] (0:5)+(0:1.2)circle(6pt);

\draw[fill=lightgray ] (0:5)+(72:1.2)circle(4pt);
\draw[fill=white] (0:5)+(72:1.2)circle(2pt);
\draw[ ] (0:5)+(72:1.2)circle(6pt);
     
\draw[fill=lightgray ] (0:5)+(144:1.2)circle(4pt);
\draw[fill=white] (0:5)+(144:1.2)circle(2pt);
\draw[ ] (0:5)+(144:1.2)circle(6pt);

\draw[fill=lightgray ] (0:5)+(216:1.2)circle(4pt);
\draw[fill=white] (0:5)+(216:1.2)circle(2pt);
\draw[] (0:5)+(216:1.2)circle(6pt);

\draw[fill=lightgray ] (0:5)+(-72:1.2)circle(4pt);
\draw[fill=white] (0:5)+(-72:1.2)circle(2pt);
\draw[] (0:5)+(-72:1.2)circle(6pt);
  
\draw[fill=darkgray] (0:5)+(185:1.0)circle(4pt);
\draw[fill=white] (0:5)+(185:1.0)circle(2pt);
\draw[ ] (0:5)+(185:1.0)circle(6pt);

\draw[fill=darkgray] (0:5)+(257:1.0)circle(4pt);
\draw[fill=white] (0:5)+(257:1.0)circle(2pt);
\draw[] (0:5)+(257:1.0)circle(6pt);

\draw[fill=darkgray ] (0:5)+(-31:1.0)circle(4pt);
\draw[fill=white] (0:5)+(-31:1.0)circle(2pt);
\draw[ ] (0:5)+(-31:1.0)circle(6pt);
         
\draw[fill=darkgray ] (0:5)+(41:1.0)circle(4pt);
\draw[fill=white] (0:5)+(41:1.0)circle(2pt);
\draw[  ] (0:5)+(41:1.0)circle(6pt);
             
\draw[fill=darkgray ] (0:5)+(-247:1.0)circle(4pt);
\draw[fill=white] (0:5)+(-247:1.0)circle(2pt);
\draw[ ] (0:5)+(-247:1.0)circle(6pt);
      
\node(a)at(-5:6.6){\small$6/5$};
\node(a)at(1.6,-1){\small ${5/4}$};
\draw(1.9,-1)--(4.8,0);
\node(a)at(8.5,-1.8){\small ${7/5}$};
\draw(8.1,-1.7)--(5.4,-1.1);
 
\draw[ ] (1,1)circle(3pt);
\draw[ ] (0.5,1)circle(3pt);
\draw[ ] (0,1)circle(3pt);
\draw[ ] (1,0)circle(3pt);
\draw[ ] (0.5,0)circle(3pt);
\draw[ ] (0,0)circle(3pt);

\node(a)at(2,1.5){${\small 1}$};       

\node(a)at(1.6,-1.9){Complete carrousel};
\end{scope} 

\begin{scope}[yshift=-5cm]
 \begin{scope}[xshift=6cm, yshift=-1.5cm]

 \draw[ ] (2,0)circle(10pt);
 \draw[fill=lightgray] (2,0)circle(7pt);
 \node(a)at(2,-0.6){\small$3/2$};
   
  \begin{scope} [xshift=1cm] 
  \draw[  ] (2,0)circle(10pt);
  \draw[ fill=lightgray] (2,0)circle(7pt);
  \node(a)at(2,-0.6){\small$3/2$};
  \end{scope}
  
  \begin{scope} [xshift=2cm] 
  \draw[ ] (2,0)circle(10pt);
  \draw[ fill=lightgray] (2,0)circle(7pt);
  \node(a)at(2,-0.6){\small$3/2$};
  \end{scope}

  \begin{scope} [xshift=0.5cm, yshift=-1.2cm] 
  \draw[ ] (2,0)circle(10pt);
  \draw[fill=lightgray] (2,0)circle(7pt);
  \node(a)at(2,-0.6){\small$3/2$};
  
   \begin{scope} [xshift=1cm] 
   \draw[  ] (2,0)circle(10pt);
   \draw[ fill=lightgray] (2,0)circle(7pt);
   \node(a)at(2,-0.6){\small$3/2$};
   \end{scope}
  \end{scope}
 \end{scope}

 \begin{scope}[yshift=-2.7cm]
 \draw[ ] (0:5) circle(2cm);
 \draw[ ] (0:5) circle(2.2cm);

 \draw[ ] (5,0)circle(14pt);
 \draw[fill=lightgray] (5,0)circle(11pt);
 \node(a)at(5,0){\small$5/4$};
     
 \draw[fill=lightgray] (0:5)+(0:1.2)circle(4pt);
 \draw[ ] (0:5)+(0:1.2)circle(6pt);

 \draw[fill=lightgray ] (0:5)+(72:1.2)circle(4pt);
 \draw[ ] (0:5)+(72:1.2)circle(6pt);
     
 \draw[fill=lightgray ] (0:5)+(144:1.2)circle(4pt);
 \draw[ ] (0:5)+(144:1.2)circle(6pt);

 \draw[fill=lightgray ] (0:5)+(216:1.2)circle(4pt);
 \draw[] (0:5)+(216:1.2)circle(6pt);

 \draw[fill=lightgray ] (0:5)+(-72:1.2)circle(4pt);
 \draw[] (0:5)+(-72:1.2)circle(6pt);
  
 \draw[ ] (0:5)+(185:1.0)circle(6pt);
 \draw[fill=darkgray] (0:5)+(185:1.0)circle(4pt);
    
 \draw[] (0:5)+(257:1.0)circle(6pt);
 \draw[fill=darkgray] (0:5)+(257:1.0)circle(4pt);
       
 \draw[fill=darkgray ] (0:5)+(-31:1.0)circle(4pt);
 \draw[ ] (0:5)+(-31:1.0)circle(6pt);
         
 \draw[fill=darkgray ] (0:5)+(41:1.0)circle(4pt);
 \draw[  ] (0:5)+(41:1.0)circle(6pt);
             
 \draw[fill=darkgray ] (0:5)+(-247:1.0)circle(4pt);
\draw[ ] (0:5)+(-247:1.0)circle(6pt);
      
\node(a)at(-5:6.6){\small$6/5$};
\node(a)at(2,1.5){\small ${1}$};
\node(a)at(8.5,-1.8){\small ${7/5}$};
\draw(8.1,-1.7)--(5.4,-1.1);
\node(a)at(1.6,-1.9){Intermediate carrousel};
\end{scope}    
\end{scope}  
\end{tikzpicture}  
\end{center}\caption{Carrousel sections for $ (zx^2+y^3)(x^3+zy^2)+z^7= 0 $}\label{fig:carrousel sections}
\end{figure}

\begin{example}\label{ex:very big0}
Our next example was
already partially studied in \cite{BNP}: the surface singularity
$(X,0)$ with equation $ (zx^2+y^3)(x^3+zy^2)+z^7= 0$.

In this case  $\Delta$  has $14$ branches with $12$ distinct tangent
  lines $L_1,\ldots,L_{12}$. $\Delta$ decomposes as follows:
  \begin{enumerate}
  \item Six branches, each lifting to a component of the polar in one
    of the two thick parts of $X$. Their tangent lines $L_1,\dots,L_6$
    move as the linear projection is changed, so their polar rates are $1$.
  \item Five branches, each tangent to one of $L_7,\ldots,L_{11}$, and
    with $3/2$ as single characteristic Puiseux exponent. So their
    Puiseux expansions have the form $ u=a_iv + b_i v^{3/2} + \cdots
    $. Their polar rates are $3/2$.
\item Three branches $\delta, \delta', \delta''$ tangent to
  the same line $L_{12}$, each with single characteristic Puiseux exponent,
  respectively $6/5, 6/5, 5/4$. Their Puiseux expansions have the form: 
        $\delta:
u = av +  b v^{6/5} +  \cdots$,
$\delta': u = av +b' v^{6/5}  +\cdots$ and 
$ \delta'': u =  av +b'' v^{5/4} +   \cdots$. 
Their polar rates are respectively $7/5, 7/5, 5/4$.
  \end{enumerate}

  The polar rates were given in \cite[Example 15.2]{BNP} except for
  the two with value $7/5$, whose computation is explained in Section
  \ref{sec:explicit computation}.  Figure \ref{fig:carrousel sections}
  shows the complete and intermediate carrousel sections. The gray
  regions in the intermediate carrousel section represent
  $\Delta$-pieces with rates $>1$ (the $B(1)$-piece is also a
  $\Delta$-piece).
\end{example}

\section{Geometric decomposition of $(X,0)$}
\label{sec:carrousel1}

In this section we describe the geometric decomposition of $(X,0)$ as a union of
semi-algebraic subgerms by starting with the
unamalgamated intermediate \carrousel\ of the discriminant curve
of a generic plane projection, lifting it to $(X,0)$, and then
performing an amalgamation of some pieces (see Lemma \ref{le:rules}).

\subsection*{Setup} 
From now on we assume $(X,0)\subset
(\C^n,0)$ and our coordinates $(z_1\dots,z_n)$ in $\C^n$ are chosen so
that $z_1$ and $z_2$ are generic linear forms and
$\ell:=(z_1,z_2)\colon X\to \C^2$ is a generic linear projection for
$X$.  
 We denote by $\Pi$ the polar curve of $\ell$ and by $\Delta = \ell
 (\Pi)$ its discriminant curve.

 Instead of considering a standard $\epsilon$-ball ${\Bbb B}_\epsilon$
 as Milnor ball for $(X,0)$, we will use, as in \cite{BNP}, a standard
 ``Milnor tube'' associated with the Milnor-L\^e fibration for the map
 $h:=z_1|_X\colon X\to \C$. Namely, for some sufficiently small
 $\epsilon_0$ and some $R>1$ we define for $\epsilon\le\epsilon_0$:
$$B_\epsilon:=\{(z_1,\dots,z_n):|z_1|\le \epsilon, |(z_1,\dots,z_n)|\le
R\epsilon\}\quad\text{and}\quad S_\epsilon=\partial B_\epsilon\,,$$
where $\epsilon_0$ and $R$ are chosen so that for $\epsilon\le
\epsilon_0$:
\begin{enumerate}
\item\label{it:mb1} $h^{-1}(t)$ intersects
the standard sphere ${\Bbb S}_{R\epsilon}$ transversely for $|t|\le \epsilon$;
\item\label{it:mb2} the polar curve $\Pi$ and its tangent cone meet $S_\epsilon$ only in the part
  $|z_1|=\epsilon$.
\end{enumerate}
The existence of such $\epsilon_0$ and $R$ is proved in \cite[Section 4]{BNP}.

We will now always work inside the Milnor balls
  $B_\epsilon$ just defined.
\begin{definition}\label{def:component}
    A \emph{component} of a semi-algebraic germ $(A,0)$ means the closure of a connected component of $(A\cap
  B_{\epsilon})\setminus \{0\}$ for sufficiently small $\epsilon$
  (i.e., the family of $B_{\epsilon'}$ with $\epsilon'\le \epsilon$
  should be a family of Milnor balls for $A$).
\end{definition}

  We lift the unamalgamated intermediate \carrousel\ of $\Delta$ to
  $(X,0)$ by $\ell$.  Any component of the inverse image of a piece of any one of the
  types $B(q)$, $A(q,q')$ or $D(q)$ is a piece of
  the same type from the point of view of its inner geometry (see
  \cite{BNP} for details). By a ``piece'' of the decomposition of
  $(X,0)$ we will always mean a component of the inverse image of a piece of $(\C^2,0)$.   
  
\begin{amalgamation}[Amalgamation in $X$] \label{amalg:X} We now simplify this
  decomposition of $(X,0)$ by amalgamating pieces in the two
  following steps.
\begin{enumerate} 
\item\label{it:amalg:X1} \emph{Amalgamating polar wedge pieces}. We
  amalgamate any polar wedge piece with the piece outside it. We call the resulting pieces \emph{polar pieces}.
\item\label{it:amalg:X2}\emph{Amalgamating empty $D$-pieces}.  Whenever a piece of $X$ is
    a $D(q)$-piece containing no part of the polar curve (we speak of
    an \emph{empty} piece) we amalgamate it with the piece which has a
    common boundary with it. These
    amalgamations may form new empty $D$-pieces. We continue this amalgamation iteratively until the only remaining
  $D$-pieces contain components of the polar curve.
\end{enumerate}
Polar pieces are then all the $D$-pieces and some of the $A(q,q)$- and  $B$-pieces.
\end{amalgamation}


\begin{remark} The inverse image of a $\Delta$-wedge $B_0$ may consist
  of several pieces. All of them are $D$-pieces and at least one of
  them is a polar wedge in the sense of Definition \ref{def:wedges}.
  In fact a polar piece contains at most one polar wedge over
  $B_0$. The argument is as follows. Assume $B_0$ is a
  \red{$\Delta$}-wedge about the component $\Delta_0$ of $\Delta$. Let
  $\Pi_0$ be the component of $\Pi$ such that $\Delta_0 = \ell(\Pi_0)$
  and let $N$ be the polar piece containing it.  Assume $\Pi'_0$ is
  another component of $\Pi$ inside $N$. According to Section
 \ref{sec:polarwedges}, $\Pi \cap N$ consists of equisingular components having
  pairwise contact $s$.  Since the projection $\ell_{\Pi}\colon \Pi
  \to \Delta$ is generic (\cite[Lemme 1.2.2 ii)]{T3}), we also have
  $d(\ell(\Pi_0 \cap S_{\epsilon}),\ell(\Pi'_0 \cap S_{\epsilon})) =
  O(\epsilon^s)$. Therefore, if $B_0$ is small enough, we get $B_0
  \cap \ell(\Pi'_0) = \{0\}$.
\end{remark}

Except for the $B(1)$-piece, each piece of the
decomposition of $(\C^2,0)$ has one ``outer boundary'' and some number
(possibly zero) of ``inner boundaries.''  When we lift pieces to
$(X,0)$ we use the same terminology \emph{outer boundary} or
\emph{inner boundary} for the boundary components of components of
lifted pieces.

After performing Amalgamation \ref{amalg:X},
if $q$ is the rate of some piece of $(X,0)$ we denote by $X_q$ the
union of all pieces of $X$ with this rate $q$. There is a finite
collection of such rates $q_1> q_2>\dots >q_{\nu}$. In fact,
$q_{\nu}=1$ since $B(1)$-pieces always exist, and $X_1$ is the union
of the $B(1)$-pieces.
\begin{definition}[\bf Geometric decomposition of $(X,0)$] \label{def:geometric decomposition}
We say 
  \emph{geometric decomposition of $(X,0)$}
for the union
$$
  (X,0)=\bigcup_{i=1}^\nu (X_{q_i},0)\cup\bigcup_{i>j}(A_{q_i,q_j},0)\,,
$$
where $A_{q_i,q_j}$ is a union of intermediate $A(q_i,q_j)$-pieces
between $X_{q_i}$ and $X_{q_j}$ and the semi-algebraic sets
$X_{q_i}$ and $A_{q_i,q_j}$ are pasted along their boundary
components.
\end{definition}

Note that the construction of pieces via  a carrousel
  involves choices which imply that, even after fixing a generic plane
  projection, the pieces are only well defined up to adding or
  removing collars of type $A(q,q)$ at their boundaries.
\begin{definition}[\bf Equivalence of pieces]\label{def:equivalent pieces}
We  say that two  pieces with same rate $q$ are \emph{equivalent}
if they can be made equal by attaching collars of type $A(q,q)$ at their boundaries. Similarly, two $A(q,q')$-pieces are equivalent if they can be made equal by removing $A(q,q)$ collars at their outer boundaries and $A(q',q')$ collars at their inner boundaries.
\end{definition}
The following is immediate from Propositions \ref{prop:resolution} \red{and \ref{prop:Anne doesn't like this}} below:
\begin{proposition}\label{prop:unicity}
  The geometric decomposition is unique up to equivalence of the pieces. \qed
\end{proposition}

 \section{Geometric decomposition through
  resolution}\label{sec:decomposition vs resolution}

In this section, we describe  the geometric decomposition  using  a
suitable  resolution of $(X,0)$. \red{We will do this in two steps, first considering the following coarser decomposition: 

\begin{definition}[\bf Inner geometric decomposition]\label{def:gdp} The \emph{inner geometric decomposition} of $(X,0)$  is obtained from the geometric decomposition \ref{def:geometric decomposition}  by  Amalgamations \ref{amalg:X} and adding:
  \begin{enumerate}
  \item[(3)]we amalgamate $B$-pieces which are $A(q,q)$-pieces (but not polar pieces) with the $A$-pieces adjacent to them.
  \end{enumerate}
We remark that, other than that polar pieces are never amalgamated, the inner geometric decomposition is the decomposition used in \cite{BNP} to classify inner bilipschitz geometry of $X$ (see Theorem 1.7 and Sec.\ 15 of \cite{BNP}).
\end{definition}}

We first need to describe the Nash modification of $(X,0)$.

\begin{definition}[\bf Nash modification]\label{def:Nash modification}
  Let $\lambda \colon X\setminus\{0\} \to \grassman(2,n)$ be the Gauss
  map which maps $x \in X\setminus\{0\}$ to the tangent plane $T_xX$.
  The closure $\widecheck X$ of the graph of $\lambda$ in $X \times
  \grassman(2,n)$ is a reduced analytic surface. The \emph{Nash
    modification of $(X,0)$} is the induced morphism $\nu\colon
  \widecheck X \to X$.
\end{definition}
  
According to \cite[Part III, Theorem 1.2]{spivakovsky}, a resolution of
$(X,0)$ factors through the Nash modification if and only if it has no
basepoints for the family of polar curves $\Pi_{\cal D}$ parametrized
by $\cal D \in \Omega$.

In this section we consider the  minimal good
resolution $\pi\colon (\widetilde X,E) \to (X,0)$ with the following three properties:
\begin{enumerate}
\item it resolves the basepoints of a general linear system of
  hyperplane sections of $(X,0)$ (i.e., it factors through the
  normalized blow-up of the maximal ideal of $X$); 
\item it resolves the basepoints of the family of polar curves of
  generic plane projections (i.e., it factors through the
  Nash modification of $X$);
  \item there are no adjacent nodes in the resolution graph (nodes are defined in Definition \ref{def:LP-curve} below).
\end{enumerate}
This resolution is obtained from the minimal good resolution of $(X,0)$ by
blowing up further until the basepoints of the two kinds are
resolved and then by  blowing up intersection points of exceptional curves corresponding to nodes. 
 
We denote by $\Gamma$ the dual resolution
graph and by $E_1,\ldots,E_r$  the exceptional curves in $E$. We
denote by $v_k$ the vertex of $\Gamma$ corresponding to $E_k$.

\begin{definition}\label{def:LP-curve} An \emph{\L-curve} is an 
   exceptional curve in $\pi^{-1}(0)$ which intersects the strict
   transform of a generic hyperplane section. The vertex of $\Gamma$
   representing an \L-curve is an \emph{\L-node}.

   A \emph{\P-curve} (\P \ for ``polar") is an exceptional curve
   in $\pi^{-1}(0)$ which intersects the strict transform of the polar
   curve of any generic linear projection. The vertex of $\Gamma$
   representing this curve is a \emph{\P-node}.
   
    A vertex of  $\Gamma$ is called a \emph{node} if it is
an \L- or \P-node or has valency $\ge 3$ or represents an exceptional curve of
genus $>0$. 

A \emph{string} of a resolution graph is a connected
  subgraph whose vertices have valency $2$ and are not nodes, and a
  \emph{bamboo} is a string attached to a non-node vertex of valency
  $1$. 
\end{definition}

For each $k=1,\ldots,r$, let $N(E_k)$ be a
small closed tubular neighbourhood of $E_k$ and let $${\cal N}(E_{k})
= \overline{N(E_k) \setminus \bigcup_{k' \neq k} N(E_{k'})}.$$ For any
subgraph $\Gamma'$ of $\Gamma$ define:
$$N(\Gamma'):= \bigcup_{v_k \in\Gamma'}N(E_k)\quad\text{and}\quad
\Nn(\Gamma'):= 
\overline{N(\Gamma)\setminus \bigcup_{v_k \notin \Gamma'}N(E_k)}\,.$$

\begin{proposition}\label{prop:resolution} The pieces of the \red{inner} geometric
  decomposition of $(X,0)$
   (Definition \ref{def:gdp}) can be described as follows:

  For each node $v_i$ of\/ $\Gamma$, let\/ $\Gamma_{i}$ be $v_i$ union
  any attached bamboos. Then, with ``equivalence'' defined as in
  Definition \ref{def:equivalent pieces},
  \begin{enumerate}
  \item the $B(1)$-pieces are equivalent to the sets $\pi({\cal  N}(\Gamma_{j}))$ where $v_j$ is an \L-node;
  \item the polar pieces are equivalent to the sets $\pi({\cal N}(\Gamma_{j}))$ where
    $v_j$ is a \P-node;
  \item the remaining $B(q)$ are equivalent to the sets $\pi({\cal N}(\Gamma_{j}))$ where
    $v_j$ is a node which is neither an \L-node nor a \P-node;
  \item the $A(q,q')$-pieces with $q \neq q'$ are equivalent to the $\pi({
      N}(\sigma))$ where $\sigma$ is a maximal string between two
    nodes.
  \end{enumerate}
\end{proposition}

\red{We describe later how the geometric decomposition results from the inner one by modifying $\Gamma$ if necessary to a graph $\widehat\Gamma$ given by labeling some vertices on strings as nodes (after possibly extending the strings by blowing up).}
A straightforward consequence will be that to each node $v_j$ of
\red{$\widehat\Gamma$} one can associate the rate $q_{j}$ of the corresponding
piece of the decomposition, and the graph \red{$\widehat\Gamma$} with nodes weighted
by their rates determines the geometric decomposition
$ (X,0)=\bigcup_{i=1}^\nu (X_{q_i},0)\cup\bigcup_{i>j}(A_{q_i,q_j},0)\,$ up to equivalence, proving  the unicity  stated in Proposition \ref{prop:unicity}. 

Before proving Proposition \ref{prop:resolution} we give two examples. \red{They each have geometric decomposition the same as inner geometric decomposition; an example where the decompositions differ is Example \ref{ex:vvb}.}
\begin{example} \label{ex:D5-resolution} Let $(X,0)$ be the $D_5$
  singularity with equation $x^2y + y^4 +z^2=0$. Carrousel
  decompositions for its discriminant $\Delta$ are described in
  Example \ref{ex:D5 carrousel}.

Let $v_1,\ldots,v_5$
  be the vertices of its minimal resolution graph indexed as follows
  (all Euler weights are $-2$):
       
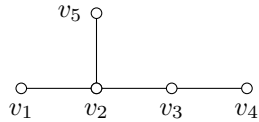
\begin{center}
\begin{tikzpicture}
  \draw[thin ](0,0)--(3,0); 
  \draw[thin ](1,0)--(1,1);
      

  \draw[fill=white   ] (1,0)circle(2pt);
  \draw[fill=white  ] (3,0)circle(2pt);
  \draw[fill =white ] (1,1)circle(2pt);      
  \draw[fill=white   ] (0,0)circle(2pt);
  \draw[fill=white   ] (2,0)circle(2pt);
  \draw[fill=white ] (1,0)circle(2pt); 
      
 \node(a)at(0,-0.3){   $v_1$};
 \node(a)at(1,-0.3){   $v_2$};
 \node(a)at(2,-0.3){   $v_3$};
 \node(a)at(3,-0.3){   $v_4$};
 \node(a)at(0.65,1){   $v_5$};
 \end{tikzpicture} 
\end{center}

The multiplicities of \red{the restriction $h$ to $X$ of a  generic linear form $\C^3 \to \C$} are given
  by the minimum of the compact part of the three divisors $(x)$,
  $(y)$ and $(z)$: $(h \circ \pi) = E_1 +2E_2+2E_3+E_4+E_5 + h^*$, where ${}^*$ means strict transform. 
 The minimal resolution resolves a general linear system of
  hyperplane sections and the strict transform of $h$ is one curve
  intersecting $E_3$.  So $v_3$ is the single \L-node. 
But as we shall see, this is not yet the resolution which resolves the
family of polars.

 The total transform by $\pi$ of the coordinate functions
$x, y$ and $z$ are:
\begin{align*}
  (x \circ \pi) &= 2E_1 +3E_2+2E_3+E_4+2E_5  +  x^* \\
  (y \circ \pi) &= E_1 +2E_2+2E_3+2E_4+E_5+ y^*\\
  (z \circ \pi) &= 2E_1 +4E_2+3E_3+2E_4+2E_5+   z^* \,.
\end{align*}

Set $f(x,y,z) = x^2y + y^4 +z^2$. The polar curve $\Pi$ of a generic
linear projection $\ell\colon (X,0) \to (\C^2,0)$ has equation $g=0$
where $g$ is a generic linear combination of the partial derivatives
$f_x = 2xy$, $f_y=x^2+4y^3$ and $f_z=2z$. The multiplicities of $g$
are given by the minimum of the compact part of the three divisors
\begin{align*}
  (f_x \circ \pi) &= 3E_1 +5E_2+4E_3+3E_4+3E_5+  f_x^* \\
  (f_y \circ \pi) &= 3E_1 +6E_2+4E_3+2E_4+3E_5+ f_y^*\\
  (f_z \circ \pi) &=  2E_1 +4E_2+3E_3+2E_4+2E_5+  f_z^*  \,.
\end{align*}
  So the total transform of $g$ is equal to:
$$(g \circ \pi) = 2E_1 +4E_2+3E_3+2E_4+{2}E_5+ \Pi^*\,.$$
In particular, $\Pi$ is resolved by $\pi$ and its strict transform
$\Pi^*$ has two components $\Pi_1^*$ and $\Pi_2^*$, which intersect
respectively $E_2$ and $E_4$.  

Since the multiplicities $m_2(f_x)=5$, $m_2(f_y)=4$ and $m_2(z)=6$
along $E_2$ are distinct, the family of polar curves of generic plane
projections has a basepoint on $E_2$. One must blow up once to
resolve the basepoint, creating a new exceptional curve $E_6$ and a
new vertex $v_6$ in the graph.  So we obtain two \P-nodes $v_4$ and $v_6$ as in the resolution graph below (omitted Euler weights
are $-2$):

   
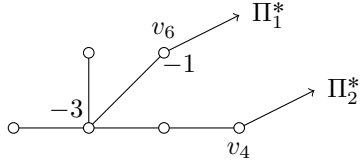
\begin{center}
\begin{tikzpicture}
 
   \draw[thin ](0,0)--(3,0);
\draw[thin ](1,0)--(1,1);
\draw[thin ](1,0)--(2,1);
      
        
   \draw[thin,>-stealth,->](2,1)--+(1,0.5);
    \draw[thin,>-stealth,->](3,0)--+(1,0.5);

\node(a)at(.7,.25){$-3$};
\node(a)at(2.2,.85){$-1$};
\node(a)at(2,1.3){   $v_6$};
 \node(a)at(3.4,1.5){   $\Pi_1^*$};
  \node(a)at(4.4,0.5){   $\Pi_2^*$};
  \node(a)at(3,-0.3){   $v_4$};

 \draw[fill=white   ] (0,0)circle(2pt);
  \draw[fill=white   ] (1,0)circle(2pt);
     \draw[fill=white   ] (2,0)circle(2pt);
  \draw[fill=white  ] (3,0)circle(2pt);
    \draw[fill=white  ] (2,1)circle(2pt);
     \draw[fill =white ] (1,1)circle(2pt);      
 \end{tikzpicture} 
\end{center}
 
We will compute in Example \ref{ex:VTrates} that the two polar rates
are $5/2$ and $2$.  It follows that the geometric decomposition
of $X$ consists of the pieces   
 $$X_{5/2}=\pi({\cal N}(E_6)), \ X_2=\pi({\cal
    N}(E_4)), \ X_{3/2}=\pi(\cal N(E_1 \cup
E_2 \cup E_5)) \hbox{ and } X_1=\pi(\cal N(E_3))\,,$$
plus intermediate A-pieces,  and that  $X_{5/2}$
and $X_2$ are the two polar pieces. 
The geometric decomposition for $D_5$ is described by
$\Gamma$ with vertices $v_i$ weighted with the corresponding
  rates $q_i$:
\begin{center}
\begin{tikzpicture}
 
   \draw[thin ](0,0)--(3,0);
\draw[thin ](1,0)--(1,1);
\draw[thin ](1,0)--(2,1);
      
        
   \draw[thin,>-stealth,->](2,1)--+(1,0.5);
    \draw[thin,>-stealth,->](3,0)--+(1,0.5);

  \draw[fill=white   ] (1,0)circle(2pt);
  \draw[fill=white  ] (3,0)circle(2pt);
     \draw[fill =white ] (1,1)circle(2pt);      
  \draw[fill=white   ] (0,0)circle(2pt);
   \draw[fill=white   ] (2,0)circle(2pt);
   \draw[fill=white ] (1,0)circle(2pt); 
     \draw[fill=white  ] (2,1)circle(2pt);

\node(a)at(.7,.25){$-3$};
\node(a)at(2.2,.85){$-1$};
      
\node(a)at(0,-0.3){   $3/2$};
\node(a)at(1,-0.3){   $3/2$};
\node(a)at(2,-0.3){   $1$};
 \node(a)at(3,-0.3){   $2$};
 \node(a)at(1,1.3){   $3/2$};
  \node(a)at(2.0,1.3){   $5/2$};
 \end{tikzpicture} 
  \end{center}

 \end{example}

\begin{example}\label{ex:very big1} We return to example \ref{ex:very big0}
   with equation $ (zx^2+y^3)(x^3+zy^2)+z^7= 0$, to clarify how the carrousels described there arise.

We first compute some invariants from the equation.  The
following picture shows the resolution graph of a general linear
system of hyperplane sections.  The negative numbers are
self-intersection of the exceptional curves while numbers in
parentheses are the multiplicities of a generic linear form. 
 
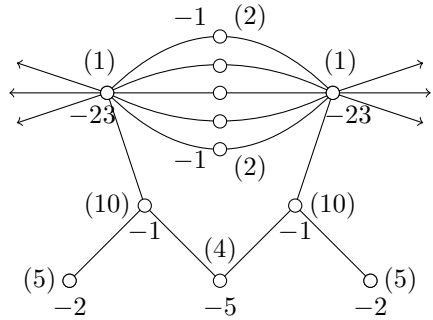
\begin{center}
\begin{tikzpicture}
  \draw[] (-2,0)circle(2.5pt);
  \draw[thin ](-2,0)--(-1,1);

  \draw[thin ](0:0)--(-1,1);
 
     \draw[thin ](0:0)--(1,1);
     
        \draw[thin ](1,1)--(2,0);
            \draw[] (2,0)circle(2.5pt);
   \draw[thin ](1,1)--(1.5,2.5);
    \draw[thin ](-1,1)--(-1.5,2.5);
    \draw[ fill=white ] (1.5,2.5)circle(2.5pt);
     \draw[ fill=white] (-1.5,2.5)circle(2.5pt);
     \draw[thin ](-1.5,2.5)--(1.5,2.5);
     
\draw[fill=white ] (-1,1)circle(2.5pt);
\draw[fill =white ] (1,1)circle(2.5pt);
\draw[fill=white] (0,0)circle(2.5pt);
\draw[fill=white] (2,0)circle(2.5pt);
\draw[fill=white] (-2,0)circle(2.5pt);

\draw[thin] (-1.5,2.5)..controls (-0.5,3) and (0.5,3)..(1.5,2.5);
\draw[thin] (-1.5,2.5)..controls (-0.5,2) and (0.5,2)..(1.5,2.5);
\draw[thin] (-1.5,2.5)..controls (-0.5,3.5) and (0.5,3.5)..(1.5,2.5);
\draw[thin] (-1.5,2.5)..controls (-0.5,1.5) and (0.5,1.5)..(1.5,2.5);

\node(a)at(-2,-0.35){$-2$};
\node(a)at(-2.4,0){$(5)$};

\node(a)at(0,-0.35){$-5$};
\node(a)at(0,0.4){$(4)$};

\node(a)at(2,-0.35){$-2$};
\node(a)at(2.4,0){$(5)$};

\node(a)at(-1,0.65){$-1$};
\node(a)at(-1.5,1){$(10)$};

\node(a)at(1,0.65){$-1$};
\node(a)at(1.5,1){$(10)$};

\node(a)at(-0.4,3.5){$-1$};
\node(a)at(0.4,3.5){$(2)$};
\node(a)at(-0.4,1.6){$-1$};
\node(a)at(0.4,1.5){$(2)$};
\node(a)at(-1.6,2.9){$(1)$};
\node(a)at(-1.7,2.2){$-23$};

\node(a)at(1.7,2.2){$-23$};
\node(a)at(1.6,2.9){$(1)$};
 
  \draw[thin,>-stealth,->](1.5,2.5)--+(1.2,0.4);
       \draw[thin,>-stealth,->](1.5,2.5)--+(1.3,0);
         \draw[thin,>-stealth,->](1.5,2.5)--+(1.2,-0.4);
         
            \draw[thin,>-stealth,->](-1.5,2.5)--+(-1.2,0.4);
       \draw[thin,>-stealth,->](-1.5,2.5)--+(-1.3,0);
         \draw[thin,>-stealth,->](-1.5,2.5)--+(-1.2,-0.4);

   \draw[ fill=white] (0,2.5)circle(2.5pt);
           \draw[ fill=white] (0,2.86)circle(2.5pt);
             \draw[ fill=white] (0,2.12)circle(2.5pt);
               \draw[ fill=white] (0,1.75)circle(2.5pt);
                 \draw[ fill=white] (0,3.25)circle(2.5pt);
  \draw[ fill=white ] (1.5,2.5)circle(2.5pt);
     \draw[ fill=white] (-1.5,2.5)circle(2.5pt);
  \end{tikzpicture} 
  \end{center}

Set $f(x,y,z) =   (zx^2+y^3)(x^3+zy^2)+z^7$. The multiplicities of the partial derivatives $(m(f_x),m(f_y),m(f_z))$ are given by: 
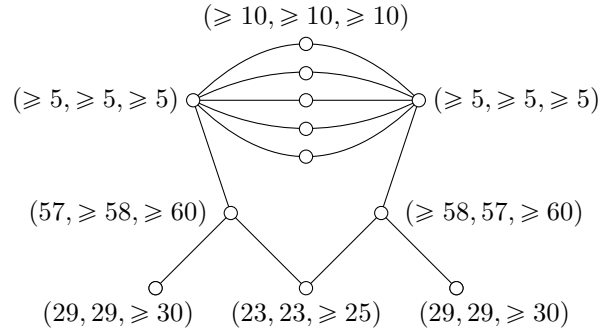
\begin{center}
\begin{tikzpicture}

              

 \draw[] (-2,0)circle(2.5pt);
  \draw[thin ](-2,0)--(-1,1);

  \draw[thin ](0:0)--(-1,1);
 
     \draw[thin ](0:0)--(1,1);

        \draw[thin ](1,1)--(2,0);
            \draw[] (2,0)circle(2.5pt);
   \draw[thin ](1,1)--(1.5,2.5);
    \draw[thin ](-1,1)--(-1.5,2.5);
 
     \draw[thin ](-1.5,2.5)--(1.5,2.5);

\draw[fill=white ] (-1,1)circle(2.5pt);
\draw[fill =white ] (1,1)circle(2.5pt);
\draw[fill=white] (0,0)circle(2.5pt);
\draw[fill=white] (2,0)circle(2.5pt);
\draw[fill=white] (-2,0)circle(2.5pt);

\draw[thin] (-1.5,2.5)..controls (-0.5,3) and (0.5,3)..(1.5,2.5);
\draw[thin] (-1.5,2.5)..controls (-0.5,2) and (0.5,2)..(1.5,2.5);
\draw[thin] (-1.5,2.5)..controls (-0.5,3.5) and (0.5,3.5)..(1.5,2.5);
\draw[thin] (-1.5,2.5)..controls (-0.5,1.5) and (0.5,1.5)..(1.5,2.5);
\node(a)at(-2.5,-0.35){$(29,29,\geq 30)$};
\node(a)at(-2.5,1){$(57,\geq 58,\geq 60)$};
\node(a)at(0,-0.35){$(23,23,\geq25)$};
\node(a)at(2.5,1){$(\geq 58,57,\geq 60)$};
\node(a)at(2.5,-0.35){$(29,29,\geq 30)$};


\node(a)at(0.,3.6){$(\geq 10,\geq10,\geq10)$};
\node(a)at(-2.8,2.5){$(\geq 5,\geq5,\geq5)$};
\node(a)at(2.8,2.5){$(\geq 5,\geq5,\geq5)$}; 

\draw[ fill=white] (0,2.5)circle(2.5pt);
           \draw[ fill=white] (0,2.86)circle(2.5pt);
             \draw[ fill=white] (0,2.12)circle(2.5pt);
               \draw[ fill=white] (0,1.75)circle(2.5pt);
                 \draw[ fill=white] (0,3.25)circle(2.5pt);
     \draw[ fill=white] (1.5,2.5)circle(2.5pt);
     \draw[ fill=white] (-1.5,2.5)circle(2.5pt);
 
  \end{tikzpicture} 
  \end{center}
  
  The basepoints of the family of polar curves of generic plane
  projections are resolved by this resolution and the strict transform
  of the generic polar curve $\Pi_{\ell}$ is given by the following
  graph. The multiplicities are those of a generic linear combination
  $g=af_x+bf_y+cf_z$.

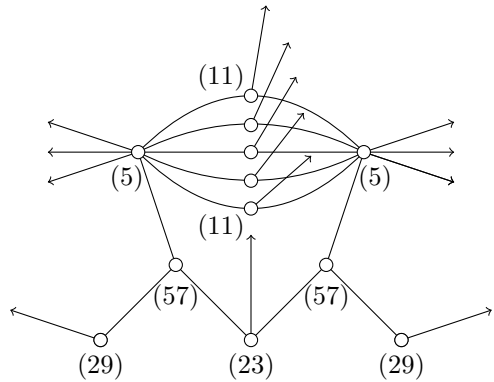
 \begin{center}
\begin{tikzpicture}  \draw[thin,>-stealth,->](-1.5,2.5)--+(-1.2,0.4);
   \draw[thin,>-stealth,->](-1.5,2.5)--+(-1.2,0);
    \draw[thin,>-stealth,->](-1.5,2.5)--+(-1.2,-0.4);
    
     \draw[thin,>-stealth,->](1.5,2.5)--+(1.2,0.4);
   \draw[thin,>-stealth,->](1.5,2.5)--+(1.2,0);
    \draw[thin,>-stealth,->](1.5,2.5)--+(1.2,-0.4);
    
     \draw[thin,>-stealth,->](1.5,2.5)--+(1.2,-0.4);
              
 \draw[thin,>-stealth,->](-2,0)--+(-1.2,0.4);
  \draw[thin,>-stealth,->](2,0)--+(1.2,0.4);
    \draw[thin,>-stealth,->](0,0)--+(0,1.4);

    \draw[thin,>-stealth,->](0,3.25)--+(0.2,1.2);
     \draw[thin,>-stealth,->](0,2.86)--+(0.5,1.1);
                  \draw[thin,>-stealth,->](0,2.5)--+(0.6,1);
                   \draw[thin,>-stealth,->](0,2.12)--+(0.7,0.9);

              \draw[thin,>-stealth,->](0,1.75)--+(0.8,0.7);

 \draw[] (-2,0)circle(2.5pt);
  \draw[thin ](-2,0)--(-1,1);

  \draw[thin ](0:0)--(-1,1);
 
     \draw[thin ](0:0)--(1,1);

        \draw[thin ](1,1)--(2,0);
            \draw[] (2,0)circle(2.5pt);
   \draw[thin ](1,1)--(1.5,2.5);
    \draw[thin ](-1,1)--(-1.5,2.5);
   
     \draw[thin ](-1.5,2.5)--(1.5,2.5);

\draw[fill=white ] (-1,1)circle(2.5pt);
\draw[fill =white ] (1,1)circle(2.5pt);
\draw[fill=white] (0,0)circle(2.5pt);
\draw[fill=white] (2,0)circle(2.5pt);
\draw[fill=white] (-2,0)circle(2.5pt);

\draw[thin] (-1.5,2.5)..controls (-0.5,3) and (0.5,3)..(1.5,2.5);
\draw[thin] (-1.5,2.5)..controls (-0.5,2) and (0.5,2)..(1.5,2.5);
\draw[thin] (-1.5,2.5)..controls (-0.5,3.5) and (0.5,3.5)..(1.5,2.5);
\draw[thin] (-1.5,2.5)..controls (-0.5,1.5) and (0.5,1.5)..(1.5,2.5);
\node(a)at(2,-0.35){$(29)$};
\node(a)at(1,0.57){$(57)$};
\node(a)at(0,-0.35){$(23)$};
\node(a)at(-1,0.57){$(57)$};
\node(a)at(-2,-0.35){$(29)$};

\node(a)at(-0.4,3.5){$(11)$};
\node(a)at(-0.4,1.5){$(11)$};
\node(a)at(-1.65,2.15){$(5)$};
\node(a)at(1.65,2.15){$(5)$};
 
     \draw[ fill=white] (0,2.5)circle(2.5pt);
           \draw[ fill=white] (0,2.86)circle(2.5pt);
             \draw[ fill=white] (0,2.12)circle(2.5pt);
               \draw[ fill=white] (0,1.75)circle(2.5pt);
                 \draw[ fill=white] (0,3.25)circle(2.5pt);
                  \draw[ fill=white] (1.5,2.5)circle(2.5pt);
     \draw[ fill=white] (-1.5,2.5)circle(2.5pt);
  \end{tikzpicture} 
  \end{center}

 We now indicate on the resolution graph the rate corresponding to each vertex. 
 
  
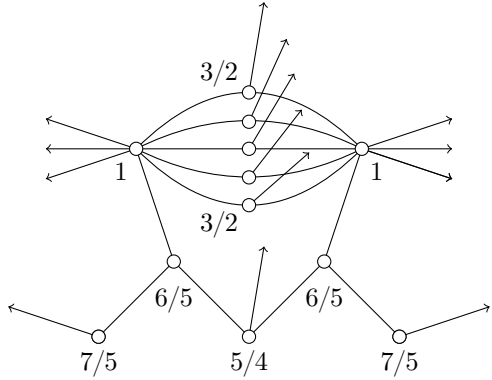
\begin{center}
\begin{tikzpicture}  \draw[thin,>-stealth,->](-1.5,2.5)--+(-1.2,0.4);
   \draw[thin,>-stealth,->](-1.5,2.5)--+(-1.2,0);
    \draw[thin,>-stealth,->](-1.5,2.5)--+(-1.2,-0.4);
    
     \draw[thin,>-stealth,->](1.5,2.5)--+(1.2,0.4);
   \draw[thin,>-stealth,->](1.5,2.5)--+(1.2,0);
    \draw[thin,>-stealth,->](1.5,2.5)--+(1.2,-0.4);
    
     \draw[thin,>-stealth,->](1.5,2.5)--+(1.2,-0.4);
              
 \draw[thin,>-stealth,->](-2,0)--+(-1.2,0.4);
  \draw[thin,>-stealth,->](2,0)--+(1.2,0.4);
    \draw[thin,>-stealth,->](0,0)--+(0.2,1.2);

    \draw[thin,>-stealth,->](0,3.25)--+(0.2,1.2);
     \draw[thin,>-stealth,->](0,2.86)--+(0.5,1.1);
                  \draw[thin,>-stealth,->](0,2.5)--+(0.6,1);
                   \draw[thin,>-stealth,->](0,2.12)--+(0.7,0.9);

              \draw[thin,>-stealth,->](0,1.75)--+(0.8,0.7);

 \draw[] (-2,0)circle(2.5pt);
  \draw[thin ](-2,0)--(-1,1);

  \draw[thin ](0:0)--(-1,1);
 
     \draw[thin ](0:0)--(1,1);

        \draw[thin ](1,1)--(2,0);
            \draw[] (2,0)circle(2.5pt);
   \draw[thin ](1,1)--(1.5,2.5);
    \draw[thin ](-1,1)--(-1.5,2.5);
 
     \draw[thin ](-1.5,2.5)--(1.5,2.5);

\draw[thin] (-1.5,2.5)..controls (-0.5,3) and (0.5,3)..(1.5,2.5);
\draw[thin] (-1.5,2.5)..controls (-0.5,2) and (0.5,2)..(1.5,2.5);
\draw[thin] (-1.5,2.5)..controls (-0.5,3.5) and (0.5,3.5)..(1.5,2.5);
\draw[thin] (-1.5,2.5)..controls (-0.5,1.5) and (0.5,1.5)..(1.5,2.5);
\node(a)at(2,-0.35){$7/5$};
\node(a)at(1,0.5){$6/5$};
\node(a)at(0,-0.35){$5/4$};
\node(a)at(-1,0.5){$6/5$};
\node(a)at(-2,-0.35){$7/5$};

\node(a)at(-0.4,3.5){$3/2$};
\node(a)at(-0.4,1.5){$3/2$};
\node(a)at(-1.7,2.2){$1$};
\node(a)at(1.7,2.2){$1$};
 
     \draw[ fill=white] (0,2.5)circle(2.5pt);
           \draw[ fill=white] (0,2.86)circle(2.5pt);
             \draw[ fill=white] (0,2.12)circle(2.5pt);
               \draw[ fill=white] (0,1.75)circle(2.5pt);
                 \draw[ fill=white] (0,3.25)circle(2.5pt);
                 
                  \draw[ fill=white] (1.5,2.5)circle(2.5pt);
     \draw[ fill=white] (-1.5,2.5)circle(2.5pt);    
\draw[fill=white ] (-1,1)circle(2.5pt);
\draw[fill =white ] (1,1)circle(2.5pt);
\draw[fill=white] (0,0)circle(2.5pt);
\draw[fill=white] (2,0)circle(2.5pt);
\draw[fill=white] (-2,0)circle(2.5pt);
\end{tikzpicture} 
\end{center}                
These rates were computed in 
  \cite{BNP} except for the rates $7/5$ which are explained in Section \ref{sec:explicit computation}
\end{example}

\begin{proof}[Proof of Proposition \ref{prop:resolution}]
  Let $\ell\colon(X,0)\to(\C^2,0)$ be again a generic plane projection. Let $\rho \colon Y \to \C^2$ be the minimal sequence of blow-ups
  starting with the blow-up of $0 \in \C^2$ which resolves the
  basepoints of the family of images $\ell(\Pi_{\cal D})$ by $\ell$ of
  the polar curves of generic plane projections and let $\Delta$ be 
\red{the discriminant curve of $\ell$}.  We set $\rho^{-1}(0) = \bigcup_{k=1}^m
  C_k$, where $C_1$ is the first curve blown up.
  
  Denote by $R$ the dual graph of $\rho$, so $v_{1}$ is its root
  vertex.  We say \emph{$\Delta$-curve} for an exceptional curve in
  $\rho^{-1}(0)$ intersecting the strict transform of $\Delta$, and  
  \emph{$\Delta$-node} for a vertex of $R$ which represents a
  $\Delta$-curve. We call any vertex of $R$ which is either $v_{1}$
  or a $\Delta$-node or a vertex with valency $\geq 3$ a \emph{node of
    $R$}.

  If two nodes are adjacent, we blow up the intersection points of the
  two corresponding curves in order to create a string between them.

  Consider the intermediate \carrousel\ of $\Delta$ after amalgamation
  of the $\Delta$-wedge pieces (so we don't amalgamate any other piece
  here).  This carrousel decomposition can be described
  as follows:
\begin{lemma}\label{lem:resolution} 
 \begin{enumerate} 
 \item \label{B(1)} The $B(1)$-piece is  equivalent to the set $\rho({\cal N}(C_{1}))$;
 \item \label{Delta} the  $\Delta$-pieces are  equivalent to the sets  $\rho({\cal N}(C_{k}))$ where
   $v_k$ is a $\Delta$-node;
 \item \label{B(q)} the remaining $B$-pieces are   equivalent to  the sets  $\rho({\cal N}(C_{k}))$ where
   $v_k$ is a node which is neither $v_{1}$ nor a $\Delta$-node;
    \item \label{D}  the sets  $\rho( N(\beta))$ where $\beta$ is a bamboo of $R$ \red{are equivalent to empty $D$-pieces};
    \item \label{A} the $A$-pieces are equivalent to the sets
      $\rho(N(\sigma))$ where $\sigma$ is a maximal string between two
      nodes.
\end{enumerate}
\end{lemma}

\begin{proof}[Proof of Lemma \ref{lem:resolution}] 
 Let $B_{\kappa}$ be a $B$-piece as defined in section \ref{sec:carrousel}:
\begin{align*}
  B_\kappa:=\Bigl\{(x,y):~&\alpha_\kappa|x^{p_k}|\le
  |y-\sum_{i=1}^{k-1}a_ix^{p_i}|\le
  \beta_\kappa|x^{p_k}|\\
  & |y-(\sum_{i=1}^{k-1}a_ix^{p_i}+a_{kj}x^{p_k})|\ge
  \gamma_\kappa|x^{p_k}|\text{ for }j=1,\dots,{m_\kappa}\Bigr\}
\end{align*}
with $\alpha_\kappa,\beta_\kappa,\gamma_\kappa$ satisfying
$\alpha_\kappa<|a_{kj}|-\gamma_\kappa<|a_{kj}|+\gamma_\kappa<\beta_\kappa$
for each $j=1,\dots,{m_\kappa}$.

Then $B_{\kappa}$ is foliated by the complex curves $C_{\lambda},
\lambda \in \Lambda$ with Puiseux expansions:
$$y=\sum_{i=1}^{k-1}a_ix^{p_i} + \lambda x^{p_k}\,,$$ 
where $\Lambda=\{\lambda \in \C ; \alpha_\kappa \leq |\lambda| \leq
\beta_\kappa \hbox{ and } |\lambda - a_{kj}| \geq \gamma_\kappa$,
$j=1,\dots,{m_\kappa}\}$.   The curves $C_{\lambda}$ are irreducible and equisingular in
terms of strong simultaneous resolution, and $B_{\kappa} =\rho({\cal
  N}(C))$ where $C$ is the exceptional component of $\rho^{-1}(0)$
which intersects the strict transforms $C_{\lambda}^*$.

We use again the notations of Section \ref{sec:carrousel}.  The
$B(1)$-piece is the closure of the cones $V^{(j)}, j=1,\ldots,m$, so it
is foliated by the complex lines with equations $y=ax$, where $a\in \C$
is such that $|a_1^{(j)}-a|\geq \eta$ for all $j=1,\ldots,m$. This
implies \eqref{B(1)} since $C_{1}$ is the exceptional curve which
intersects the strict transform of those lines.

Recall that a $\Delta$-piece is a $B$-piece obtained by amalgamating
$\Delta$-wedges with a $B$-piece outside them. By Definition
\ref{def:wedges}, a $\Delta$-wedge is foliated by curves $y=\sum_{p_j
  \leq s} a_jx^{p_j} + a x^s$, which are resolved in family by the
resolution $\rho$, their strict transforms intersecting a
$\Delta$-node. This proves \eqref{Delta} and then \eqref{B(q)}.

 \eqref{D} and  \eqref{A} immediately follow from the fact that the $D$- and $A$-pieces are the closures of the complement of the $B$-pieces.
\end{proof}
      
We now complete the proof of Proposition \ref{prop:resolution}. Let
$\pi_{h\!j} \colon X_{h\!j} \to X$ be the Hirzebruch-Jung resolution of $(X,0)$
obtained by pulling back the morphism $\rho$ by the cover $\ell$ and
then normalizing and resolving the remaining quasi-ordinary
singularities. We denote its dual graph by $\Gamma_{h\!j}$. Denote by $\ell'
\colon X_{h\!j} \to Y$ the morphism defined by $\ell \circ \pi_{h\!j} = \rho \circ
\ell'$.

\red{Now $\pi_{h\!j}$ factors through $\pi$ by a morphism $\mu
  \colon X_{h\!j} \to \widetilde{X}$. Let us choose a bundle neighborhood
  $N(C_k)$ for each component $C_k$ of $\rho^{-1}(0)$. Then $\ell'^{-1}(N(C_k))$ is a union of bundle neighbourhoods $N(E'_j)$ of curves $E'_j$ of $(\rho\circ\ell')^{-1}(0)$. 
 For each component $E_i$ of
  $\pi^{-1}(0)$ we can define the bundle neighbourhood $N(E_i)$ by
$$N(E_i) = \mu\bigl(\bigcup_{E_j' \subset \mu^{-1}(E_i) }N(E'_j)\bigr).$$

By construction, we have the following two correspondences between the
resolution graphs $R$ associated with $\rho^{-1}(0)$ and $\Gamma$
associated with $\pi^{-1}(0) \subset \widetilde{X}$.
\begin{enumerate}
\item If $R_0$ is a subgraph of $R$, let $\Gamma_{R_0}$ be the
  subgraph of $\Gamma$ which represents the irreducible components of
  $\mu( {\ell'}^{-1}( \bigcup_{v_k \in \R_0} C_k))$; if $R_0 \subset
  R$ is a maximal string between two nodes of $R$, then $\Gamma_{R_0}$
  is a union of strings in $\Gamma$;

\item For each node (resp.\ \P-node, resp.\ \L-node) $v_i$ of
  $\Gamma$, let $E'_j$ be the irreducible component of $\mu^{-1}(E_i)$
  which maps surjectively to $E_i$ and set $C_k = \ell'(E'_j)$. Then
  $v_k$ is a node (resp.\ $\Delta$-node, resp.\ the root $v_{1}$) of
  $R$.
\end{enumerate}
We then obtain that for each node (resp.\ \P-node, resp.\ \L-node)
$v_i$ of $\Gamma$, the set $\pi({{\cal N}(E_i)})$ is a connected
component of $(\ell\circ\pi)^{-1}({\cal N}(C_k))$ where $v_k$ is a node (resp.\
a \P-node, resp.\ the root $v_{1}$) of $R$, and Proposition
\ref{prop:resolution} is now a straightforward consequence of Lemma
\ref{lem:resolution}.}
\end{proof}

\red{Finally, we return to the geometric decomposition (Definition \ref{def:geometric decomposition}). Note that it differs from the ``inner geometric decomposition'' (Definition \ref{def:gdp}) only in that some annular pieces may have been decomposed by $A(q,q)$-pieces into sequences of annular pieces.

    \begin{proposition}\label{prop:resolution1}    There exists a resolution $\widehat \pi$ obtained from $\pi$ by blowing-up further, if necessary, intersection points of exceptional curves (so creating only new valency $2$-vertices)  such that  the $A(q,q)$-pieces of the geometric decomposition of $(X,0)$ (Definition \ref{def:geometric decomposition}) are equivalent to some pieces $\pi({\cal N}(v))$ where $v$ are valency $2$ nodes of the resolution graph $\widehat\Gamma$ of $\widehat \pi$ which are not $\cal P$- or $\cal L$-nodes. 
\end{proposition}
We give a more detailed description of the location of $A(q,q)$-pieces in the next section, so we omit the (easy) proof of this proposition.
}
\red{\section{The $A(q,q)$-pieces of the geometric decomposition.} \label{sec:A(q,q)}
One goal of this section is to provide tools which will be used in
Sections \ref{sec:detecting1} and \ref{sec:explicit computation} for
detecting the geometric decomposition and the carrousel using only the
outer geometry of $(X,0)$.

The following proposition, which is a consequence of Lemmas
\ref{lem:rates} and \ref{lem:X"_{q_i}}, describes which $B$-pieces of
type $A(q_j,q_j)$ of the geometric decomposition appear in annular
$A(q_i,q_k)$-pieces (with $q_i<q_j<q_k$) of the \emph{inner} geometric
decomposition. Note that for any such piece $A$ of the geometric
decomposition the map $\ell|_A\colon A\to\ell(A)$ is an unbranched
covering.
\begin{proposition}\label{prop:Anne doesn't like this}
An $A(q_j,q_j)$-piece occurs inside an  $A(q_i,q_k)$-piece $A$ of the  inner geometric decomposition  (Definition \ref{def:gdp}) with $q_i<q_j<q_k$ if and only if either:
\begin{enumerate}
\item\label{it:Aqq1} there is a non-annular $B(q_j)$-piece of the geometric
  decomposition with $\ell( B(q_j))\subset \ell(A)$ and the outer
  boundary of $\ell( B(q_j))$ is isotopic to the inner and outer
  boundaries of $\ell(A)$, or:
\item\label{it:Aqq2} there is an $A(q_{i'},q_{k'})$-piece $A'$ of the inner geometric decomposition with $q_{i'}<q_j<q_{k'}$ and
  \begin{enumerate}
  \item\label{one} the inner boundaries of $\ell(A)$ and $\ell(A')$ are inside
    $\ell(A')$ resp.\ $\ell(A)$;
  \item\label{two} $d(\partial_i(\ell(A))\cap
    S^3_\epsilon,\partial_i(\ell(A'))\cap
    S^3_\epsilon)=O(\epsilon^{q_j})$, where $d$ is outer distance,
    $\partial_i$ means inner boundary and $S^3_\epsilon$ is the sphere
    of radius $\epsilon$ about $0$ in $\C^2$.
  \end{enumerate}
\end{enumerate}
\end{proposition}
\noindent{\bf Remarks.} Note that the
  $A(q_j,q_j)$-piece inside an $A(q_i,q_k)$-piece with $q_i<q_j<q_k$
  is unique up to equivalence (Definition \ref{def:equivalent pieces}).
  In case \eqref{it:Aqq2}  by symmetry there is an
  $A(q_j,q_j)$-piece inside each of $A$ and $A'$; we have no example of this, although we expect that it can occur. 

\subsection*{Closeness and proximity classes} 
  For $(Y,0)$ a subgerm of $(X,0)$ or $(\C^2,0)$, we use the notation
  $N_{q,a}(Y)$ to mean the union of components of $\{x\in X:d(x,Y)\le
  a|x|^q\}$ which intersect $Y \setminus\{0\}$. Here again, $d$ refers to
  outer distance.

  For $i=2,\ldots,\nu$ let $X^{(i-1)}$ be the union of all $X_{q_j}$
  with $j\le i-1$ and all $A(q_j,q_k)$-pieces with $k<j\le i-1$
  connecting them.  
  \begin{remark}\label{rk:equivalent N_q} The sets 
 $N_{p,\beta}(X^{(i-1)})$  can be used to describe the components of $X_{q_i}$ with non-empty inner boundaries in the following way.

 First, for $1<p\le q_{i-1}$ we can describe $N_{p,\beta}(X^{(i-1)})$
 up to equivalence in terms of the carrousel as follows. Let
 $B^{(i-1)}$ be the union of all carrousel pieces which are
 $B(q_j)$-pieces with $j\le i-1$ and $A(q_j,q_k)$-pieces with $k<j\le
 i-1$. For $\beta >0$, let $\cal N^{(i-1)}_{p, \beta}$ be the union of
 the components of $\ell^{-1}(N_{p,\beta}(B^{(i-1)}))$ which contain
 components of $X^{(i-1)}$. Then for $p \not\in
 \{q_{i},\ldots,q_{\nu-1} \}$ and any $\beta, \beta'>0$,
 $N_{p,\beta}(X^{(i-1)})$ is equivalent to $\cal
 N^{(i-1)}_{p,\beta'}$.  Moreover, the topology of
 $N_{p,\beta}(X^{(i-1)})$ does not depend on $p$ if $q_j<p<q_{j-1}$,
 and it has the same topology also if $p=q_j$ and $\beta$ is sufficiently
 large or $p=q_{j-1}$ and $\beta$ is sufficiently small, and then again  $N_{p,\beta}(X^{(i-1)})$ is equivalent to $\cal
 N^{(i-1)}_{p,\beta'}$ if $\beta'$ is also sufficiently large resp.\ small.

  As a consequence, the components of $X_{q_i}$ with non-empty boundary are, up to equivalence,
    included among the components
    of $$\overline{N_{q_i,\beta}(X^{(i-1)})\setminus
      N_{q_i,\alpha}(X^{(i-1)})}$$
  with $\alpha$ sufficiently small and $\beta$ sufficiently large. 
    We will use this in subsection \ref{subbsecsec:Xq2 and Aq2q1}
\end{remark}

\begin{definition} Let $Y_1$ and $Y_2$ be two components of
  $N_{q,\beta}(X^{(i-1)})$ and let $q'$ be the rational number
  such that $d(Y_1 \cap S_{\epsilon}, Y_2\cap S_{\epsilon}) =
  O(\epsilon^{q'})$.  We say that $Y_1$ and $Y_2$ are \emph{close} if
  $q'\geq q$.  Closeness defines an equivalence relation on the
  components of $N_{q,\beta}(X^{(i-1)})$. We call the equivalence
  classes the \emph{proximity classes}.

  Note that the notions ``closeness'' and ``proximity classes'' are
  invariant under replacing pieces by equivalent pieces (Definition
  \ref{def:equivalent pieces}). Moreover, closeness can be measured
  just on outer boundaries, i.e., $d(Y_1 \cap S_{\epsilon}, Y_2\cap
  S_{\epsilon}) = O(\epsilon^{q'})$ if and only if $d(\partial_oY_1
  \cap S_{\epsilon}, \partial_oY_2\cap S_{\epsilon}) =
  O(\epsilon^{q'})$, where $\partial_o$ means outer boundary.
  Therefore the definition of closeness 
  extends to the components of $\overline{N_{q,\beta}(X^{(i-1)})\setminus
  N_{q,\alpha}(X^{(i-1)})}$ for $\alpha<\beta$, and depends on $q$ but not
    on $\alpha$ and $\beta$. 
\end{definition}

\begin{lemma} Let $q_i'$ be the infimum of $q<q_{i-1}$ for which the
  topology of $N_{q,1}(X^{i-1})$ has not changed. So we have
  $q'_i \leq q_{i}$ As $q$ with $q'_i <q \le q_{i-1}$ decreases, the
  number of proximity classes of $N_{q,1}(X^{(i-1)})$ is
  non-increasing, so the number of elements in each proximity class is
  non-decreasing.
\end{lemma}
\begin{proof}
 Let $q<q' \le q_{i-1}$. Let $Y_1$ and $Y_2$ be two components of
  $N_{q,1}(X^{(i-1)})$ and $Y'_j:=Y_j \cap N_{q',1}(X^{(i-1)})$, $j=1,2$,
the components of
  $N_{q',1}(X^{(i-1)})$ contained respectively in $Y_1$ and $Y_2$. If
  $Y'_1$ and $Y'_2$ are close, then $Y_1$ and $Y_2$ are close.
\end{proof}

\begin{definition}
Let $Y$ be a component of $X^{(i-1)}$ and $q<q_{i-1}$. We say that \emph{the cardinality of the proximity class of $Y$ grows at $q$} if for any $p$ with $q<p\leq q_{i-1}$,  the proximity class of the component of 
$N_{p,1}(X^{(i-1)})$ containing $Y$ has constant cardinality while for any $p<q$, it has cardinality strictly larger. 
\end{definition}

\begin{lemma} \label{le:outer boundary} For $1<p \leq q_{i-1}$,
  let $N_1$ and $N_2$ be two components of $\cal
    N^{(i-1)}_{p,\beta}$ (see Remark \ref{rk:equivalent N_q}) and
  $N'_1$ and $N'_2$ pieces equivalent to them (e.g., components
    of $\ell^{-1}(N_{p,\beta}(B^{(i-1)}))$). Then $\ell(N_1)$ and
  $\ell(N_2)$ have the same outer boundaries if and only if $N'_1$ and
  $N'_2$ are close.
\end{lemma}

\begin{proof} Let $q$ be such that the outer distance $d(N'_1 \cap
  S_{\epsilon},N'_2\cap S_{\epsilon})$ is $O(\epsilon^q)$. 
  
  Assume $N'_1$ and $N'_2$ are close, i.e., $q\geq p$.  Depending
    on $p$, choose $j$ and $q'$ with $q_j<q'<p\le q_{j-1}$
    and let $N''_1$ and $N''_2$
  be $N'_1$ and $N'_2$ union $A(q',p)$-pieces added on their outer
  boundaries. Since $\ell$ is a linear projection, $d(\ell(x),\ell(y))
  \leq d(x,y)$ for any $x,y \in \C^n$. Therefore $d(\ell(N'_1 \cap
  S_{\epsilon}),\ell(N'_2\cap S_{\epsilon}))$ is $O(\epsilon^{q''})$
  with $q\leq q''$. Assume $\ell(N_1)$ and $ \ell(N_2)$ do not have
  equal outer boundaries. Then $\ell(N''_1)$ and $\ell(N''_2)$ cannot
  have equal outer boundaries, and since $\ell(N''_1)$ and
  $\ell(N''_2)$ are subgerms in $\C^2$ with $A(q',p)$-pieces at their
  outer boundaries, we obtain that $d(\ell(N''_1 \cap
  S_{\epsilon}),\ell(N''_2\cap S_{\epsilon}))$ is $O(\epsilon^{q_0})$
  with $q_0\leq q'$.  Finally, $q''\leq q_0$, as $d(\ell(N''_1 \cap
  S_{\epsilon}),\ell(N''_2\cap S_{\epsilon})) \leq d(\ell(N'_1 \cap
  S_{\epsilon}),\ell(N'_2\cap S_{\epsilon}))$. Summarizing:
 $$q\leq q'' \leq q_0 \leq q' < p \le q\,,$$
which is a contradiction.
 
Conversely, suppose  $p > q$. Since the outer distance $d(N'_1 \cap
  S_{\epsilon},N'_2\cap S_{\epsilon})$ is $O(\epsilon^q)>\!\!>O(\epsilon^p)$, if 
  $\ell_{\cal D} \neq \ell$ is a generic plane projection for
$X$  then  $\ell_{\cal D}(N'_1) \cap
 \ell_{\cal D}(N'_2)$ is empty. Since  $N_j'$ is equivalent to  $N_j$ for $j=1,2$ the same holds for $\ell_{\cal D} (N_1)$ and $\ell_{\cal D} (N_2)$, so  $\ell(N_1)$ and $\ell(N_2)$ could not have equal outer boundaries.
\end{proof}

Let $\ell \colon (X,0) \to (\C^2,0)$ be a generic projection and let
$\Delta$ be its discriminant curve.  Recall that for all
$i=1,\ldots,\nu-1$, $q_i$ is one or more of: a polar rate, an essential Puiseux
exponent of $\Delta$ or a coincidence exponent between two branches of
$\Delta$.   Let $\Delta_0$ be a component of
$\Delta$ and $s_0$ its polar rate  and let $q_i\le s_0$ be a
carrousel rate which is
either equal to $s_0$ or is an essential Puiseux exponent of $\Delta_0$
or a coincidence exponent of $\Delta_0$ with another component of
$\Delta$. Let $B$ be the $B(q_i)$-piece of the carrousel such that the
$D(q_i)$-piece $D$ formed of a union of carrousel pieces and 
with common outer boundary with $B$, contains $\Delta_0$.  Now let $D'$
be the component of $\ell^{-1}(D)$ which contains the component of the
polar curve $\Pi_0$ over $\Delta_0$ and let $\widetilde B$ be the union of components of $X_{q_i}$ contained in $D'$.
\begin{definition}\label{def:associated component}
  $\widetilde B$ is called the \emph{bunch of components associated with $(\Delta_0,q_i)$}. Note that each outer boundary of $D'$ is an outer boundary of a component of $\widetilde B$ and by Lemma \ref{le:outer boundary} the components of $\widetilde B$ are close to each other.
\end{definition}
\begin{lemma} \label{lem:rates} Let $i \in \{1, \ldots,\nu-1 \}$.  Let $(\Delta_0,q_i)$ be a pair as above with $q_i\ne s_0$. Then the associated components of $\widetilde B$ have non-empty inner boundary. Moreover, 
\begin{enumerate}
\item\label{it:Bbar2} if $q_i$ is an essential Puiseux exponent of  $\Delta_0$, then some $B(q_i)$-piece in $\widetilde B$  is not an $A(q_i,q_i)$-piece;
\item\label{it:Bbar3} if $q_i$ is a coincidence exponent between two branches $\Delta_0$ and $\Delta_1$ of $\Delta$ which is strictly less than their polar rates $s_0$ and $s_1$, then the corresponding two components  $B_0$ and $B_1$ of $  X_{s_0}$ and $X_{s_1}$ satisfy $d(B_0 \cap S_{\epsilon}, B_1 \cap S_{\epsilon}) = O(\epsilon^{q_i})$; 
\end{enumerate}
Conversely, for all $j,k\in \{ 1 \ldots,\nu-1\}$, if  there exist  two components  $B_0$ and $B_1$  of $  X_{q_j}$ and $  X_{q_k}$ respectively such that $d(B_0 \cap S_{\epsilon}, B_1 \cap S_{\epsilon}) = O(\epsilon^{q})$ with $q < q_j, q_k$, then $q=q_i$ for some $i >j,k$, and $q_i$ is a coincidence coefficient between two components of $\Delta$ corresponding to the components $B_0$ and $B_1$.  
\end{lemma}

\begin{proof} The first part of the Lemma is obvious since any piece of the geometric decomposition with empty inner boundary is a polar piece.

Now assume $q_i$ is an essential Puiseux exponent of $\Delta_0$ and is
  not its polar rate.  
Let $D$, $B$ and $\widetilde B$ be as in the paragraph before Definition \ref{def:associated component} and let $D_0$ be the component of  $\overline{D \setminus B}$ which contains
  $\Delta_0$. 
 Denote by $F_t$ the set
  $\{z_1=t\} \cap X$ where $\ell=(z_1,z_2)$ and $z_1$ is a generic
  linear function for $X$. Since $q_i$ is an essential exponent, the
  section $\{z_1=t\} \cap D_0$ consists of several discs whose
  boundary components are inner boundaries of the section $B \cap
  \{z_1=t\}$. These boundaries lift by $\ell$ into the inner
  boundaries of some components of $\widetilde B$. Let $B'$ be such a component. Since $q_i$ is not a polar rate, no component of
  $\ell^{-1}(D_0)$ amalgamates with $B'$. So the section $F_t \cap B'$
  has several inner boundary components, and $B'$ cannot be an annular
  piece. This proves \eqref{it:Bbar2}.

  Assume now that $q_i$ is a coincidence exponent between two branches
  $\Delta_0$ and $\Delta_1$ of $\Delta$ which is strictly less than
  their polar rates $s_0$ and $s_1$. Let $\Pi_0$ and $\Pi_1$ be the
  two branches of the polar curve such that
  $\ell(\Pi_i)=\Delta_i$. According to \cite[page 462 Lemme 1.2.2
  ii)]{T3}, the restriction $\ell|_{\Pi} \colon \Pi \to \Delta$ is
  a generic plane projection of the curve $\Pi$, so $d(\Delta_0 \cap
  S_{\epsilon}, \Delta_1 \cap S_{\epsilon} ) = d(\Pi_0 \cap
  S_{\epsilon}, \Pi_1 \cap S_{\epsilon} ) = O(\epsilon^{q_i})$.  
Let
  $B_0$ (resp.\ $B_1$) be the $B(s_0)$-piece of $X_{s_0}$ (resp.\
  $B(s_1)$-piece of $X_{s_1}$) which contains the component $\Pi_0$
  (resp.\ $\Pi_1$) of the polar curve over $\Delta_0$ (resp.\
  $\Delta_1$). Since $q_i<s_0, s_1$, we have $d(B_0 \cap
  S_{\epsilon},B_1 \cap S_{\epsilon}) = d(\Pi_0 \cap S_{\epsilon},
  \Pi_1 \cap S_{\epsilon} ) = O(\epsilon^{q_i})$. This proves \eqref{it:Bbar3}.

  Conversely, assume there exist two components $B_0$ and $B_1$ of $
  X_{q_j}$ and $ X_{q_k}$ respectively such that $d(B_0 \cap
  S_{\epsilon}, B_1 \cap S_{\epsilon}) = O(\epsilon^{q})$ with $q <
  q_j, q_k$. Let $B'_1$ and $B'_2$ be the components of $X^{(j)}$
  resp.\ $X^{(k)}$ with empty inner boundaries and which have common
  outer boundaries with $B_0$ resp.\
    $B_1$. Let $\Pi_0$ and $\Pi_1$ be components of the polar curve
  $\Pi$ inside $B'_1$ resp.\ $B'_2$. Since $B'_1$ and $B'_2$ are
  unions of B and A-pieces with rates $\ge q$, then
  $d(\Pi_0 \cap S_{\epsilon}, \Pi_1 \cap S_{\epsilon} ) =
  O(\epsilon^q)$. Again by \cite[page 462 Lemme 1.2.2 ii)]{T3}, we
  have $d(\Delta_0 \cap S_{\epsilon}, \Delta_1 \cap S_{\epsilon} ) =
  O(\epsilon^q)$ where $\ell(\Pi_0)=\Delta_0$ and
  $\ell(\Pi_1)=\Delta_1$. Therefore $q$ is the coincidence coefficient
  between $\Delta_0$ and $\Delta_1$.
\end{proof}

\begin{lemma} \label{lem:X"_{q_i}} Let $i \in \{2,\ldots, \nu-1\}$.  Let $X''_{q_i}$
  consist of the components of $X_{q_i}$ with nonempty inner boundary.
  Then $X''_{q_i}$ is a union of proximity classes, each of which has
  at least one component of $X_{q_i}$ corresponding to a pair
  $(\Delta_0,q_i)$ as above.  More precisely, if $B$ is a
  $B(q_i)$-piece of the carrousel, the components of $X''_{q_i}$ whose
  images by $\ell$ have same outer boundary as $B$ are the components
  of a proximity class, and the components of a proximity class
  $\widehat{B}$ are in $X''_{q_i}$ if and only if either
  $\widehat{B}$ contains a $B(q_i)$-piece which is not an
  $A(q_i,q_i)$-piece or the cardinality of the proximity class of
  $X^{(i-1)} \cap \widehat{B}$ grows at  $q_i$.
\end{lemma}

\begin{proof}   $X''_{q_i}$ is non-empty if and only if there exists a component $\Delta_0 $ as above such that $q_i \neq s_0$. Let $B$ be a $B(q_i)$-piece of the carrousel  corresponding to such a pair $(\Delta_0,q_i)$.  By Lemma \ref{le:outer boundary}, any two components of $X_{q_i}$ projecting onto $B$ are close, so they belong to the same proximity class $\widehat{B}$, and any component of $\widehat{B}$ with non-empty inner boundary belongs to $X''_{q_i}$. Conversely, it is clear that any component of $X''_{q_i}$ is inside such a proximity class $\widehat{B}$.  

Moreover, according to Lemma \ref{lem:rates},  any such  $\widehat{B}$ contains a component which is not an
  $A(q_i,q_i)$-piece or the cardinality of the proximity class of
  $X^{(i-1)} \cap \widehat{B}$ grows at $q_i$ (this gives cases  \eqref{one} and \eqref{two} of Proposition \ref{prop:Anne doesn't like this}). 
\end{proof}
}

\section*{\bf Part 3: Analytic invariants from Lipschitz
  geometry}\label{Part 3}

\section{General hyperplane sections and maximal ideal  cycle}\label{sec:resolve hyperplane pencil}

In this section we prove parts \eqref{it:resolution of pencils} to
\eqref{it:hyperplane section} of Theorem \ref{th:invariants from
  geometry} in the Introduction.
We need some preliminary results about the relationship between
thick-thin decomposition and resolution graph and about the usual and
metric tangent cones. 
\red{The geometric decompositions  (see, e.g., Propositions \ref{prop:resolution} and \ref{prop:resolution1}), which are strong refinements of the thick-thin decomposition, are not needed in this section.}

\subsection*{Thick-thin decomposition and resolution graph}
According to \cite[Theorem 1.6]{BNP}, the thick-thin decomposition of
  (X,0) is determined up to homeomorphism close to the
  identity by the
  inner Lipschitz geometry, and hence by the outer Lipschitz geometry (specifically,  the
    homeomorphism $\phi\colon(X,0)\to(X,0)$ satisfies $d(p,\phi(p))\le
    |p|^q$ for some $q>1$). It is described
  through resolution as follows. The thick part $(Y,0)$ is the union
  of semi-algebraic sets which are in bijection with the \L-nodes of
  any resolution which resolves the basepoints of a general linear
  system of hyper\-plane sections. More precisely, let us consider the
  minimal good resolution $\pi' \colon (X',E) \to (X,0)$
  of this type and let $\Gamma$ be its
  resolution graph. Let $v_1,\ldots,v_r$ be the \L-nodes of
  $\Gamma$. For each $i=1,\ldots,r$ consider the subgraph $\Gamma_i$
  of $\Gamma$ consisting of the \L-node $v_i$ plus any attached
  bamboos (ignoring \P-nodes). Then the
  thick part is given by: $$(Y,0)=\bigcup_{i =1}^r (Y_i,0)\quad\text{where}\quad  Y_i=\pi'(N(\Gamma_i))\,.$$ Let
  $\Gamma'_1,\ldots,\Gamma'_s$ be the connected components of $\Gamma
  \setminus \bigcup_{i=1}^r \Gamma_i$. Then the thin part is:
 $$(Z,0) = \bigcup_{j=1}^s (Z_j,0)\quad\text{where}\quad  Z_j=\pi'(\cal N(\Gamma'_j))\,.$$
Recall that we denote by
$A^{(\epsilon)} := A \cap S_{\epsilon}$ (with $S_\epsilon=\partial
B_\epsilon$) the link of any semi-algebraic subgerm $(A,0) \subset
(X,0)$ for $\epsilon$ sufficiently small.
We call the links $Y_i^{(\epsilon)}$ and $Z_j^{(\epsilon)}$   \emph{thick and thin zones} respectively.

\subsection*{Inner geometry and \L-nodes} 
Since the graph $\Gamma_i$ is star-shaped, the thick zone $Y_i^{(\epsilon)}$
is a Seifert piece in a graph decomposition of the link
$X^{(\epsilon)}$. Therefore the inner Lipschitz geometry already tells us a lot
about the location of \L-nodes: for a thick zone 
$Y_i^{(\epsilon)}$ with unique Seifert fibration (i.e., not an
$S^1$-bundle over a disk or annulus) the corresponding \L-node is
determined in any negative definite plumbing graph for the pair
$(X^{(\epsilon)}, Y_i^{(\epsilon)})$.  However, a thick zone may be 
a solid torus ${D}^2\times {S}^1$ or toral annulus
$I\times {S}^1\times {S}^1$. Such a zone corresponds to a
vertex on a chain  in the resolution graph (i.e., a subgraph whose vertices have valency $1$ or $2$ and represent curves of genus $0$) and different vertices along
the chain correspond to topologically equivalent solid tori or
toral annuli in the link $X^{(\epsilon)}$. Thus, in general inner
Lipschitz geometry  is insufficient to determine the \L-nodes and we need to appeal
to the outer metric.
 
\subsection*{Tangent cones} We will use the Bernig-Lytchak
  map $\phi\colon \mathcal T_0(X)\to T_0(X)$ between the metric
  tangent cone $\mathcal T_0(X)$ and the usual tangent cone $T_0(X)$
  (\cite{BL}). We will need its description as given in \cite[Section
  9]{BNP}. 

  Consider the restriction $h=z_1 |_X\colon X\to \C$ of a generic linear 
  form. The map $h$ restricts to a fibration $\zeta_j\colon
  Z_j^{(\epsilon)}\to S^1_\epsilon$, and, as described in
  \cite[Theorem 1.7]{BNP}, the components of the fibers have inner
  diameter $o(\epsilon^q)$ for some $q>1$. If one scales the inner
  metric on $X^{(\epsilon)}$ by $\frac1\epsilon$ then in the
  Gromov-Hausdorff limit as $\epsilon\to 0$ the components of the
  fibers of each thin zone collapse to points, so each thin zone
  collapses to a circle. On the other hand, the rescaled thick zones
  are basically stable, except that their boundaries collapse to the
  circle limits of the rescaled thin zones. The result is the link
  $\mathcal T^{(1)}X$ of the so-called metric tangent cone $\mathcal
  T_0X$ (see \cite[Section 9]{BNP}), decomposed as $$ \mathcal
  T^{(1)}X={\lim_{\epsilon\to0}}^{GH}\frac1\epsilon X^{(\epsilon)}
  =\bigcup \mathcal T^{(1)}Y_i\,, $$ where $\mathcal
  T^{(1)}Y_i=\lim_{\epsilon\to0}^{GH}\frac1\epsilon Y_i^{(\epsilon)}$,
  and these are glued along circles.

One can also consider $\frac
  1\epsilon X^{(\epsilon)}$ as a subset of the unit  $(2n-1)$-sphere 
 $S_1=\partial B_1 \subset
  \C^n$ and form the Hausdorff limit in $S_1$ to get the link
  $T^{(1)}X$ of the usual tangent cone $T_0X$ (this is the same as
  taking the Gromov-Hausdorff limit for the outer metric).  One thus
  sees a natural branched covering map $\mathcal T^{(1)}X \to
  T^{(1)}X$ which extends to a map of cones $\phi\colon
  \mathcal T_0(X) \to T_0(X)$ (first described in \cite {BL}).  

We denote by $T^{(1)}Y_i$ the piece of $T^{(1)}X$ corresponding to
$Y_i$ (but note that two different $Y_i$'s can have the same
  $T^{(1)}Y_i$). 

\begin{proof}[Proof of \eqref{it:resolution of pencils} of Theorem
  \ref{th:invariants from geometry}] 
Let  $L_j$ be the tangent line to $Z_j$ at $0$ and $h_j$ the map $h_j:=z_1|_{\partial
  (L_j\cap B_\epsilon)}\colon\partial (L_j\cap
B_\epsilon)\buildrel\cong\over\to  S^1_\epsilon$.  We
  can rescale the fibration $h_j^{-1}\circ \zeta_j$ to a fibration $\zeta'_j\colon
  \frac1\epsilon Z_j^{(\epsilon)}\to \partial (L_j\cap B_1)$, and
  written in this form $\zeta'_j$ moves points distance
  $o(\epsilon^{q-1})$, so the fibers of 
  $\zeta'_j$ are shrinking at this rate.  In particular, once
  $\epsilon$ is sufficiently small  the outer Lipschitz geometry of
  $\frac1\epsilon Z_j^{(\epsilon)}$ determines this fibration up to
  homotopy, and hence also 
  up to isotopy, since homotopic fibrations of a $3$-manifold to $S^1$
  are isotopic (see e.g., \cite[p.~34]{EN}).

  Consider now a rescaled thick piece $M_i=\frac 1\epsilon
  Y_i^{(\epsilon)}$.  The intersection  $K_i\subset M_i$ of $M_i$
  with the rescaled link of the curve $\{h=0\} \subset (X,0)$ 
  is a union of fibers of the Seifert fibration of $M_i$. The
  intersection of a Milnor fiber of $h$ with $M_i$ gives a homology
  between $K_i$ and the union of the  curves along which a Milnor
  fiber meets $\partial M_i$, and by the previous paragraph these
  curves are discernible from the outer Lipschitz geometry, so the homology
  class of $K_i$ in $M_i$ is known. It follows that the number of
  components of $K_i$ is known and $K_i$ is therefore known up to
  isotopy, at least in the case that $M_i$ has unique Seifert
  fibration. If $M_i$ is a toral annulus the argument still works, but
  if $M_i$ is a solid torus we need a little more care.

  If $M_i$ is a solid torus it corresponds to an \L-node which is a
  vertex of a bamboo. If it is the extremity of this bamboo then the map
$\mathcal T^{(1)}Y_i\to T^{(1)}Y_i$ is a covering. Otherwise
  it is a covering branched along its central circle. Both
  the branching degree $p_i$ and the degree $d_i$ of the map
$\mathcal T^{(1)}Y_i\to T^{(1)}Y_i$ are determined by the
  Lipschitz geometry, so we can compute $d_i/p_i$, which is the number
  of times the Milnor fiber meets the central curve of the solid torus
  $M_i$. A tubular neighbourhood of this curve meets the Milnor fiber
  in $d_i/p_i$ disks, and removing it gives us a toral annulus for
  which we know the intersection of the Milnor fibers with its
  boundary, so we find the topology of $K_i\subset M_i$ as before.

  We have thus shown that the Lipschitz geometry determines the
  topology of the link $\bigcup K_i$ of the strict transform of $h$ in the
  link $X^{(\epsilon)}$.  Denote $K'_i=K_i$ unless $(M_i,K_i)$ is a
  knot in a solid torus, i.e., $K_i$ is connected and $M_i$ a
  solid torus, in which case put $K'_i=2K_i$ (two parallel copies of
  $K_i$). The resolution graph we are seeking represents a minimal
  negative definite plumbing graph for the pair $(X^{(\epsilon)},
  \bigcup K'_i)$, for $(X,0)$. By \cite{neumann81} such a plumbing
  graph is uniquely determined by the topology. When decorated with
  arrows for the $K_i$ only, rather than the $K'_i$, it gives the
  desired decorated resolution graph $\Gamma$. So $\Gamma$ is
  determined by $(X,0)$ and its Lipschitz geometry.
\end{proof}

\begin{proof}[Proof of    \eqref{it:multiplicity and Zmax} of Theorem
  \ref{th:invariants from geometry}]
  Recall that $\pi'\colon  X'\to X$ denotes the minimal good
  resolution of $(X,0)$ which resolves a general linear system of
  hyperplane sections.  Denote by $h^*$ the strict transform by $\pi'$
  of the generic linear form $h$ and let $\bigcup_{k=1}^d E_k$ the
  decomposition of the exceptional divisor $\pi'^{-1}(0)$ into its
  irreducible components. By point \eqref{it:resolution of pencils} of
  the theorem, the Lipschitz geometry of $(X,0)$ determines the
  resolution graph of $\pi'$ and also determines $h^*\cdot
  E_k$ for each $k$. We therefore recover the total transform
  $(h):=\sum_{k=1}^d m_k(h)E_k+h^*$ of $h$ (since $E_l\cdot (h)=0$ for
  all $l=1,\ldots,d$ and the intersection matrix $(E_k\cdot E_l)$ is
  negative definite).

  In particular, the maximal ideal cycle $\sum_{k=1}^d m_k(h)E_k$ is
  determined by the geometry, and the multiplicity of $(X,0)$ also, since it is
  given by the intersection number $\sum_{k=1}^d m_k(h)E_k\cdot h^*$.
\end{proof}

\begin{proof}[Proof of \eqref{it:hyperplane section} of Theorem
  \ref{th:invariants from geometry}] Let $H$ be the  kernel
  of a  linear form $\C^n \to \C$ which is generic for $X$ and $K=X\cap H$ the hyperplane
  section. So $$K\subset \bigcup_{i=1}^r\{\pi'(\cal N(E_i)):E_i \text{
    an \L-curve}\}\,.$$ For $i=1,\dots,r$ the tangent cone of
  $K\cap \pi'(\cal N(E_{i}))$ at $0$ is a union of lines, say
  $L_{i1}\cup\dots\cup L_{i\alpha_i}$. For $\nu=1,\dots,\alpha_i$
  denote by $K_{i\nu}$ the union of components of $K$ which are
  tangent to $L_{i\nu}$. Note the curves $E_i$, $i=1,\dots,r$ are the
  result of blowing up the maximal ideal $\mathfrak m_{X,0}$ and
  normalizing. If two of the curves $E_i$ and $E_j$ coincided before
  normalization, then we will have equalities of the form
  $L_{i\nu}=L_{j\nu'}$ with $i\ne j$ (whence also
  $K_{i\nu}=K_{j\nu'}$).  

For fixed $i$ and varying $\nu$ the curves
  $K_{i\nu}$ are in an equisingular family and hence have the same
  outer Lipschitz geometry.
  We must show that we can recover the outer Lipschitz geometry of
  $K_{i\nu}$ from the outer Lipschitz geometry of $X$.  The argument
  is similar to that of the proof of Proposition \ref {prop:plane
    curve}.
  We consider a continuous arc $\gamma_i$ inside
$\pi'(\cal N(E_i))$ 
with the property
  that $d(0,\gamma(t))=O(t)$.  Then for all $k$ sufficiently small, the
  intersection $X \cap B(\gamma_i(t),kt)$ consists of some number $\mu_i$
  of $4$-balls $D^4_1(t),\ldots,D^4_{\mu_i}(t)$ with
  $d(D^4_j(t),D^4_k(t)) = O(t^{q_{jk}})$ before we change the metric
  by a bilipschitz homeomorphism.  The $q_{jk}$ are determined by the
  outer Lipschitz geometry of $(X,0)$ and are still determined after a
  bilipschitz change of the metric by the same argument as the last
  part of the proof of \ref {prop:plane curve}. Moreover, they
  determine the outer Lipschitz geometry of $K_{i\nu}$, as in Section
  \ref{sec:curve}. They  also determine the number $m_i$ of components of
  $K_{i\nu}$ which are in $\pi'(\mathcal N(E_i))$. 

By the above proof
  of \eqref{it:multiplicity and Zmax} of Theorem \ref{th:invariants
    from geometry}, the number $\alpha_i$ of $K_{i\nu}$'s for fixed $i$ is
  $(E_i\cdot h^*)/m_i$, and hence determined by the outer Lipschitz
  geometry of $X$. Since different $K_{i\nu}$'s have
  different tangent lines, the outer geometry of their union is
  determined, completing the proof.
  \end{proof}

\section{Detecting the decomposition } \label{sec:detecting1}

The aim of this section is to prove that the geometric decomposition
$(X,0)=\bigcup_{i=1}^\nu (X_{q_i},0)\cup\bigcup_{i>j}(A_{q_i,q_j},0)$
introduced in Section \ref{sec:carrousel1} can be recovered up to
equivalence (Definition \ref{def:equivalent pieces}) using the outer
Lipschitz geometry of $X$. Specifically, we will prove:
\begin{proposition}\label{prop:recover Xq}
  The outer Lipschitz geometry of $(X,0)$ determines a decomposition
  $(X,0)=\bigcup_{i=1}^\nu(X'_{q_i},0)\cup\bigcup_{i>j}(A'_{q_i,q_j},0)$
  into semi-algebraic subgerms $(X'_{q_i},0)$ and $(A'_{q_i,q_j},0)$ glued
  along their boundaries, where each $\overline{X'_{q_i}\setminus
    X_{q_i}}$  and  $\overline {A'_{q_i,q_j}\setminus A_{q_i,q_j}}$  is a union of collars of type $A(q_i,q_i)$.

  So $X'_{q_i}$ is obtained from $X_{q_i}$ by adding an $A(q_i,q_i)$
  collar on each outer boundary component of $X_{q_i}$ and removing one at each inner boundary component, while
  $A'_{q_i,q_j}$ is obtained from $A_{q_i,q_j}$ by adding an
  $A(q_i,q_i)$ collar on each outer boundary component  of
  $A_{q_i,q_j}$ and removing an
  $A(q_j,q_j)$ collar at each inner boundary component.
\end{proposition}

An important part of the proof of Proposition \ref{prop:recover Xq} consists in discovering the
polar pieces in the germ $(X,0)$ by exploring them with small balls. Let us first introduce notation and  sketch this method.  

We will use the coordinates and the family of Milnor balls
$B_\epsilon$, $0<\epsilon\le \epsilon_0$ introduced in Section
\ref{sec:carrousel1}.  We always work inside the Milnor ball
$B_{\epsilon_0}$ and we reduce $\epsilon_0$ when necessary.

Let $q_1>q_2>\dots>q_\nu=1$ be the series of
rates for the geometric decomposition of $X$. In this section we will
assume that $\nu>1$ so $q_1>1$.

For $x\in X$ define $\mathscr B(x,r)$ to be the component containing $x$
of the intersection $B_r(x)\cap X$, where $B_r(x)$ is the ball of
radius $r$ about $x$ in $\C^n$.

\begin{definition}
  For a  subset $\cal B$ of $X$, define the \emph{abnormality
    $\alpha(\cal B)$ of $\cal B$} to be maximum ratio of inner to
  outer distance in $X$ for pairs of distinct points in
  $\cal B$.
\end{definition}

The idea of abnormality was already used by Henry and Parusi\'nski to
prove the existence of moduli in bilipschitz equivalence of map germs
(see \cite[Section 2]{HP}).

\begin{definition}
  We say a subset of $X\setminus\{0\}$ has {\it trivial topology} if
  it is contained in a topological ball which is a subset of
  $X\setminus\{0\}$. Otherwise, we say the subset has {\it essential
    topology}.
\end{definition}

\subsection{Sketch} Our method is based on the fact that a non-polar
piece is ``asymptotically flat" so abnormality of small balls inside
it is close to $1$, while a polar piece is ``asymptotically
curved'' (Lemmas \ref{le:bubbles} and \ref{le:parallel}). It consists
in detecting the polar pieces using sets $\mathscr B(x,a|x|^q)$ as
``searchlights".  Figure \ref{fig:polarzone} represents schematically
a polar piece (in gray).
\begin{figure}[ht]
\centering
\begin{tikzpicture} 
\node at(5.1,0)[right]{$0$};
 \path[ fill=lightgray] (3.27,2.95) .. controls (3.5,3.1) and
 (4.25,3.8) .. (4,4) -- (4,4) .. controls (3.75,4.24) and (2.9,3.8)
 .. (2.74,3.67) -- (2.74,3.67) .. controls ( 3, 3.4) and (3.2 ,3.1 )
 .. (3.27,2.95); \draw[line width=0.5pt ] (5,0) .. controls (4.3,1.5)
 and (3.6,3.5) .. (3.4,4.5);
\node(a)at(3.2,4.7){   $\Pi_{{\cal D}}$};
\node(a)at(4.34,4.9){   $\Pi_{{\cal D'}}$};
 \path[ fill=darkgray] (5,0) .. controls (4.3,1.5) and (4,2.4)
 .. (4.05,3.8) -- (4.05,3.8) .. controls (3.9,3.5) and (3.6,3.3) ..
 (3.4,3.1) -- (3.4,3.1).. controls (3.8,2.6) and (4.4,1.3) .. (5,0);
 \draw[line width=0.5pt ] (5,0) .. controls (4.6,1) and (2,1.4)
 .. (0,1);
\draw[line width=0.5pt ] (5,0) .. controls (4.3,1.5) and (4,2.4)
.. (4.05,4) -- (4.05,4) -- (4.1,5);
\draw[line width=0.5pt ] (4,4) .. controls (4.5,3.5) and (0.6,1)
.. (0,1);
\draw[line width=0.5pt ] (4,4) .. controls (3.5,4.5) and (0.4,2.1)
.. (-0.2,2);
\draw[dotted, line width=0.5pt ] (5,0) .. controls (4.5,1) and
(2.6,1.7) .. (1.8,1.9);
\draw[ line width=0.5pt ] (1.8,1.9) .. controls (1,2.1) and (0.5,2.1)
.. (-0.2,2);
\draw[line width=0.5pt ] (5,0) .. controls (4.4,1.3) and (3.8,2.6)
.. (3.4,3.1);
\draw[dotted, line width=0.5pt ] (5,0) .. controls (4.5,1.1) and
(3.4,2.7) .. (3.28,2.95);
\draw[ line width=0.5pt ] (3.28,2.95) .. controls (3.1,3.2) and
(3.25,3.1) .. (2.7,3.7);
\end{tikzpicture} 
\caption{A polar piece}
\label{fig:polarzone}
\end{figure}
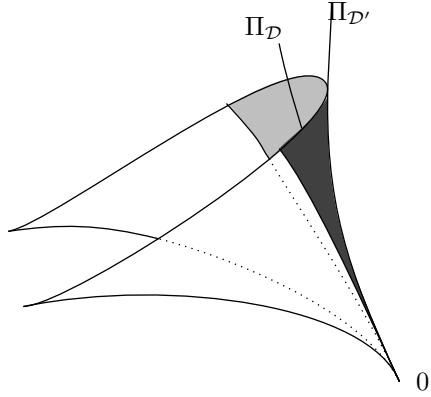

In order to get a first flavour of what will happen later, we will now
visualize some pieces and some sets $\mathscr B(x,a|x|^q)$ by drawing
their real slices.  We call \emph{real slice} of a real algebraic set
$Z \subset \C^n$ the intersection of $Z$ with $\{z_1=t\} \cap P$
where \red{$z_1$} is a generic linear form and $P$ a general $(2n-1)$-plane in
$\C^n \cong \Bbb R^{2n}$.  We use again the notations of the
previous sections: $\ell =(z_1,z_2)$ is a generic plane projection for
$(X,0)$ and $h=z_1|_X$.

We first consider a component $M
  \subset X_{q_i}$ of the geometric decomposition of $(X,0)$ which is
not a polar piece.  We assume that $q_i>1$ (so $i \in
  \{1,\ldots,\nu-1\}$) and that the sheets of the cover $\ell |_X \colon X \to
\C^2$ inside this zone have pairwise contact $ >q_i$, i.e., there
exists $q'_i >q_i$ such that for an arc $\gamma_0$ in $M$, all the
arcs $\gamma_0' \neq \gamma_0$ in $\ell^{-1}(\ell(\gamma_0))$ have
contact at most $q'_i$ with $\gamma_0$ and the contact $q'_i$ is reached
for at least \red{two of these arcs in $M$}.
 
In Figure \ref{fig:horns} the dotted circles represent the boundaries of
the real slices of balls $ B(x,a|x|^q) \subset \C^n$ for some $x\in
M$. The real slices of the corresponding sets $\mathscr B(x,a|x|^q) \subset X$
are the thickened arcs.
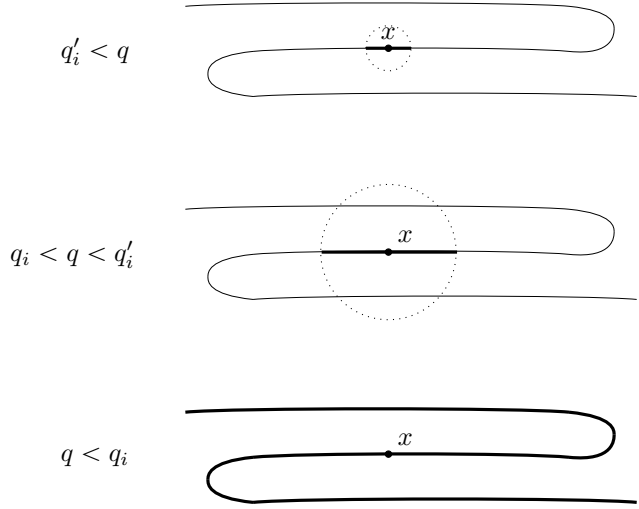
\begin{figure}[ht]
\centering
\begin{tikzpicture} 
\begin{scope}[scale=0.6]
        \node at(-6.5,0)[ ]{$q'_i<q$};
       
       \draw[line width=0.1pt] (-4.5,1).. controls (-3.5,1.1) and (3,1.1)..(4,1);   
     \draw[line width=0.1pt] (5,0.5).. controls (5,0.6) and (5,0.9)..(4,1);   
       \draw[line width=0.1pt] (5,0.5).. controls (5,0.4) and (5,-0.1)..(4,0);   
        \draw[line width=0.1pt] (4,0).. controls (3,0.1) and (-2,0.1)..(-3,0);   
          \draw[line width=0.1pt] (-4,-0.5).. controls (-4,-0.6) and (-4,-0.9)..(-3,-1);   
       \draw[line width=0.1pt] (-4,-0.5).. controls (-4,-0.4) and (-4,-0.1)..(-3,0); 
    \draw[line width=0.1pt] (5.5,-1).. controls (4.5,-0.9) and (-2,-0.9)..(-3,-1);   
  \draw[ dotted]  (0,0.07)circle(0.5);
   \draw[fill]  (0,0.07)circle(2pt);
    \node at(0,0.4){$x$};
    \draw[line width=1.2pt] (-0.50,0.07)--(0.50,0.07);
   
\begin{scope}[yshift=-4.5cm]
     \node at(-7,0)[ ]{$q_i<q<q'_i$};
     \draw[line width=0.1pt] (-4.5,1).. controls (-3.5,1.1) and (3,1.1)..(4,1);   
     \draw[line width=0.1pt] (5,0.5).. controls (5,0.6) and (5,0.9)..(4,1);   
       \draw[line width=0.1pt] (5,0.5).. controls (5,0.4) and (5,-0.1)..(4,0);   
        \draw[line width=0.1pt] (4,0).. controls (3,0.1) and (-2,0.1)..(-3,0);   
          \draw[line width=0.1pt] (-4,-0.5).. controls (-4,-0.6) and (-4,-0.9)..(-3,-1);   
       \draw[line width=0.1pt] (-4,-0.5).. controls (-4,-0.4) and (-4,-0.1)..(-3,0); 
    \draw[line width=0.1pt] (5.5,-1).. controls (4.5,-0.9) and (-2,-0.9)..(-3,-1);   
  \draw[ dotted]  (0,0.05)circle(1.5);
   \draw[fill]  (0,0.05)circle(2pt);
    \node at(0,0.4)[right]{$x$};
    \draw[line width=1.2pt] (-1.48,0.05)--(1.52,0.05);
     \end{scope} 
    
\begin{scope}[yshift=-9cm]       
  \node at(-6.5,0)[ ]{$q<q_i$};
      \draw[line width=1.2pt] (-4.5,1).. controls (-3.5,1.1) and (3,1.1)..(4,1);   
     \draw[line width=1.2pt] (5,0.5).. controls (5,0.6) and (5,0.9)..(4,1);   
       \draw[line width=1.2pt] (5,0.5).. controls (5,0.4) and (5,-0.1)..(4,0);   
        \draw[line width=1.2pt] (4,0).. controls (3,0.1) and (-2,0.1)..(-3,0);   
          \draw[line width=1.2pt] (-4,-0.5).. controls (-4,-0.6) and (-4,-0.9)..(-3,-1);   
       \draw[line width=1.2pt] (-4,-0.5).. controls (-4,-0.4) and (-4,-0.1)..(-3,0); 
    \draw[line width=1.2pt] (5.5,-1).. controls (4.5,-0.9) and (-2,-0.9)..(-3,-1);   
   \draw[fill]  (0,0.07)circle(2pt);
    \node at(0,0.4)[right]{$x$};            
     \end{scope} 
   \end{scope}
  \end{tikzpicture}  
    \caption{Balls through a non-polar piece}
  \label{fig:horns}
\end{figure}

Next, we consider a polar piece $N\subseteq X_{q_i}$. If $x\in N$ and
$a$ is large enough then $\mathscr B(x,a|x|^{q_i})$ either has
essential topology or high abnormality (or both).  Figure \ref{fig:horns2}
represents the real slice of a ball $B(x,a|x|^{q_i})$ with $x\in N$
when $\mathscr B(x,a|x|^{q_i})$ has high abnormality. It includes the real slice of $N$ which is the gray arc.

  \begin{figure}[ht]
\centering
\begin{tikzpicture} 
\begin{scope}[scale=0.6]
   \draw[line width=0.1pt] (-4.5,1).. controls (-3.5,1.1) and (3,1.1)..(4,1);   
     \draw[line width=0.1pt] (5,0.5).. controls (5,0.6) and (5,0.9)..(4,1);   
       \draw[line width=0.1pt] (5,0.5).. controls (5,0.4) and (5,-0.1)..(4,0);   
        \draw[line width=0.1pt] (4,0).. controls (3,0.1) and (-2,0.1)..(-3,0);           
          \draw[line width=1.2pt] (-1.4,0.05)--(-3,0);            
          \draw[line width=1.2pt] (-1.55,-0.95)--(-3,-1);           
          \draw[gray,line width=2pt] (-4,-0.5).. controls (-4,-0.6) and (-4,-0.9)..(-3,-1);   
       \draw[gray,line width=2pt] (-4,-0.5).. controls (-4,-0.4) and (-4,-0.1)..(-3,0); 
    \draw[line width=0.1pt] (5.5,-1).. controls (4.5,-0.9) and (-2,-0.9)..(-3,-1);   
  \draw[ dotted]  (182.5:3.7)circle(2.3);
   \draw[fill=white]  (182.5:3.7)circle(2pt);
    \node at(177:3.9){$x$};
    \node at(190:4.5){$N$};
 \end{scope} 
 \end{tikzpicture}  
  \caption{ }
  \label{fig:horns2}
\end{figure}

The rest of the section is organized as follows. The aim of
Subsections \ref{subsec:q1} and \ref{subsec:X'q1} is to specify the set
$(X'_{q_1},0)$ of Proposition \ref{prop:recover Xq}. This will be the first
step of an induction carried out in Subsections \ref{subsec:q2} to
\ref{subsec:X'qi} leading to the construction of subgerms
$(X'_{q_i},0)$ and $(A'_{q_i,q_j},0) \subset (X,0)$ having the
properties stated in Proposition \ref{prop:recover Xq}.  Finally, the
bilipschitz invariance of the germs $(X'_{q_i},0)$ and $A'_{q_i,q_j}$ is proved in
Proposition \ref{prop:invariance}, completing the proof of Proposition
\ref{prop:recover Xq}.

\subsection{Finding the rate $\boldsymbol{q_1}$} \label{subsec:q1}

\begin{lemma}\label{le:ball projection} 
  Let $\ell\colon X\to \C^2$ be a generic plane projection, and
  $A\subset X$ a union of polar wedges about the
  components of the polar curve for $\ell$.  If
  $x\in X$ and $r<|x|$ are such that $\mathscr B(x,r)\cap
  A=\emptyset$, then there exists $K>1$ such that $\ell$ restricted to
  $\mathscr B(x,\frac r{K^2})$ is an \red{inner} $K$-bilipschitz homeomorphism onto its
  image.  Moreover, $\mathscr B(x,\frac r{K^2})$ has trivial topology and
  abnormality $\le K$.
\end{lemma}
\begin{proof}There exists $K$ such that $\ell|_{\overline{X\setminus A}}$ is
  locally a $K$-bilipschitz map (see Section
  \ref{sec:polarwedges}). Let $B(y,r)\subset \C^2$ denote the ball of
  radius $r$ about $y\in \C^2$.  Then $ B(\ell(x), \frac
  rK)\subseteq\ell(\mathscr B(x,r))$. Let $\cal B$ be the component of
  $\ell^{-1}(B(\ell(x), \frac r{K}))$ which is contained in $\mathscr
  B(x,r)$. So $\ell|_{\cal B}$ is a $K$-bilipschitz homeomorphism of
  $\cal B$ onto its image $B(\ell(x), \frac r{K})$.  Now $\mathscr B(x, \frac
  r{K^2})\subseteq \cal B$, so $\mathscr B(x, \frac r{K^2})$ has trivial
  topology. Moreover, for each pair of points $x',x'' \in \mathscr B(x,
  \frac r{K^2})$, we have:
$$d_{inner}(x',x'') \leq K d_{\C^2}(\ell(x'),\ell(x''))\,.$$
On the other hand, $ d_{\C^2}(\ell(x'),\ell(x''))\leq d_{outer}(x',x'') $, so
$$ d_{inner}(x',x'')/d_{outer}(x',x'') \leq K\,.$$ 
Thus $\alpha( \mathscr B(x,\frac r{K^2}))\le K$.  
\end{proof}

\begin{lemma}\label{le:bubbles}
\begin{enumerate}
\item\label{it:bubbles-1} There exist $a> 0$, $K_1>1$ and $\epsilon_0>0$ such
  that for all $x\in X\cap B_{\epsilon_0}$ the set $\mathscr B(x,a|x|^{q_1})$
  has trivial topology and abnormality at most $K_1$.
\item\label{it:bubbles-2}  For all $q>q_1$ and $a>0$ there exist $K_1>1$ and $\epsilon_1$ with $\epsilon_1\le \epsilon_0$ such that   for all $x\in X\cap B_{\epsilon_1}$ the set $\mathscr B(x,a|x|^{q})$
  has trivial topology and abnormality at most $K_1$.
\item\label{it:bubbles-3} For all $q<q_1$, $K_2>1$ and $a>0$
  there exist $\epsilon_1$ with $\epsilon_1\le\epsilon_0$ such that for all $x\in X_{q_1}
  \cap B_{\epsilon_1}$ we have either $\alpha(\mathscr B(x,a|x|^{q}))>
  K_2$ or $\mathscr B(x,a|x|^{q})$ has essential topology.
\end{enumerate}
\end{lemma}
Before proving the lemma we note the following immediate corollary:
\begin{proposition}[Finding $q_1$]\label{prop:finding q1}
  $q_1$ is the infimum of all $q$ satisfying:  there exists $K_1>0$
  such that for all $a>0$ there exists $\epsilon_1>0$ such that all
  sets $\mathscr B(x,a|x|^q)$ with $x \in X\cap B_{\epsilon_1}$ have trivial
  topology and abnormality at most $K_1$.\qed
\end{proposition}
\begin{proof}[Proof of Lemma \ref{le:bubbles} \eqref{it:bubbles-1} and \eqref{it:bubbles-2}] We just prove \eqref{it:bubbles-1}, since \eqref{it:bubbles-2} follows immediately from  \eqref{it:bubbles-1}, using that $\mathscr B(x,a|x|^q)=\mathscr B(x,a'|x|^{q_1})$ with $a'=a|x|^{q-q_1}$, which can be made as small as one wants by reducing $\epsilon_1$.

 Let $\ell_{\cal D_1} \colon(X,0) \to (\C^2,0)$ and $\ell_{\cal D_2}
  \colon(X,0) \to (\C^2,0)$ be a generic pair of generic plane
  projections. For $i=1,2$ let $A_i$ be a union of polar wedges about
  the components of the polar curve $\Pi_{\cal D_i}$.

  Note that a polar wedge of rate $s$ about a component $\Pi_0$ of the
  polar curve is, up to equivalence, the component containing $\Pi_0$ of a
  set of the form $\bigcup_{x\in\Pi_0}\mathscr B(x,c|x|^s)$, so we may
  take our polar wedges to be of this form, and we may choose $c$ and
  $\epsilon_0$ small enough that none of the polar wedges in $A_1$ and
  $A_2$ intersect each other in $B_{2\epsilon_0}\setminus\{0\}$. Now
  replace $c$ by $c/4$ in the construction of the polar wedges. Since
  $q_1\ge s$ for every polar rate $s$, for any $x\in
  (X\cap B_{\epsilon_0})\setminus\{0\}$ the set $\mathscr B(x,c |x|^{q_1})$
  is then disjoint from one of $A_1$ and $A_2$. Thus  part
  \eqref{it:bubbles-1} of the lemma follows from Lemma \ref{le:ball
    projection}, with $K_1$ chosen such that each restriction
  $\ell_{D_i}$ to $(X \setminus A_i) \cap B_{2\epsilon}$ is a local
  $K_1$-bilipschitz homeomorphism, and with $a=\frac c{K_1^2}$. 
\end{proof}

\begin{proof}[Proof of Lemma \ref{le:bubbles} \eqref{it:bubbles-3}]
  Fix $q<q_1$, $K_2>1$ and $a>0$. We will deal with $X_{q_1}$
  component by component, so we assume for simplicity that
  $X_{q_1}$ consists of a single component (see Definition \ref{def:component}).

  Assume first that $X_{q_1}$ is not a $D(q_1)$-piece. Then it is a
  $B$-piece and there exist $\epsilon\le \epsilon_0$ and $K>1$ such that $X_{q_1}\cap
  B_{\epsilon}$ is $K$-bilipschitz equivalent to a standard model as
  in Definition \ref{def:p1}. Let $x \in X_{q_1} \cap
  B_{\epsilon}$. If $d$ is the diameter of the fiber $F$ used in constructing the model
  then $\mathscr B(x,2Kd|x|^{q_1})$ contains a shrunk copy of the fiber $F$. 
  Then, if $\epsilon$ is small enough, $\mathscr B(x,a |x|^{q})$ also
  contains a copy of $F$, since $q<q_1$.  In \cite{BNP} it is proved
  that $F$ contains closed curves which are not null-homotopic in
  $X\setminus\{0\}$, so $\mathscr B(x,a|x|^{q})$ has essential topology.
  
  We now assume
  $X_{q_1}$ is a $D(q_1)$-piece.
  Consider the resolution $\pi \colon (\widetilde{X},E) \to (X,0)$
  with dual graph $\Gamma$ defined in Section
  \ref{sec:decomposition vs resolution}. We will use again the notations of Section
  \ref{sec:decomposition vs resolution}.  For each irreducible component of $E$, let $N(E_i)$ be a small
  closed tubular neighbourhood of $E_i$. Let \red{$\Gamma_{q_1}$} be the
  subgraph of $\Gamma$ which consists of the \P-node corresponding to $X_{q_1}$  plus attached bamboos. After adjusting $N(E_i)$ if necessary, one
  can assume that the strict transform $\overline{\pi^{-1}(X_{q_1}
    \setminus \{0\})}$ of $X_{q_1}$ by $\pi$ is the set $\cal N(\Gamma_{q_1})$
  (see Proposition \ref{prop:resolution}). 
  
We set $E'=\bigcup_{v_i \in \Gamma_{q_i}} E_i$ and  $E''=\overline{E \setminus E'}$. 

Let $\widehat{\alpha}_q \colon (X \setminus \{0\}) \times (0,\infty)
\to [1, \infty)$ be the map which sends $(x,a)$ to the abnormality
$\alpha(\mathscr B(x,a|x|^{q}))$ and let $\widetilde{\alpha}_q \colon
(\widetilde{X} \setminus E) \times (0,\infty) \to [1, \infty) $ be
the lifting of $\widehat{\alpha}_q $ by $\pi$.
The intersection $E \cap \overline{\pi^{-1}(X_{q_1} \setminus \{0\})}$
is a compact set inside $E' \setminus E''$, so to prove Lemma
\ref{le:bubbles} \eqref{it:bubbles-3}, it suffices to prove that for
each $y\in E' \setminus E''$,
$$\lim_{x \to y \atop x \in \widetilde{X}\setminus E} \widetilde{\alpha}_q(x,a) = \infty\,.$$

 This follows from the following more general Lemma, which will be used again later. 
\end{proof}

\begin{lemma} \label{le:non-normal} Let $N$ be a polar piece of rate
  $q_i$ i.e., $N=\pi(\cal N(\Gamma_{q_i}'))$ where $\Gamma_{q_i}$ is a
  subgraph of $\Gamma$ consisting of a \P-node plus any attached
  bamboo. Assume that the outer boundary of $N$ is connected. Set
  $E'=\bigcup_{v_j \in \Gamma_{q_i}} E_j$ and $E''=\overline{E \setminus
    E'}$. Then for all $q$ with $q<q_i$ and all $y\in E' \setminus E''$,
$$\lim_{x \to y \atop x \in \widetilde{X}\setminus E} \widetilde{\alpha}_q(x,a) = \infty\,.$$  \end{lemma}

The proof of Lemma \ref{le:non-normal} needs a preparatory Lemma
\ref{le:parallel}.

Recall that the resolution $\pi$ factors through the Nash modification
$\nu \colon \widecheck{X} \to X$. Let $\sigma \colon \widetilde{X} \to
\grassman(2,n)$ be the map induced by the projection $p_2 \colon
\widecheck X \subset X \times \grassman(2,n) \to \grassman(2,n)$.  The
map $\sigma$ is well defined on $E$ and according to \cite[Section
2]{GS} (see also \cite[Part III, Theorem 1.2]{spivakovsky}), its
restriction to $E$ is constant on any connected component of the
complement of \P-curves in $E$.  The following lemma about limits of
tangent planes follows from this:

\begin{lemma}\label{le:parallel}
{\rm (1)}~~Let $\Gamma' $ be a maximal connected component of the graph $\Gamma$ 
with its \P-nodes removed (we call this a \emph{\P-Tjurina component}
of\/ $\Gamma$).  There exists $P_{\,\Gamma'} \in \grassman(2,n)$ such
that $\lim_{t\to 0}T_{\gamma(t)}X=P_{\,\Gamma'}$ for any real analytic
arc $\gamma \colon ([0,\epsilon],0) \to (\pi(N(\Gamma')),0)$ with
$|\gamma(t)|=O(t)$ and whose strict transform meets $\bigcup_{v \in
  \Gamma'} E_v$.

{\rm (2)}~~Let $E_{k} \subset E$ be a \P-curve and $x \in E_{k}$ be a
smooth point of the exceptional divisor $E$. There exists a plane
$P_{x}\in \grassman(2,n)$ such that $\lim_{t\to
  0}T_{\gamma(t)}X=P_{x}$ for any real analytic arc $\gamma \colon
([0,\epsilon],0) \to (\pi( {\mathcal N}(E_{k})),0)$ with
$|\gamma(t)|=O(t)$ and whose strict transform meets $E$ at $x$.\qed
\end{lemma}

\begin{proof}[Proof of Lemma \ref{le:non-normal}]
Let $y\in E' \setminus E''$ and let $\gamma \colon [0,\epsilon] \to X$ be
a real analytic arc inside $N \cap B_{\epsilon}$ whose strict
transform $\widetilde{\gamma}$ by $\pi$ meets $E$ at $y$ and such that
$|\gamma(t)|=O(t)$. We then have to prove that $$\lim_{t \to 0}
\widetilde{\alpha}_q(\widetilde{\gamma}(t),a) = \infty\,.$$ 

Let $B$ be the component of the unamalgamated intermediate carrousel
such that the outer boundary of a component $B'$ of $\ell^{-1}(B)$ is
the outer boundary of $N$. Let $N'\subseteq N$ be $B'$ with its polar
wedges amalgamated, so that $\ell|_{N'}\colon N'\to \ell(N')$ is a
branched cover of degree $>1$.

 Then 
  $\ell(N)$ has an $A(q'',q_i)$ annulus outside it for some $q''$,
  which we will simply call $A(q'',q_i)$. The lift of $A(q'',q_i)$ by
  $\ell$ is a covering space. Denote by $\tilde A(q'',q_i)$ the
  component of this lift that intersects $N$; it is connected since the
  outer boundary of $N$ is connected, and the degree of the covering is $>1$ since its inner boundary is the outer boundary of $N'$.

$\tilde A(q'',q_i)$ is
  contained in a $N(\Gamma')$ where $\Gamma'$ is a \P-Tjurina
  component of the resolution graph $\Gamma$ so we can apply part (1)
  of Lemma \ref{le:parallel} to any suitable arc inside it. This will
  be the key argument later in the proof.

  We will prove the lemma for $q'$ with $q''<q'<q_i$ since it is then
  certainly true for smaller $q'$.

Choose $p'$ with $q'<p'<q_i$
and consider the arc $\gamma_0\colon
[0,\epsilon]\to \C^2$ defined by $\gamma_0 = \ell \circ \gamma $ and the
function
 $$ \gamma_s(t) := \gamma_0(t)+(0,st^{p'})\quad \text{for
 }(s,t)\in[0,1]\times[0,\epsilon] \,.
 $$ 
 We can think of this as a family, parametrized by $s$, of arcs
 $t\mapsto \gamma_s(t)$, or as a family, parametrized by $t$, of real
 curves $s\mapsto \gamma_s(t)$. For $t$ sufficiently small
 $\gamma_1(t)$ lies in $\ell(\mathscr B(\gamma(t),a'|\gamma(t)|^{q'}))$ and also lies in the $A(q'',q_i)$  mentioned
 above. Note that for any $s$ the point $\gamma_s(t)$ is distance
 $O(t)$ from the origin.

 We now take two different continuous lifts $\gamma_s^{(1)}(t)$ and
 $\gamma_s^{(2)}(t)$ by $\ell$ of the family of arcs $\gamma_s(t)$,
 for $0\le s\le 1$, with $\gamma_0^{(1)}=\gamma$ and $\gamma_0^{(2)}$
 also in $N$ (possible since the covering degree of $\tilde
 A(q'',q_i)\to A(q'',q_i)$ is $>1$). Since $q<p'$, $\gamma_s^{(1)}(t)$
 and $\gamma_s^{(2)}(t)$ lie in $\mathscr B(\gamma(t),a'|\gamma(t)|^q)$
 for each $t, s$.
  
  To make notation simpler we set
 $P_1 = \gamma_1^{(1)}$ and $P_2 = \gamma_1^{(2)}$. Since the points
 $P_1(t)$ and $P_2(t)$ are on different sheets of the covering of
 $A(q'',q_i)$, a shortest path between them will have to travel through
 $N$, so its length $l_{inn}(t)$ satisfies $l_{inn}(t)=O(t^{p'})$.

 We now give a rough estimate of the outer distance
 $l_{out}(t)=|P_1(t)-P_2(t)|$ which will be sufficient to show
 $\lim_{t\to 0}(l_{inn}(t)/l_{out}(t))=\infty$, completing the proof. For
 this, we choose $p''$ with $p'<p''<q_i$ and consider the arc
 $p\colon[0,\epsilon] \to \Bbb C^2$ defined by:
 $$p(t):= \gamma_{s_t}(t)=\gamma_0(t)+(0,t^{p''})\quad\text{with
 }s_t:=t^{p''-p'}\,,
 $$ 
 and its two liftings $p_1(t):= \gamma_{s_t}^{(1)}(t)$ and $p_2(t):=
 \gamma_{s_t}^{(2)}(t)$, belonging to the same sheets of the cover
 $\ell$ as the arcs $P_1$ and $P_2$. A real slice of the situation is
 represented in Figure \ref{fig:lemma}.

 \begin{figure}[ht]
\centering
\begin{tikzpicture} 
\begin{scope}[scale=1.5]
 \draw[line width=0.3pt] (0,0).. controls (1,0.1) and (3,0.1)..(4,0);   
 \draw[line width=0.3pt] (0,1.5).. controls (1,1.6) and (3,1.6)..(4,1.5);   
     \draw[line width=1.5pt] (5,0.75).. controls (5,0.8) and (5,1.4)..(4,1.5);   
       \draw[line width=1.5pt] (5,0.75).. controls (5,0.7) and (5,-0.1)..(4,0);   
       
       \draw[fill]  (8.5:5.05)circle(2pt);
       \node(a)at(8.5:5.3){   $\Pi_{\ell}$};
        \draw[fill]  (15:4.9)circle(2pt);
          \draw[fill]  (2:4.75)circle(2pt);
        \node(a)at(16:5.5){   $\gamma_0^{(1)}(t)$};
           \node(a)at(1:4.8)[right]{   $ \gamma_0^{(2)}(t)$};
\node(a)at(9:4.5){   $N$};

   \node(a)at(-6:2){   $P_2(t)$}; 
    \draw[fill]  (2:2)circle(2pt);
    
   \node(a)at(43:2.7){   $P_1(t)$};
    \draw[fill]  (38.5:2.5)circle(2pt);

     \node(a)at(-4:3.5){   $p_2(t)$}; 
    \draw[fill]  (0.5:3.5)circle(2pt);
    
   \node(a)at(27:4){   $p_1(t)$};
    \draw[fill]  (24:3.8)circle(2pt);

 \draw[line width=0.6pt,dotted] (2:2)--(38.5:2.5);
     \node(a)at(27:1.7){   $l_{out}(t)$};


\end{scope}
  \end{tikzpicture} 
   \caption{}
  \label{fig:lemma}
\end{figure}

 The points $p_1(t)$ and $p_2(t)$ are inner distance $O(t^{p''})$
 apart by the same argument as before, so their outer distance is at
 most $O(t^{p''})$. By Lemma \ref{le:parallel} the line from $p(t)$ to
 $\gamma_1(t)$ lifts to almost straight lines which are almost
 parallel, from $p_1(t)$ to $P_1(t)$ and from $p_2(t)$ to $P_2(t)$
 respectively, with degree of parallelism increasing as $t\to 0$. Thus
 as we move from the pair $p_i(t)$, $i=1,2$ to the pair $P_i(t)$ the
 distance changes by $f(t)(t^{p'}-t^{p''})$ where $f(t)\to 0$ as $t\to
 0$. Thus the outer distance $l_{out}(t)$ between the pair is at most
 $O(t^{p''})+f(t)(t^{p'}-t^{p''})$. Dividing by $l_{inn}(t)=O(t^{p'})$
 gives $l_{out}(t)/l_{inn}(t)=O(f(t))$, so $\lim_{t\to
   0}(l_{out}(t)/l_{inn}(t))=0$. Thus $\lim_{t\to
   0}(l_{inn}(t)/l_{out}(t))=\infty$, completing the proof of  Lemma \ref{le:non-normal}.
\end{proof}

\subsection{Constructing $\boldsymbol{X'_{q_1}}$} \label{subsec:X'q1}

\begin{definition}[{\bf $\boldsymbol q$-neighbourhood}]
  Let $(U,0)\subset (\widehat X,0)\subset (X,0)$ be semi-algebraic
  sub-germs.  A \emph{$q$-neighbourhood of $(U,0)$ in $(\widehat
    X,0)$} is a germ $(N,0)\supseteq (U,0)$ with $N\subseteq \{x\in
  \widehat X~|~d(x,U)\le Kd(x,0)^q\}$ for some $K$, using inner metric
  in $(\widehat X,0)$. 
 \end{definition}

For   $a>0$ and $L>1$, define:
\begin{align*}
  Z_{q_1,a,L}:=\bigcup \bigl\{&\mathscr
B(x,a|x|^{q_1}): \alpha(\mathscr B(x,a|x|^{q_1}))>L \text{  or }\\ &\mathscr
B(x,a|x|^{q_1}) \text{ has essential topology } \bigr\}
\end{align*}

\begin{lemma} \label{le:Xq1} For sufficiently large $a$ and $L$ and
  sufficiently small $\epsilon_1\leq \epsilon_0$, the set $Z_{q_1,a,L} \cap
  B_{\epsilon_1}$ is a $q_1$-neighbourhood of $X_{q_1} \cap
  B_{\epsilon_1}$ in $X \cap B_{\epsilon_1}$.
\end{lemma}

\begin{proof} As before, to simplify notation we assume that
  $X_{q_{1}}\setminus \{0\}$ is connected, since otherwise we can argue  component by component.
  The following three steps prove the lemma.
  \begin{enumerate}
  \item\label{itp:1} If $X_{q_1}$ is not a $D$-piece there exists
    $a_1$ sufficiently large and $\epsilon$ sufficiently small that
    for each $a\ge a_1$ and $x\in X_{q_1}\cap B_\epsilon$ the set
    $\mathscr B(x,a|x|^{q_1})$ has essential topology.
  \item\label{itp:2} If $X_{q_1}$ is a $D$-piece then for any $L>1$
    there exists $a_1$ sufficiently large and $\epsilon$ sufficiently
    small that for each $a\ge a_1$ and $x\in X_{q_1}\cap B_\epsilon$
    the set $\mathscr B(x,a|x|^{q_1})$ has abnormality $>L$.
  \item\label{itp:3} There exists a $K>0$ such that for all $a>0$ there
    exists a $q_1$-neighbourhood $Z_{q_1}(a)$ of $X_{q_1}$ and
    $\epsilon>0$ such that any set $\mathscr B(x,a|x|^{q_1})$ not intersecting
    $Z_{q_1}(a)$ has trivial topology and
    abnormality $\le K$.
  \end{enumerate}

  Item \eqref{itp:1} is the same proof as the first part of the proof
  of Lemma \ref{le:bubbles} \eqref{it:bubbles-3}.

  For Item \eqref{itp:2} we use again the resolution $\pi \colon
  (\widetilde{X},E)\to (X,0)$ with dual graph $\Gamma$ defined in
  section \ref{sec:decomposition vs resolution}  and the notations introduced in the proof of Lemma
  \ref{le:bubbles} \eqref{it:bubbles-3}.

Denote by $E_{q_1}$ the union of components of $E$ corresponding to vertices of $\Gamma$ which are nodes of rate $q_1$ and attached bamboos.
  We claim that for each $y \in E_{q_1} \setminus E' $ we have:
  $$\lim_{a \to \infty} \lim_{x \to y\atop x \not\in E } \widetilde{\alpha}_{q_1}(x,a) = \infty\,.$$  
  Indeed, consider a real analytic arc $\gamma$ in $X$ such that
  $|\gamma(t)|=O(t)$ and $\widetilde{\gamma}(0)= y$.  Choose
  $q_2<q<q_1$.  For each $t$, set $a(t) = |\gamma(t)|^{q-q_1}$, so we
  have $\lim_{t \to 0} a(t)=\infty$.  Since $q_2<q<q_1$, $\mathscr
  B(\gamma(t),|\gamma(t)|^{q})$ has trivial topology for all $t$. By
  Item \eqref{it:bubbles-3} of Lemma \ref{le:bubbles} we have that for
  any $K_2 >1$ and $a=1$, there is $\epsilon<\epsilon_0$ such that for
  each $t\leq \epsilon$, $\alpha(\mathscr B(\gamma(t),|\gamma(t)|^{q}))\ge
  K_2$.  As $\mathscr B(\gamma(t),|\gamma(t)|^{q})= \mathscr
  B(\gamma(t),a(t)|\gamma(t)|^{q_1})$, we then obtain
 $$ \lim_{t\to 0}\widetilde\alpha_{q_1}(\widetilde\gamma(t),a(t))=\lim_{t \to 0} \alpha(\mathscr B(\gamma(t),a(t)|\gamma(t)|^{q_1})) =\infty\,.$$

 Choose $L>1$. The set $\cal N(\Gamma_{q_ i}$ intersects $E$ in a compact set
 inside $E_{q_1} \setminus (E_{q_1}\cap E')$. Therefore, by continuity
 of $\widetilde\alpha_{q_1}$, there exists $a_1 >0$ and $\epsilon >0$ such
 that for all $a\geq a_1$ and $x \in ( \cal N(\Gamma_{q_1}) \setminus E) \cap
 \cal \pi^{-1}(B_{\epsilon})$ we have $\widetilde\alpha_{q_1}(x,a)=\alpha(\mathscr B(\pi(x),a|\pi(x)|^{q_1}))>L$.
 This completes the proof of Item \eqref{itp:2}

 To prove Item \eqref{itp:3} we pick $\ell_1$, $\ell_2$, $K$ and $c$
 as in the proof of Lemma \ref{le:bubbles} \eqref{it:bubbles-1}. So
 with $A_1$ and $A_2$ as in that proof, any $\mathscr B(x,c |x|^{q_1})$ is
 disjoint from one of $A_1$ and $A_2$. We choose $\epsilon>0$ which we
 may decrease later.

 We now choose $a>0$.  Denote by $A'_1$ and $A'_2$ the union of
 components of $A_1$ and $A_2$ with rate $<q_1$. Note that the
 distance between any pair of different components of $A'_1\cup A'_2$
 is at least $O(r^{q_2})$ at distance $r$ from the origin, so after
 decreasing $\epsilon$ if necessary, any $\mathscr B(x,aK^2 |x|^{q_1})$
 which intersects one of $A'_1$ and $A'_2$ is disjoint from the other.
 Thus if $\mathscr B(x,aK^2 |x|^{q_1})$ is also disjoint from $X_{q_1}$,
 then the argument of the proof of Lemma \ref{le:bubbles}
 \eqref{it:bubbles-1} shows that $\mathscr B(x,a|x|^{q_1})$ has trivial
 topology and abnormality at most $K$. Since the union of those
 $B(x,aK^2 |x|^{q_1})$ which are not disjoint from $X_{q_1}$ is a
 $q_1$-neighbourhood of $X_{q_1}$, Item \eqref{itp:3} is proved.
  \end{proof}

The germ $(X'_{q_1},0)$ of Proposition \ref{prop:recover Xq} will be  defined as a smoothing of $Z_{q_1,a,L}$ for $L$ and $a$ sufficiently large. Let us first define what we mean by smoothing.

The outer boundary of $X_{q_1}$ attaches to $A(q',{q_1})$-pieces of
the (non-amalgamated) carrousel decomposition with $q'<q_1$, so we can add
$A(q_1,q_1)$ collars to the outer boundary of $X_{q_1}$ to obtain a
$q_1$-neighbourhood of $X_{q_1}$ in
$X$. We use $X^+_{q_1}$ to denote such an
enlarged version of $X_{q_1}$. An arbitrary $q_1$-neighbourhood of
$X_{q_1}$ in $X$ can be embedded in one
of the form $X^+_{q_1}$, and we call the process of replacing such a neighbourhood by $X^+_{q_1}$ a \emph{smoothing} of $X_{q_1}$.

This smoothing process can be applied more generally as follows:
\begin{definition}[{\bf Smoothing}] \label{def:smoothing} 
Let
   $(Y,0)$ be a subgerm of $(X,0)$ which is a union of $B$- and $A$-pieces lifted from a plane projection of $X$. So $(Y,0)$ has (possibly empty) inner and  outer boundaries.  Assume that the outer boundary components of $Y$  are
   inner boundary components of $A(q',q)$-pieces with $q'<q$.

Let $W$ be a $q$-neighbourhood of $Y$ in $(X,0)$ and let $V$ be the union of
the components of $\overline{W \setminus Y}$ which are inside
$A(q',q)$-pieces attached to the outer boundaries of $Y$. We call
the union $Z=V \cup Y$ an \emph{outer $q$-neighbourhood} of $Y$.

We call a \emph{smoothing} of $Z$ any outer $q$-neighbourhood $Z^+$
of $Z$ obtained by adding to $Y$ some $A(q,q)$-pieces to its outer
boundaries. 
\end{definition} 
 
\begin{remark}The smoothing $Z^+$ of $Z$ is uniquely determined
    from $Z$ up to homeomorphism since a $q$-neighbourhood of the form
    $Z^+$ is characterized among all outer $q$-neighbourhoods of $Z$
    by the fact that the links of its outer boundary components are
    tori and the number of outer boundary components is minimal among
    outer $q$-neighbourhoods of $Z$ in $X$.
\end{remark}

\subsection*{Definition of the germ $\boldsymbol{(X'_{q_1},0)}$}
 
Let $a>0$ and $L>1$ sufficiently large as in Lemma \ref{le:Xq1}. We
define $X'_{q_1}$ as the smoothing
$$X'_{q_1}:=Z_{q_1,a,L}^+\,.$$
Since $Z_{q_1,a,L}$ is a $q_1$ neighbourhood of $X_{q_1}$ inside a
Milnor ball $B_\epsilon$, the germ $(X'_{q_1},0)$ is equivalent to
$(X_{q_1},0)$.

\subsection{Discovering $\boldsymbol{q_2}$}\label{subsec:q2}

\red{Using the results of  Section \ref{sec:carrousel1} and the  beginning of Section \ref{sec:detecting1}}, we discover $q_2$ in two steps.

\subsection*{Step 1} We consider $N_{q,1}(X'_{q_1})^+$ with $q\le q_1$. If $q$ is close to $q_1$ this just adds collars to the boundary
 of $X'_{q_1}$ so it does not change the topology. 
Let $q_1'$ be the infimum of $q$ for which the topology  \red{and the number of proximity classes}  of 
   $ N_{q,1}(X'_{q_1})^+$ has not changed.  Then for $\alpha>0$
sufficiently small, the topology $N_{q'_1,\alpha}(X'_{q_1})^+$ has
also not changed.

\subsection*{Step 2} Using essentially the same argument as we used to
discover $q_1$(Proposition \ref{prop:finding q1}) we discover $q_2$ as follows.
$q_2$ is the infimum of all $q$ with $q_1'< q<q_1$ such that there exists
$K_1>0$ such that for all $a>0$ there exists $\epsilon_1>0$ such that
all sets $\mathscr B(x,a|x|^q)$ with $x\in X\setminus
   N_{q'_1,\alpha}(X'_{q_1})^+$ have trivial topology and abnormality at
most $K_1$. If there are no such $q$ we set $q_2=q'_1$.

\subsection{Constructing $\boldsymbol{X'_{q_2}}$ and
  $\boldsymbol{A'_{q_2,q_1}}$}\label{subbsecsec:Xq2 and Aq2q1} 

\red{$X'_{q_2}$ decomposes as the union $X'_{q_2} = X''_{q_2} \cup X'''_{q_2}$ where $X''_{q_2}$ (resp.\ $X'''_{q_2}$) is the union of components which have nonempty (resp.\ empty) inner boundary.   

We first discover $X'''_{q_2}$.  We use again the procedure of Subsection
\ref{subsec:X'q1}, working now inside $X\setminus N_{q,a}(X'_{q_1})$ with $q$ slightly
smaller than $q_2$ (see also Step 2 of Subsection \ref{subsec:q2}).  Namely, choose $q$ slightly smaller than $q_2$ and set: 
 \begin{align*}
  Z_{q_2,a,L}:=\bigcup \bigl\{&\mathscr
B(x,a|x|^{q_2}) :  x \in X\setminus N_{q,a}(X'_{q_1}) \text{ and}\\ & \alpha(\mathscr B(x,a|x|^{q_2}))>L \text{  or } \mathscr
B(x,a|x|^{q_2}) \text{ has essential topology} \bigr\}
\end{align*}
with $a$ and $L$ sufficiently large. Then $X'''_{q_2}:=  Z_{q_2,a,L}^+$ is a $q_2$-neighbourhood of the union of components of $X_{q_2}$ having empty inner boundary. 

We now describe $X''_{q_2}$.  
If $q_2 > q'_1$, then $X''_{q_2}$ is the union of the components $A$ of  $\overline{N_{q_2,\beta}(X'_{q_1})^+\setminus N_{q_2,\alpha}(X'_{q_1})^+}$ with $d(A\cap S_\epsilon,X'''_{q_2}\cap S_\epsilon)=O(\epsilon^q)$ and $q\ge q_2$.  If $q_2=q'_1$ then $X''_{q_2}$ has some extra components which are described as follows. There exist  $\alpha$ sufficiently small that    $N_{q_2,\alpha}(X'_{q_1})^+$ has the same topology as $X'_{q_1}$ and there exists  $\beta>0$ such that for all $a\ge \beta$ the topology of  $N_{q_2,a}(X'_{q_1})^+$ is no longer that of $X'_{q_1}$ and increasing $a$ does not change the topology. 

Note that  $\overline{N_{q_2,\beta}(X'_{q_1})^+ \setminus X'_{q_1}}$ consists of at least one piece which is not an $A(q_2,q_1)$-piece and maybe some  $A(q_2,q_1)$-pieces. 
Then the extra components of 
$X''_{q_2}$ are  the components of $\overline{N_{q_2,\beta}(X'_{q_1})^+\setminus N_{q_2,\alpha}(X'_{q_1})^+}$ which are not $A(q_2,q_2)$-pieces  or such that the  cardinality of their proximity class is strictly greater than that of the corresponding components of  $N_{q',1}(X'_{q_1})^+$ with $q_2<q'<q_1$.

Now $\overline{N_{q_2,\alpha}(X'_{q_1})^+ \setminus X'_{q_1}}$ consists of
$A(q_2,q_1)$-pieces. We define $A'_{q_2,q_1}$ to be the union of these
$A(q_2,q_1)$-pieces which intersect $N_{q_2,\alpha}(X'_{q_1})^+$.
}
\subsection{Constructing $\boldsymbol{X'_{q_i}}$ and $\boldsymbol{A'_{q_i,q_j}}$ for $\boldsymbol{i>2}$}\label{subsec:X'qi}

We now assume we have already done the construction for smaller $i$. Let
$X^{(i-1)}$ be the union of all $X'_{q_j}$ with $j\le i-1$ and all
$A(q_j,q_k)$-pieces connecting an $X'_{q_j}$ with $X'_{q_k}$ with
$k<j\le i-1$.
 
We  use the same arguments as in the construction for $i=2$, using
$X^{(i-1)}$ in place of $X'_{q_1}$ in the discovery of $q_i$ and construction of the sets  $X'_{q_i}$ and $A'_{q_i,q_j}$.  

This iterative procedure ends when   \red{steps 1 and  2 of Subsection \ref{subsec:q2} adapted to the discovery of
$q_i$  leads to $q_i=1$, in which case $i=\nu$}.  We then define
$X'_{q_\nu}=X'_1$ as the closure of the complement of the result of
gluing a piece of type $A(1,q_j)$ on each outer boundary of
$X^{(\nu-1)}$.

This completes the construction from the outer metric of the
decomposition of $(X,0)$ in Proposition \ref{prop:recover Xq}. Finally, we show that the decomposition can still be recovered after
a bilipschitz change to the metric.


\begin{proposition}\label{prop:invariance}
  Let $(X,0) \subset (\Bbb C^n,0)$ and $(X',0)\subset (\Bbb C^{n'},0)$
  be two germs of normal complex surfaces endowed with their outer
  metrics. Assume that there is a 
  bilipschitz map $\Phi \colon (X,0) \to (X',0)$. Then the inductive
  process described in Subsections \ref{subsec:q1} to
  \ref{subsec:X'qi} leads to the same sequence of rates $q_i$ for both
  $(X,0)$ and $(X',0)$ and for $a>0$ and $L>1$ sufficiently large, the
  corresponding sequences of subgerms $Z_{q_i,a,L}$ in $X$ and
  $Z'_{q_i,a,L}$ in $X'$ have the property that $\Phi(Z_{q_i,a,L}^+)$
  and $Z'^+_{q_i,a,L}$ are equivalent.
 \end{proposition}

\begin{proof}

  Let $K$ be the bilipschitz constant of $\Phi$ in a fixed
  neighbourhood $V$ of the origin. The proof follows from the
  following three observations.

(1).~ Given $x\in V$, $a>0$ and $q\geq 1$, there exists $a'>0$ such that 
$$\Phi(\mathscr B(x,a|x|^q))\subset \mathscr B(\Phi(x),a'|\Phi(x)|^q)\,.$$
Indeed, $\Phi(\mathscr B(x,a|x|^q))\subset \mathscr B(\Phi(x),aK|x|^q) \subset
\mathscr B(\Phi(x),aK^{q+1}|\Phi(x)|^q)$. So any $a'\geq aK^{q+1}$ works.

(2).~ Let $N$ be a subset of $V$. Then the abnormality of $\Phi(N)$ is controlled by that of $N$: 
$$\frac 1{K^2}\alpha(N) \leq \alpha(\Phi(N)) \leq K^2 \alpha(N)\,.$$

(3).~Let $(U,0)\subset(\widehat{X},0) \subset(X,0)$ be
  semi-algebraic sub-germs. If $(N,0)$ is a $q$-neighbourhood of
  $(U,0)$ in $(\widehat{X},0)$ then $(\Phi(N),0)$ is a
  $q$-neighbourhood of $(\Phi(U),0)$ in $(\Phi(\widehat{X}),0)$.
\end{proof}

\section{Explicit computation of polar rates}\label{sec:explicit computation}

The argument of the proof of Lemma
\ref{le:non-normal} enables one to compute the rate
of a polar piece in simple examples such as the following.
  Assume $(X,0)$ is a hypersurface
  with equation $z^2=f(x,y)$ where $f$ is reduced.  The projection
  $\ell = (x,y)$ is generic and its discriminant curve $\Delta$ has
  equation $f(x,y)=0$.  Consider a branch $\Delta_0$ of $\Delta$ which
  lifts to a polar piece $N$ in $X$. We consider the Puiseux expansion
  of $\Delta_0$:
 $$y=\sum_{i\geq1} a_i x^{q_i} \in \C\{x^{1/m}\}\,.$$
 
 The polar rate  $s$ of $N$ is the minimal $r \in \frac1m \N$ such that
 for any small $\lambda$, the curve $\gamma$: $y=\sum_{i\geq1} a_i
 x^{q_i} + \lambda x^r$ is in $\ell(N)$.  In order to compute $s$, we
 set $r'=r m$   and we
 parametrize $\gamma$ as:
 $$ x=w^{m}, \ \ y= \sum_{i\geq1} a_i w^{m q_i} + w^{r'}\,. $$
 Replacing in the equation of $X$, and approximating by elimination of the
 monomials with higher order in $w$, we obtain $z^2 \sim aw^{r'+b'}$
 for some $a\ne 0$ and some positive integer $b'$. We then  have an outer
   distance of  $O(x^{\frac {r+b}{2}})$ between the two sheets of the
   cover $\ell$, where $b = \frac{b'}{m}$. So 
 the  optimal $s$ such that the curve $\gamma$ is in $\ell(N)$ is
 given by $\frac{s+b}{2} = s$, i.e., $s=b$.   
   
\begin{example}\label{ex:VTrates}  We apply this to the singularity $D_5$ with equation $z^2=-(x^2y
 + y^4)$ (see Example \ref{ex:D5-resolution}). The discriminant curve of
 $\ell = (x,y)$ has equation $y(x^2+y^3)=0$.  For the polar piece
 $\pi({\mathcal N}(E_6))$, which projects on a neighbourhood of the
 cusp $\delta_2 = \{x^2+y^3=0\}$, we use $y=w^2$, $x=iw^3+w^{r'}$, so
 $z^2\sim-2iw^{5+r'}$, so  $b'=5$ and this polar piece has rate $s=5/2$.  Similarly
 one computes that the polar
 piece $\pi({\mathcal N}(E_4))$, which projects on a neighbourhood of
 $\delta_1 = \{y=0\}$, has rate $2$.
\end{example}

Notice that in the above computation, we have recovered  the relation $q(r)=\frac{r+s}{2}$ established for $r\geq s$ in  the proof of Lemma \ref{lemma:constant}. In fact, combining  Lemma \ref{lemma:constant} and Lemma
\ref{le:non-normal}, we can compute polar rates in a more general setting, as we explain now,  using e.g., Maple.

Let $Y$ be a polar piece in $(X,0)$ which is a $D(s)$-piece. Let $q$ be the rate of the $B$-piece outside it, so there is an intermediate $A(q,s)$-piece  $A$ between them. Assume we know the rate $q$.  Let $\Pi_0$ be a component of $\Pi$ inside $Y$ and let $\Delta_0 = \ell(\Pi_0)$. Consider a Puiseux expansion of $\Delta_0$  as before and an irreducible  curve $\gamma$ with Puiseux expansion 
$$y=\sum_{i\geq1} a_i x^{q_i} + a x^r \,.$$
having contact $r > q$ with $\Delta_0$. Let $L_{\gamma}=\ell^{-1}(\gamma)$ and let $\ell' \colon X \to \C^2$ be a generic plane projection of $(X,0)$ which is also generic for $L_{\gamma}$. Let $L'_{\gamma}$ be the intersection of $L_{\gamma}$ with $Y \cup A$.     Let $q(r)$ be the greatest characteristic exponent of the curve  $ \ell'(L'_{\gamma})$. Given $r$ one can compute $q(r)$ from the equations of $X$. 

If $q< r < s$, then, according to Lemma \ref{le:non-normal}, $q(r)>r$. If $r \geq s$,  then, according to Lemma \ref{lemma:constant}, $q(r) \leq s$.  Therefore, testing the inequality $q(r)>r$, one can choose an $r_0$ big enough so that $q(r_0) \leq r_0$, i.e., $r_0 \geq s$. Now, we have $q(r_0)=\frac{r_0+s}{2}$ (again by Lemma \ref{lemma:constant}). Therefore $s=2 q(r_0)+r_0$. 

\begin{example} This method  gives rise to the polar rate $s=\frac{7}{5}$ claimed in Example \ref{ex:very big0}, using a Maple computation. 
\end{example}

\section{Carrousel decomposition  from Lipschitz geometry} 
\label{sec:finding remaining polars}
 
The aim of this section is to complete the proof of Theorem
\ref{th:invariants from  geometry}. We will first
  prove that the outer Lipschitz geometry determines
  the sections of the complete \carrousel\ of the discriminant curve
  $\Delta$ of a generic plane projection of $(X,0)$.

Let $\ell=(z_1,z_2)\colon X\to \C^2$ be the generic plane projection
and $h=z_1|_X$ the generic linear form chosen in Sections \red{\ref{sec:carrousel1} and}
\ref{sec:resolve hyperplane pencil}. We denote by $ F(t) :=
h^{-1}(t)$ the Milnor fibre of $h$ for $t\in(0,\epsilon]$.  \red{We will continue to use the terminology $X_{q_i}$ and $A_{q_i,q_j}$ for the pieces of the geometric decomposition as constructed from the carrousel in Section \ref{sec:carrousel1}, and $X'_{q_i}$ and $A'_{q_i,q_j}$ for pieces equivalent to them (Definition \ref{def:equivalent pieces}) as recovered using outer geometry (Section \ref{sec:detecting1}).}

\begin{proposition} \label{prop:find carrousel} The  outer Lipschitz
  geometry of $(X,0)$ determines:
 \begin{enumerate}
 \item \label{eq:carrousel} the  combinatorics of the  complete  carrousel section of $\Delta$.
 \item \label{eq:polar} the number of components of the polar curve
   $\Pi$ in each $B$- or $D$-piece of $(X,0)$.
 \end{enumerate}
\end{proposition}

\begin{proof}[Proof of Proposition \ref{prop:find carrousel}]
  By Proposition \ref{prop:recover Xq} we may assume that we have
  recovered the pieces $X_{q_i}$ of $X$ for $i=1,\dots,\nu$ and the
  intermediate $A$-pieces up
  to equivalence (Definition \ref{def:equivalent pieces}).

  We will recover the combinatorics of the carrousel sections by an
  inductive procedure on the rates $q_i$ starting with $q_1$. We
  define $A\!X_{q_i}$ to be
$$A\!X_{q_i}:= X_{q_i}\cup \bigcup_{k<i}A_{q_i,q_k}\,.$$

Note that $A\!X_{q_1}=X_{q_1}$. We first recover the  pieces
  of $\ell(A\!X_{q_1} \cap F(\epsilon))$ and inside them, the section of
  the complete carrousel of $\Delta$ beyond rate $q_1$ (Definition
  \ref{def:beyond}). Then we glue to some of their outer boundaries the pieces
  corresponding to $\ell((A\!X_{q_1}\cup A\!X_{q_2})\cap
  F(\epsilon))\setminus\ell (A\!X_{q_1} \cap F(\epsilon))$, and inside
  them, we determine the  complete carrousel beyond $q_2$, and so
  on. At each step, the outer Lipschitz geometry will determine the
  shape of the new pieces we have to glue, how they are glued, and
  inside them, the complete carrousel section beyond the corresponding
  rate.

  For any $i<\nu$ we will denote by \red{$A\!X_{q_i+}$} the result of adding
  $A(q',q_i)$-pieces to all components of the outer boundary
  $\partial_oA\!X_{q_i}$ of $A\!X_{q_i}$ for some $q'$ with
  $q_{i+1}<q'<q_i$.  We can assume that the added $A$-pieces are
  chosen so that the map $\ell$ maps each by a branched covering map to an
  $A(q',q_i)$-piece in $\C^2$. We will denote by $\partial_oA\!X_{q_i+}$
  the outer boundary of $A\!X_{q_i+}$, so it is a horn-shaped cone on a
  family of tori with rate $q'$. We denote
$$F_{q_i}(t):=F(t)\cap A\!X_{q_i},\quad F_{q_i+}(t):=F(t)\cap A\!X_{q_i+}\,,$$ 
and their outer boundaries by
$$\partial_oF_{q_i}(t)=F(t)\cap \partial_oA\!X_{q_i},\quad \partial_oF_{q_i+}(t)=F(t)\cap \partial_oA\!X_{q_i+}\,.$$

A straightforward consequence of \red{the proof of} Theorem
\ref{th:invariants from geometry}\eqref{it:resolution of pencils} is
that the outer Lipschitz geometry determines the isotopy class of the
Milnor fibre $F(t)$ and how the Milnor fiber lies in a neighbourhood
of the exceptional divisor of the resolution described there. In
particular it determines  the intersection
$F_{q_i}(t) = F(t)\cap A\!X_{q_i}$ for $i<\nu$ and $t\in (0,\epsilon]$
\red{up to  small deformation caused by any bilipschitz change of
  metric}.

Set $B_{q_1}=\ell(A\!X_{q_1})$ and
  $B_{q_1}(\epsilon):=\ell(F_{q_1}(\epsilon))$, so $B_{q_1}(\epsilon)$
  consists of the innermost disks of the intermediate carrousel
  section of $\Delta$. Notice that $B_{q_1}(\epsilon)$ is the union of
  all the sections of the complete carrousel beyond rate $q_1$.  Consider
  the restriction $\ell \colon F_{q_1}(\epsilon) \to B_{q_1}(\epsilon)
  $. We want to show that this branched cover can be seen in terms of
  the outer Lipschitz geometry of a neighbourhood of $A\!X_{q_1}$.  Since
  $B_{q_1}(\epsilon)$ is a union of disks, the same is true for
  $B_{q_1+}(\epsilon):=\ell(F_{q_1+}(\epsilon))$.

  Consider a component $F_{q_1,0}(t)$ of $F_{q_1}(t)$ and let
  $A\!X_{q_1,0}$ be the component of $A\!X_{q_1}$ such that
  $F_{q_1,0}(t)\subseteq F(t)\cap A\!X_{q_1,0}$. Consider a continuous
  arc $ p\colon[0,\epsilon] \to \partial_oA\!X_{q_1,0}$ such that
  $p(t)\in \partial_oF_{q_1,0}(t)$ for $t\in(0,\epsilon]$. \red{By the argument of} Lemma \ref{le:outer boundary}, if $\delta$ is big enough then for
  all $t$ small enough a component of $F_{q_1}(t)$ maps to the
  component $B_{q_1,0}(t):=\ell(F_{q_1,0}(t))$ of $B_{q_1}(t)$ if and
  only if it intersects the ball $B(p(t),\delta t^{q_1})$.
  
  If there is a component of $F_{q_1}(t)$ which does not map on $
  B_{q_1,0}(t)$, we iterate the process until we have determined
  the number of discs in $B_{q_1}(t)$ and how $\ell$ maps the
  components of $F_{q_1}(t)$ onto them.

 Using similar arcs on the outer boundaries of the components of
  $\overline{A\!X_{q_1+} \setminus A\!X_{q_1}}$, we obtain by the same
  argument that the outer Lipschitz geometry determines the degree of the
  restriction of $\ell$ on each annular component of
  $\overline{F_{q_1+}(t) \setminus F_{q_1}(t)}$, and then, the degree
  of $\ell$ restricted to each component of $\partial_o F_{q_1+}(t)$,
  resp.\ $\partial_o F_{q_1}(t)$ (alternatively, this follows from the
  proof of Lemma \ref{le:non-normal}).

  We hence obtain the degree $n$ of $\ell|_{F_{q_1,0}(t)}\colon
  F_{q_1,0}(t)\to B_{q_1,0}(t)$.  We wish to determine the number
  \red{$\beta_0$} of branch points of this map. Since the cover is general,
  the inverse image by $\ell|_{F_{q_1,0}(t)}$ of a branch point
  consists of $n-1$ points in $F_{q_1,0}(t)$ and the Hurwitz formula
  applied to $\ell|_{ F_{q_1,0}(t)}$ gives:
$$\red{\beta_0}=n-\chi(F_{q_1,0}(t)) \,.$$
\red{Denote by $\beta$ the product of $\beta_0$ by the number of components of $AX_{q_1,0}\cap F(t)$.} 
 
We now determine the number of components of the polar inside the
component $A\!X_{q_1,0}$ of $A\!X_{q_1}$ which contains $F_{q_1,0}(t)$.  We use a
similar argument as in the proof of Proposition \ref{prop:plane
  curve}, taking account of Remark \ref{rk:beyond}.

\red{We denote by $D(t)$ the connected component containing $p(t)$ of the intersection of $A\!X_{q_1}$ with a ball $B(p(t),\eta t)$, $\eta <\!\!<1$.}
There is a $\mu>0$ such that for all
$q>q_1$ sufficiently close to $q_1$ the ball
$B(p(t),t^{q})$ intersects 
\red{$\ell^{-1}(\ell(D(t)))$} in $\mu$  \red{connected components}
$D_1(t),\ldots,D_{\mu}(t)$ for all small $t$.
For each $j,k$ with $1 \leq j<k\leq \mu$, let $q_{jk}$ be defined by
$d(D_j(t),D_k(t)) = O(t^{q_{jk}})$.

Let $A$ be the set of indices $j\in\{1,\dots,\mu\}$ \red{with $D_j(t) \subset A\!X_{q_1+,0}$}. The
outermost piece of $A\!X_{q_1,0}$ before amalgamation is a union of
equisingular curves $\Pi_{\cal D} \cap A\!X_{q_1,0}$ with $\cal D $
generic and each curve $\Pi_{\cal D} \cap A\!X_{q_1,0}$ consists of
equisingular components having pairwise contact $q_1$ (see Section
\ref{sec:polarwedges}). Therefore the
collection of rates $q_{j,k}$ indexed by $j,k \in A$ reflects the
outer Lipschitz geometry beyond rate $q$ (Definition \ref{def:beyond}) of each
irreducible component $\Pi_{\cal D,0}$ of $\Pi_{\cal D} \cap
A\!X_{q_1,0}$. In particular, it determines the number $r$ of points in
the intersection $\Pi_{\cal D,0} \cap F(\epsilon)$. Then $\Pi \cap
A\!X_{q_1,0}$ consists of $\frac{\beta}{r}$ equisingular irreducible
curves with pairwise contact $q_1$ and such that each component has
the outer Lipschitz geometry of $\Pi_{\cal D,0}$. 

  \red{As in the proof of Proposition \ref{prop:plane curve}, the components we use to determine the numbers $q_{jk}$ may disintegrate into several pieces with a bilipschitz change of metric, but this can be dealt with as in that proof.}

To complete the initial step of our induction, we will now recover the
section of the complete carrousel of $\Delta$ beyond rate $q_1$, i.e.,
inside each component of $B_{q_1}(\epsilon)$.  
As we
have  just seen, the collection of rates $q_{j,k}$ indexed by $j,k \in
A$ determines the number of components of the polar inside each
component $A\!X_{q_1,0}$ of $A\!X_{q_1}$ and their outer
geometry. Since the projection $\ell=\ell_{\cal D}$ is a generic
projection for its polar curve $\Pi=\Pi_{\cal D}$ \cite[Lemme 1.2.2
ii)]{T3}, we have then showed that the outer Lipschitz geometry recovers the
outer Lipschitz geometry of $\ell(\Pi)$ beyond rate $q_1$ inside $B_{q_1,0}$, or
equivalently, the section of the complete carrousel of the
discriminant curve $\Delta=\ell(\Pi)$ inside $B_{q_1,0}(\epsilon)$.

Doing this for each component $B_{q_1,0}(\epsilon)$ of $B_{q_1}(\epsilon)$ we then
reconstruct the complete carrousel section of $\Delta$ beyond rate $q_1$.

This completes the initial step of our induction.

\red{\subsection*{Overlapping} In the next steps $i
\geq 2$, we give special attention to the following  phenomenon. It
can happen that $\ell(A\!X_{q_i})$ contains connected components of some
$\ell(A\!X_{q_j})$ for $j<i$. We say that $A\!X_{q_i}$ \emph{overlaps}
$A\!X_{q_j}$. This occurs when some components of $X_{q_i}$ are obtained by amalgamation of $D$-pieces to $B(q_i)$-pieces lifted from $B_{q_i}$-pieces of the carrousel. This phenomenon is illustrated further in Example
\ref{ex:very big2} where the thick part $A\!X_{1}$ overlaps the
components of a thin part.}


\smallskip We now consider $F_{q_2}(\epsilon)$. As before, we know
from the outer Lipschitz metric for $F_{q_2+}(\epsilon)$ how the outer boundary
$\partial_oF_{q_2}(\epsilon)$ covers its image. We will focus first on
a single component $F_{q_2,0}(\epsilon)$ of $F_{q_2}(\epsilon)$. The
map $\ell|_{F_{q_2,0}(\epsilon)}\colon F_{q_2,0}(\epsilon)\to
\ell(F_{q_2,0}(\epsilon))$ is a covering map in a neighbourhood of its
outer boundary, whose degree, $m$ say, is determined from the outer Lipschitz
geometry. The image $\partial_o\ell(F_{q_2,0}(\epsilon))$ of the outer
boundary is a circle which bounds a disk $B_{q_2,0}(\epsilon)$ in
the plane $\{z_1=\epsilon\}$. The image $\ell(F_{q_2,0}(\epsilon))$ is this
$B_{q_2,0}(\epsilon)$, possibly with some smaller disks removed,
depending on whether and how $A\!X_{q_{2}}$ overlaps $A\!X_{q_1}$.  The
image of the inner boundaries of $F_{q_2,0}(\epsilon)$ will consist of disjoint
circles inside $B_{q_2,0}(\epsilon)$ of size proportional to $\epsilon^{q_1}$.  Consider the components of
$F_{q_1}(\epsilon)$ which fit into the inner boundary components of
$F_{q_2,0}(\epsilon)$. Their images form a collection of disjoint disks
$B_{q_1,1}(\epsilon),\dots,B_{q_1,s}(\epsilon)$.  For each $j$ denote by $F_{q_1,j}(\epsilon)$ the
union of the components of $F_{q_1}(\epsilon)$ which meet $F_{q_2,0}(\epsilon)$ and map to
$B_{q_1,j}(\epsilon)$ by $\ell$. By the first step of the induction we know the
degree $m_j$ of the map $\partial_o F_{q_1,j}(\epsilon)\to \partial_o
B_{q_1,j}(\epsilon)$. This degree may be less than $m$, in which case
$\ell^{-1}(B_{q_1,j}(\epsilon))\cap F_{q_2,0}(\epsilon)$ must consist of $m-m_j$
disks. Thus, after removing $\sum_j(m-m_j)$ disks from $F_{q_2,0}(\epsilon)$, we
have a subset $\hat F_{q_2,0}(\epsilon)$ of $F_{q_2,0}(\epsilon)$ which maps to
$\overline{B_{q_2,0}(\epsilon)\setminus \bigcup_{j=1}^s B_{q_1,j}(\epsilon)}$ by a
branched covering. Moreover, the branch points of the branched cover
$\hat F_{q_2,0}(\epsilon)\to \overline{B_{q_2,0}(\epsilon)\setminus \bigcup_{j=1}^s
  B_{q_1,j}(\epsilon)}$ are the intersection points of $F_{q_2,0}(\epsilon)$ with the
polar $\Pi$. We again apply the Hurwitz formula to discover the 
number of these branch points: it is $m(1-s)-\chi(\hat
F_{q_2,0}(\epsilon))=m(1-s)-\chi(F_{q_2,0}(\epsilon))+\sum_{j=1}^s(m-m_j)$.

Then we use balls along arcs as before to find the number of branches
of the polar in each component of $A\!X_{q_2}$, and since outer Lipschitz geometry
tells us which components of $F_{q_2}(\epsilon)$ lie over which
components of $\ell(F_{q_2}(\epsilon))$, we recover how many
components of the discriminant meet each component of
$\ell(F_{q_2}(\epsilon))$.

Different components $F_{q_2,0}(\epsilon)$ of $F_{q_2}(\epsilon)$ may correspond to the
same $B_{q_2,0}(\epsilon)$, but as already described, this is detected using
outer Lipschitz geometry (Lemma \ref{le:outer boundary}). We write $B_{q_2}(\epsilon)$ for
the union of the disks $B_{q_2,0}(\epsilon)$ as we run through the components of
$F_{q_2}$. The resulting embedding of $B_{q_1}$ in $B_{q_2}$ is the
next approximation to the \carrousel. Moreover, by the same procedure
as in the initial step, we determine the complete carrousel section of
$\Delta$ beyond rate $q_2$, i.e., within $B_{q_2}=\ell(A\!X_{q_1}\cup A\!X_{q_2})$.

\smallskip Iterating this procedure for $F_{q_i}(\epsilon)$ with $2<i<\nu$, we build up for $i=1,\dots,\nu-1$ the picture
of the complete carrousel section for $\Delta$ beyond rate $q_i$ while
finding the number of components of the polar in each component of
$A\!X_{q_i}$ and therefore the number of components of the discriminant
in the pieces of the \carrousel.

Finally for a component $A\!X_{1,0}$ of $A\!X_{q_\nu}=A\!X_1$ the
degree of the map to $\ell(A\!X_{1,0})$ can be discovered by the same
ball procedure as before; alternatively it is the multiplicity of the
maximal ideal cycle at the corresponding $\cal L$-node, which was
determined in Item \eqref{it:multiplicity and Zmax} of Theorem
\ref{th:invariants from geometry}. The same Euler characteristic
calculation as before gives the number of components of the polar in
each piece, and the ball procedure as before determines the complete
carrousel section of $\Delta$ beyond rate $q_\nu=1$, i.e., the complete
carrousel section.
\end{proof}
   
\begin{proof}[Completion of proof of Theorem \ref{th:invariants from
    geometry}] Parts \eqref{it:resolution of pencils} to
  \eqref{it:hyperplane section}  of Theorem \ref{th:invariants from
    geometry} were proved in Section
  \ref{sec:resolve hyperplane pencil} and part
  \eqref{it:discriminant} is proved above
  (Proposition \ref{prop:find carrousel}).  \eqref{it:discriminant}$\Rightarrow$\eqref{it:curves} is 
straightforward:
By Pham-Teissier \cite{PT} (see also Fernandes \cite{F},
Neumann-Pichon \cite{NP}), the outer Lipschitz  geometry of a complex curve in $\C^n$ is determined by the
topology of a generic plane projection of this curve.  By Teissier
\cite[page 462 Lemme 1.2.2 ii)]{T3}, if one takes a generic plane
projection $\ell\colon (X,0)\to(\C^2,0)$, then this projection is
generic for its polar curve. Therefore the topology of the
discriminant curve determines the outer  Lipschitz  geometry of the polar curve.

It remains to prove part
  \eqref{it:resolution of pencils1}.

  We consider the decomposition of the link $X^{(\epsilon)}$ as the
  union of the Seifert fibered manifolds $X_{q_i}^{(\epsilon)}$ and
  show first that we know the Seifert fibration for each
  $X_{q_i}^{(\epsilon)}$.  Indeed, the Seifert fibration for a
  component of a $X_{q_i}^{(\epsilon)}$ is unique up to isotopy unless
  that component is the link of a $D$- or $A$-piece. Moreover, on any
  torus connecting two pieces $X_{q_i}^{(\epsilon)}$ and
  $X_{q_j}^{(\epsilon)}$ we know the relative slopes of the Seifert
  fibrations, since this is given by the rates $q_i$ and $q_j$. So if
  there are pieces of the decomposition into Seifert fibered pieces
  where the Seifert fibration is unique up to isotopy, the Seifert fibration
  is determined for all pieces. This fails only if the link
  $X^{(\epsilon)}$ is a lens space or torus bundle over $S^1$, in
  which case $(X,0)$ is a cyclic quotient singularity or a cusp
  singularity. These are taut, so the theorem is trivial for
  them.

  We have shown that the outer Lipschitz geometry determines the decomposition
  of $(X,0)$ as the union of subgerms $(X_{q_i},0)$ and intermediate $A$-pieces as well as the
  number of components of the polar curve in each piece.  The link $L$ of
  the polar curve is a union of Seifert fibers in the links of some of the
  pieces, so we know the topology of the polar curve.  By taking two
  parallel copies of $L$ (we will call it $2L$), as in the proof of
  \eqref{it:resolution of pencils} of Theorem \ref{th:invariants from
    geometry}, we fix the Seifert fibrations and can therefore recover
  the dual graph $\Gamma$ we are seeking as the minimal negative definite
  plumbing diagram for the pair $(X^{(\epsilon)},2L)$.  Since we know
  the number of components of the polar curve in each piece, the proof is
  complete.
\end{proof}

\begin{example} \label{ex:D5 2} Let us show how the procedures
  described in Section \ref{sec:detecting1} and in proof of 
  \eqref{eq:carrousel} of Proposition \ref{prop:find carrousel} work
  for the singularity $D_5$, i.e., recover the decomposition $X
  = X_{5/2} \cup X_2 \cup X_{3/2} \cup X_1$   plus intermediate A-pieces, and the complete carrousel sections of the discriminant curve $\Delta$.  It starts by
  considering high $q$, and then, by 
  decreasing gradually  $q$, one reconstructs the sequence
  $X_{q_1},\ldots,X_{q_{\nu}}$ up to equivalence as described in Proposition
  \ref{prop:recover Xq}, using the procedure of Section \ref{sec:detecting1}.   
At each step we will do the following:  
 \begin{enumerate}
 \item  On the left in the figures below we draw the dual resolution
   graph with black vertices corresponding to the piece $X_{q_i}$
   just discovered, while the vertices corresponding to the previous
   steps are in gray, and the remaining ones in white. We weight each
   vertex by the corresponding rate $q_i$.
 \item We describe the cover $\ell \colon F_{q_i} \to B_{q_i}$ and $B_{q_i}$  and
   on the right we draw  the new pieces of the carrousel section. To simplify the figures (but not the data), we will rather draw the sections  of the $A\!X_{q_i}$-pieces. The
    $ \Delta$-pieces are in gray.  We add arrows on the graph corresponding to
     the strict transform of the corresponding branches of the
     polar. 
\end{enumerate}
We now carry out these steps. We use the indexing $v_1,\ldots,v_6$ of the vertices of $\Gamma$ introduced in Example  \ref{ex:D5-resolution}.
 \begin{itemize}
 \item The characterization of Proposition \ref{prop:finding q1} leads to $q_1=5/2$. We get $X_{5/2} =\pi({\mathcal N}(E_6))$ up to collars. Using
   the multiplicities of the generic linear form (determined by
   \eqref{it:multiplicity and Zmax} of Theorem \ref{th:invariants from
     geometry}) we obtain that $F_{5/2} = F \cap \pi({\mathcal N}(E_6)) $
   consists of two discs.  Then $B_{5/2}=\ell(F_{5/2})$ also
   consists of two discs. By the Hurwitz formula, each of them contains
   one branch point of $\ell$.
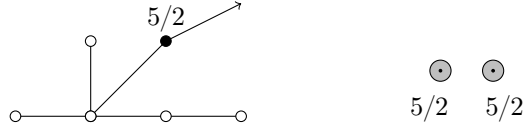
   \begin{center}
\begin{tikzpicture}
 
   \draw[thin ](0,0)--(3,0);
\draw[thin ](1,0)--(1,1);
\draw[thin ](1,0)--(2,1);
        
\draw[thin,>-stealth,->](2,1)--+(1,0.5);

  \draw[fill=white   ] (1,0)circle(2pt);
  \draw[fill=white  ] (3,0)circle(2pt);
     \draw[fill =white ] (1,1)circle(2pt);      
  \draw[fill=white   ] (0,0)circle(2pt);
   \draw[fill=white   ] (2,0)circle(2pt); 
   \draw[fill=white ] (1,0)circle(2pt); 
     \draw[fill=black  ] (2,1)circle(2pt);

  \node(a)at(2,1.3){   $5/2$};


 \draw[ fill=lightgray] ( 0:6)+(120:0.7)circle(4pt);
  \draw[fill=black] ( 0:6)+(120:0.7)circle(0.5pt);

     \node(a)at(5.5,0.1){  $5/2$};
      \node(a)at(6.5,0.1){  $5/2$};

 \draw[ fill=lightgray] ( 0:6)+(60:0.7)circle(4pt);
  \draw[fill=black] ( 0:6)+(60:0.7)circle(0.5pt);

 \end{tikzpicture} 
   \end{center}
\item $q_2=2$ is the next rate appearing,   and
  $X_{2} = \pi({\mathcal N}(E_4))$. $F_{2}$ is a disk as well as
  $B_2$. Then it creates another carrousel section disk. There is one
  branching point inside $B_2$.
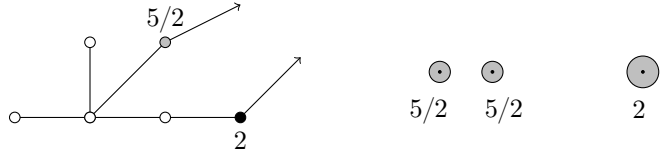
\begin{center}
\begin{tikzpicture}
    \draw[thin ](0,0)--(3,0);
\draw[thin ](1,0)--(1,1);
\draw[thin ](1,0)--(2,1);
      
\draw[thin,>-stealth,->](3,0)--+(0.8,0.8);
        
\draw[thin,>-stealth,->](2,1)--+(1,0.5);

  \draw[fill=white   ] (1,0)circle(2pt);
  \draw[fill=black  ] (3,0)circle(2pt);
     \draw[fill =white ] (1,1)circle(2pt);      
  \draw[fill=white   ] (0,0)circle(2pt);
   \draw[fill=white   ] (2,0)circle(2pt);
   \draw[fill=white ] (1,0)circle(2pt); 
     \draw[fill=lightgray ] (2,1)circle(2pt);

  \node(a)at(2,1.3){   $5/2$};     
 \node(a)at(3,-0.3){   $2$};
     

    \draw[ fill=lightgray] ( 0:6)+(120:0.7)circle(4pt);
  \draw[fill=black] ( 0:6)+(120:0.7)circle(0.5pt);

     \node(a)at(5.5,0.1){  $5/2$};
      \node(a)at(6.5,0.1){  $5/2$};
     
 \draw[ fill=lightgray] ( 0:6)+(60:0.7)circle(4pt);
 \draw[fill=black] ( 0:6)+(60:0.7)circle(0.5pt);

 \node(a)at(8.3,0.1){  $2$};
     
      \draw[ fill=lightgray] ( 0:8)+(60:0.7) circle(6pt);
 \draw[fill=black] ( 0:8)+(60:0.7) circle(0.5pt);

 \end{tikzpicture} 
\end{center}
\item Next is $q_3=3/2$. $F_{3/2}$ is a sphere with 4 holes. Two of the boundary circles are common boundary with $F_{1}$, the two others   map on the same circle   $\partial B_{5/2}$.  A simple observation of the multiplicities of the generic linear form shows that there $X_{3/2}$ does not overlap $X_{5/2}$ so $B_{3/2} = \ell(F_{3/2})$ is connected with one outer boundary component and inner boundary glued to $B_{5/2}$.
  
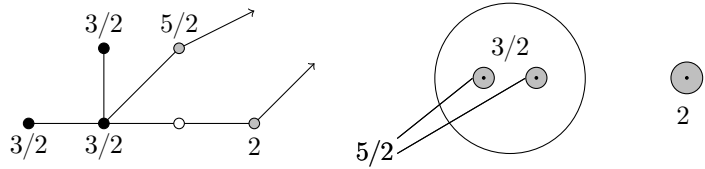
\begin{center}
\begin{tikzpicture}
 
   \draw[thin ](0,0)--(3,0);
\draw[thin ](1,0)--(1,1);
\draw[thin ](1,0)--(2,1);
      
\draw[thin,>-stealth,->](3,0)--+(0.8,0.8);
        
\draw[thin,>-stealth,->](2,1)--+(1,0.5);

  \draw[fill=black   ] (1,0)circle(2pt);
  \draw[fill=lightgray  ] (3,0)circle(2pt);
     \draw[fill =black  ] (1,1)circle(2pt);      
  \draw[fill=black    ] (0,0)circle(2pt);
   \draw[fill=white   ] (2,0)circle(2pt);

     \draw[fill=lightgray  ] (2,1)circle(2pt);
  
\node(a)at(0,-0.3){   $3/2$};
\node(a)at(1,-0.3){   $3/2$};
 \node(a)at(3,-0.3){   $2$};
 \node(a)at(1,1.3){   $3/2$};
  \node(a)at(2,1.3){   $5/2$}; 
  
  \begin{scope}[xshift=0.4cm]
    \draw[  ] ( 0:6)+(90:0.6)circle(1);


    \draw[ fill=lightgray] ( 0:6)+(120:0.7)circle(4pt);
  \draw[fill=black] ( 0:6)+(120:0.7)circle(0.5pt);

\draw[ fill=lightgray] ( 0:6)+(60:0.7)circle(4pt);
 \draw[fill=black] ( 0:6)+(60:0.7)circle(0.5pt);
  
    \node(a)at(4.2,-0.4){  $5/2$};
       \draw(4.5,-0.2)--(5.5,0.6);
        \draw(4.5,-0.4)--(6.2,0.6);
  
  \node(a)at(8.3,0.1){  $2$};
     
      \draw[ fill=lightgray] ( 0:8)+(60:0.7) circle(6pt);
  \draw[fill=black] ( 0:8)+(60:0.7) circle(0.5pt);
    
    \node(a)at(6,1){$3/2$};

      \node(a)at(4.2,-0.4){  $5/2$};
       \draw(4.5,-0.2)--(5.5,0.6);
        \draw(4.5,-0.4)--(6.2,0.6);

\end{scope} \end{tikzpicture} 
\end{center}
\item Last  is $q_4=1$, which corresponds to the thick piece.  
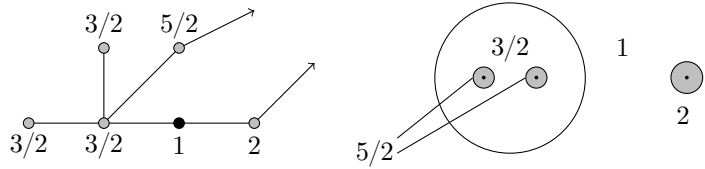
\begin{center}
\begin{tikzpicture}
 
   \draw[thin ](0,0)--(3,0);
\draw[thin ](1,0)--(1,1);
\draw[thin ](1,0)--(2,1);
      
\draw[thin,>-stealth,->](3,0)--+(0.8,0.8);
        
\draw[thin,>-stealth,->](2,1)--+(1,0.5);

  \draw[fill=lightgray   ] (1,0)circle(2pt);
  \draw[fill=lightgray  ] (3,0)circle(2pt);
     \draw[fill =lightgray  ] (1,1)circle(2pt);      
  \draw[fill=lightgray    ] (0,0)circle(2pt);
   \draw[fill=black   ] (2,0)circle(2pt);

     \draw[fill=lightgray  ] (2,1)circle(2pt);
  
\node(a)at(0,-0.3){   $3/2$};
\node(a)at(1,-0.3){   $3/2$};
\node(a)at(2,-0.3){   $1$};
 \node(a)at(3,-0.3){   $2$};
 \node(a)at(1,1.3){   $3/2$};
  \node(a)at(2,1.3){   $5/2$}; 
  
  \begin{scope}[xshift=0.4cm]
    \draw[  ] ( 0:6)+(90:0.6)circle(1);
   
    \node(a)at(6,1){$3/2$};
     \node(a)at(7.5,1){$1$};
    
  \draw[ fill=lightgray] ( 0:6)+(120:0.7)circle(4pt);
  \draw[fill=black] ( 0:6)+(120:0.7)circle(0.5pt);

 \draw[ fill=lightgray] ( 0:6)+(60:0.7)circle(4pt);
  \draw[fill=black] ( 0:6)+(60:0.7)circle(0.5pt);
  
    \node(a)at(4.2,-0.4){  $5/2$};
       \draw(4.5,-0.2)--(5.5,0.6);
        \draw(4.5,-0.4)--(6.2,0.6);
  
   \node(a)at(8.3,0.1){  $2$};
     
      \draw[ fill=lightgray] ( 0:8)+(60:0.7) circle(6pt);
  \draw[fill=black] ( 0:8)+(60:0.7) circle(0.5pt);
  \end{scope}
 \end{tikzpicture} 
\end{center}
\end{itemize}

We read in the carrousel section the topology of the discriminant
curve $\Delta$ from this carrousel sections: one smooth branch and one
transversal cusp with Puiseux exponent 3/2.
\end{example}

\begin{example}\label{ex:very big2}
  We now describe how one would reconstruct the decomposition $X =
  \bigcup X_{q_i}$ plus intermediate A-pieces and the carrousel section from the outer Lipschitz
  geometry for Example \ref{ex:very big0} (see also Example \ref{ex:very
    big1}). \red{We will again simplify the figures by drawing only the sections of $A\!X_{q_i}$-pieces.} The sequence of rates $q_i$  appears in the following order:
$$q_1=7/5, q_2=3/2, q_3=5/4, q_4=6/5, q_5=1\,.$$
\noindent {\bf Overlapping.} The inverse image by $\ell$ of the
  carrousel section of $\ell (X_{5/4})$ consists of two annuli
  (according to the multiplicities of the generic linear form) and
  each of them double-covers the disk image.  Then the restriction of
  $\ell$ to $X_{5/4}$ has degree $4$ while the total degree of the
  cover $\ell |_{X}$ is 6. Therefore either $X_{6/5}$ or $X_1$
  overlaps $X_{5/4}$. Moreover, the restriction of $\ell$ to $X_{6/5}$
  also has degree $4$ according to the multiplicity $10$ at the
  corresponding vertices. Therefore $X_1$ overlaps $X_{5/4}$.

We now show  step by step the construction of the carrousel section.  Here again we represent the sections of the $A\!X_{q_i}$-pieces rather than that of the $X_{q_i}$ and intermediate annular pieces. 

\begin{itemize}
\item $q_1=7/5$.  $X_{7/5}$ consists of two polar pieces. We represent
  the two $\Delta$-pieces with two different gray colors in the carrousel section.
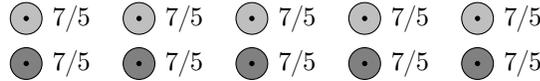
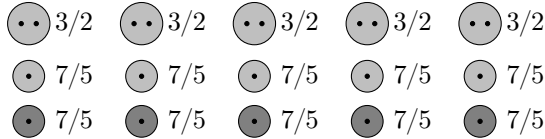
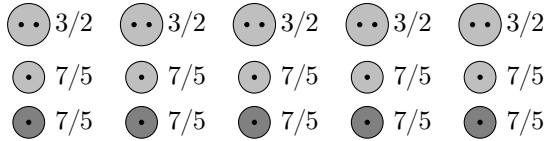
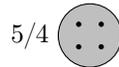
\begin{center}
\begin{tikzpicture}  
 
   \draw[ fill=lightgray] (2,0)circle(6pt);
   \draw[ fill=black] (2,0)circle(.8pt);
   \node(a)at(2.6,0){$7/5$};

   \draw[ fill=lightgray] (3.5,0)circle(6pt);
   \draw[ fill=black] (3.5,0)circle(.8pt);
   \node(a)at(4.1,0){$7/5$};

   \draw[ fill=lightgray] (5,0)circle(6pt);
   \draw[ fill=black] (5,0)circle(.8pt);
   \node(a)at(5.6,0){$7/5$};
   
   \draw[ fill=lightgray] (6.5,0)circle(6pt);
   \draw[ fill=black] (6.5,0)circle(.8pt);
   \node(a)at(7.1,0){$7/5$};
     
   \draw[ fill=lightgray] (8,0)circle(6pt);
   \draw[ fill=black] (8,0)circle(.8pt);
   \node(a)at(8.6,0){$7/5$};

 \begin{scope}[yshift=-0.6cm]
 
   \draw[ fill=gray] (2,0)circle(6pt);
   \draw[ fill=black] (2,0)circle(.8pt);
   \node(a)at(2.6,0){$7/5$};

   \draw[ fill=gray] (3.5,0)circle(6pt);
   \draw[ fill=black] (3.5,0)circle(.8pt);
   \node(a)at(4.1,0){$7/5$};

   \draw[ fill=gray] (5,0)circle(6pt);
   \draw[ fill=black] (5,0)circle(.8pt);
   \node(a)at(5.6,0){$7/5$};
   
   \draw[ fill=gray] (6.5,0)circle(6pt);
   \draw[ fill=black] (6.5,0)circle(.8pt);
   \node(a)at(7.1,0){$7/5$};
     
   \draw[ fill=gray] (8,0)circle(6pt);
   \draw[ fill=black] (8,0)circle(.8pt);
   \node(a)at(8.6,0){$7/5$};
  
\end{scope}
\end{tikzpicture}  
\end{center}
\item $q_2=3/2$
\begin{center}
\begin{tikzpicture}  
      
   \draw[fill=lightgray] (2,0)circle(8pt);
   \draw[fill=black] ( 2.1,0)circle(.8pt);
   \draw[fill=black] ( 1.9,0)circle(.8pt);
   \node(a)at(2.6,0){$3/2$};
   
  \begin{scope} [xshift=1.5cm] 
   \draw[ fill=lightgray] (2,0)circle(8pt);
   \draw[fill=black] ( 2.1,0)circle(.8pt);
   \draw[fill=black] ( 1.9,0)circle(.8pt);
   \node(a)at(2.6,0){$3/2$};
  \end{scope}
  
  \begin{scope} [xshift=3cm] 
   \draw[ fill=lightgray] (2,0)circle(8pt);
   \draw[fill=black] ( 2.1,0)circle(.8pt);
   \draw[fill=black] ( 1.9,0)circle(.8pt);
   \node(a)at(2.6,0){$3/2$};
  \end{scope}
  
  \begin{scope} [xshift=4.5cm] 
   \draw[fill=lightgray] (2,0)circle(8pt);
   \draw[fill=black] ( 2.1,0)circle(.8pt);
   \draw[fill=black] ( 1.9,0)circle(.8pt);
   \node(a)at(2.6,0){$3/2$};
  \end{scope}
  
  \begin{scope} [xshift=6cm] 
   \draw[fill=lightgray] (2,0)circle(8pt);
   \draw[fill=black] ( 2.1,0)circle(.8pt);
   \draw[fill=black] ( 1.9,0)circle(.8pt);
   \node(a)at(2.6,0){$3/2$};
  \end{scope}
  
 \begin{scope}[yshift=-0.7cm]

   \draw[ fill=lightgray] (2,0)circle(6pt);
   \draw[ fill=black] (2,0)circle(.8pt);
   \node(a)at(2.6,0){$7/5$};

   \draw[ fill=lightgray] (3.5,0)circle(6pt);
   \draw[ fill=black] (3.5,0)circle(.8pt);
   \node(a)at(4.1,0){$7/5$};

   \draw[ fill=lightgray] (5,0)circle(6pt);
   \draw[ fill=black] (5,0)circle(.8pt);
   \node(a)at(5.6,0){$7/5$};
   
   \draw[ fill=lightgray] (6.5,0)circle(6pt);
   \draw[ fill=black] (6.5,0)circle(.8pt);
   \node(a)at(7.1,0){$7/5$};
     
   \draw[ fill=lightgray] (8,0)circle(6pt);
   \draw[ fill=black] (8,0)circle(.8pt);
   \node(a)at(8.6,0){$7/5$};

 \begin{scope}[yshift=-0.6cm]
 
   \draw[ fill=gray] (2,0)circle(6pt);
   \draw[ fill=black] (2,0)circle(.8pt);
   \node(a)at(2.6,0){$7/5$};

   \draw[ fill=gray] (3.5,0)circle(6pt);
   \draw[ fill=black] (3.5,0)circle(.8pt);
   \node(a)at(4.1,0){$7/5$};

   \draw[ fill=gray] (5,0)circle(6pt);
   \draw[ fill=black] (5,0)circle(.8pt);
   \node(a)at(5.6,0){$7/5$};
   
   \draw[ fill=gray] (6.5,0)circle(6pt);
   \draw[ fill=black] (6.5,0)circle(.8pt);
   \node(a)at(7.1,0){$7/5$};
     
   \draw[ fill=gray] (8,0)circle(6pt);
   \draw[ fill=black] (8,0)circle(.8pt);
   \node(a)at(8.6,0){$7/5$};
 
 \end{scope}
\end{scope}

\end{tikzpicture}  
\end{center}
\item $q_3=5/4$
\begin{center}
\begin{tikzpicture} 

   \draw[fill=lightgray] (2,0)circle(8pt);
   \draw[fill=black] ( 2.1,0)circle(.8pt);
   \draw[fill=black] ( 1.9,0)circle(.8pt);
   \node(a)at(2.6,0){$3/2$};
   
  \begin{scope} [xshift=1.5cm] 
   \draw[ fill=lightgray] (2,0)circle(8pt);
   \draw[fill=black] ( 2.1,0)circle(.8pt);
   \draw[fill=black] ( 1.9,0)circle(.8pt);
   \node(a)at(2.6,0){$3/2$};
  \end{scope}
  
  \begin{scope} [xshift=3cm] 
   \draw[ fill=lightgray] (2,0)circle(8pt);
   \draw[fill=black] ( 2.1,0)circle(.8pt);
   \draw[fill=black] ( 1.9,0)circle(.8pt);
   \node(a)at(2.6,0){$3/2$};
  \end{scope}
  
  \begin{scope} [xshift=4.5cm] 
   \draw[fill=lightgray] (2,0)circle(8pt);
   \draw[fill=black] ( 2.1,0)circle(.8pt);
   \draw[fill=black] ( 1.9,0)circle(.8pt);
   \node(a)at(2.6,0){$3/2$};
  \end{scope}
  
  \begin{scope} [xshift=6cm] 
   \draw[fill=lightgray] (2,0)circle(8pt);
   \draw[fill=black] ( 2.1,0)circle(.8pt);
   \draw[fill=black] ( 1.9,0)circle(.8pt);
   \node(a)at(2.6,0){$3/2$};
  \end{scope}
  
 \begin{scope}[yshift=-0.7cm]

   \draw[ fill=lightgray] (2,0)circle(6pt);
   \draw[ fill=black] (2,0)circle(.8pt);
   \node(a)at(2.6,0){$7/5$};

   \draw[ fill=lightgray] (3.5,0)circle(6pt);
   \draw[ fill=black] (3.5,0)circle(.8pt);
   \node(a)at(4.1,0){$7/5$};

   \draw[ fill=lightgray] (5,0)circle(6pt);
   \draw[ fill=black] (5,0)circle(.8pt);
   \node(a)at(5.6,0){$7/5$};
   
   \draw[ fill=lightgray] (6.5,0)circle(6pt);
   \draw[ fill=black] (6.5,0)circle(.8pt);
   \node(a)at(7.1,0){$7/5$};
     
   \draw[ fill=lightgray] (8,0)circle(6pt);
   \draw[ fill=black] (8,0)circle(.8pt);
   \node(a)at(8.6,0){$7/5$};

 \begin{scope}[yshift=-0.6cm]
 
   \draw[ fill=gray] (2,0)circle(6pt);
   \draw[ fill=black] (2,0)circle(.8pt);
   \node(a)at(2.6,0){$7/5$};

   \draw[ fill=gray] (3.5,0)circle(6pt);
   \draw[ fill=black] (3.5,0)circle(.8pt);
   \node(a)at(4.1,0){$7/5$};

   \draw[ fill=gray] (5,0)circle(6pt);
   \draw[ fill=black] (5,0)circle(.8pt);
   \node(a)at(5.6,0){$7/5$};
   
   \draw[ fill=gray] (6.5,0)circle(6pt);
   \draw[ fill=black] (6.5,0)circle(.8pt);
   \node(a)at(7.1,0){$7/5$};
     
   \draw[ fill=gray] (8,0)circle(6pt);
   \draw[ fill=black] (8,0)circle(.8pt);
   \node(a)at(8.6,0){$7/5$};
 
 \end{scope}
\end{scope}
    
 \begin{scope}[yshift=-2.1cm]

 \begin{scope}[xshift=-2cm]  
  \draw[fill=lightgray] (5,-0.4)circle(12pt);
  \draw[fill=black] ( 4.85,-0.25)circle(.8pt);
  \draw[fill=black] ( 4.85,-0.55)circle(.8pt);
  \draw[fill=black] ( 5.15,-0.25)circle(.8pt);
  \draw[fill=black] ( 5.15,-0.55)circle(.8pt);
   \node(a)at(4.2,-0.4){$5/4$};
 \end{scope}  
\end{scope}  
     
\end{tikzpicture}  
\end{center}

\item $q_4=6/5$ and $q_5=1$
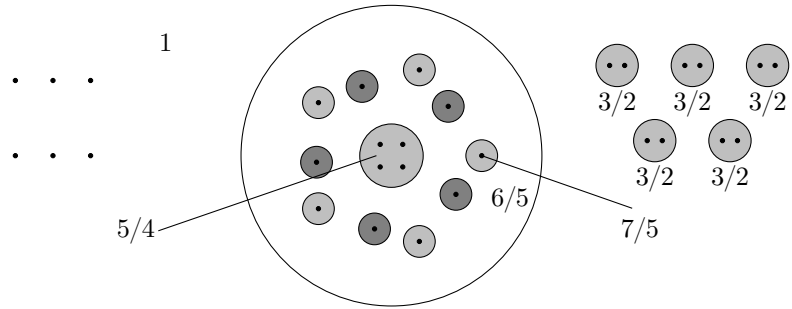
\begin{center}
\begin{tikzpicture}  
\begin{scope}[yshift=-4.5cm]
 \begin{scope}[xshift=6cm, yshift=-1.5cm]
   
  \draw[fill=lightgray] (2,0)circle(8pt);
  \draw[fill=black] ( 2.1,0)circle(.8pt);
  \draw[fill=black] ( 1.9,0)circle(.8pt);
  \node(a)at(2,-0.5){$3/2$};
   
  \begin{scope} [xshift=1cm] 
   \draw[ fill=lightgray](2,0)circle(8pt);
   \draw[fill=black] ( 2.1,0)circle(0.8pt);
   \draw[fill=black] ( 1.9,0)circle(0.8pt);
   \node(a)at(2,-0.5){$3/2$};
  \end{scope}
  
  \begin{scope} [xshift=2cm] 
   \draw[fill=lightgray] (2,0)circle(8pt);
   \draw[fill=black] ( 2.1,0)circle(0.8pt);
   \draw[fill=black] ( 1.9,0)circle(0.8pt);
   \node(a)at(2,-0.5){$3/2$};
  \end{scope}
  
  \begin{scope} [xshift=0.5cm, yshift=-1cm] 
   \draw[fill=lightgray] (2,0)circle(8pt);
   \draw[fill=black] ( 2.1,0)circle(0.8pt);
   \draw[fill=black] ( 1.9,0)circle(0.8pt);
   \node(a)at(2,-0.5){$3/2$};
  \end{scope}
  
  \begin{scope} [xshift=1.5cm, yshift=-1cm] 
   \draw[fill=lightgray] (2,0)circle(8pt);
   \draw[fill=black] ( 2.1,0)circle(0.8pt);
   \draw[fill=black] ( 1.9,0)circle(0.8pt);
   \node(a)at(2,-0.5){$3/2$};
  \end{scope}
 \end{scope}

 \begin{scope}[yshift=-2.7cm]
   \draw[ ] (0:5) circle(2cm);
 
    \draw[fill=lightgray] (5,0)circle(12pt);
    \draw[fill=black] ( 4.85,-0.15)circle(0.8pt);
  
    \draw[fill=black] ( 4.85,0.15)circle(0.8pt);
    
    \draw[fill=black] ( 5.15,-0.15)circle(0.8pt);
      
    \draw[fill=black] ( 5.15,0.15)circle(0.8pt);
  
     \draw[ fill=lightgray] (0:5)+(0:1.2)circle(6pt);
       \draw[ fill=black] (0:5)+(0:1.2)circle(.8pt);
  
    \draw[fill=lightgray ] (0:5)+(72:1.2)circle(6pt);
       \draw[ fill=black] (0:5)+(72:1.2)circle(.8pt);
  
    \draw[fill=lightgray ] (0:5)+(144:1.2)circle(6pt);
       \draw[ fill=black] (0:5)+(144:1.2)circle(.8pt);

    \draw[fill=lightgray ] (0:5)+(216:1.2)circle(6pt);
       \draw[ fill=black] (0:5)+(216:1.2)circle(.8pt);
  
  \draw[fill=lightgray ] (0:5)+(-72:1.2)circle(6pt);
       \draw[ fill=black] (0:5)+(-72:1.2)circle(.8pt);

    \draw[fill=gray] (0:5)+(185:1.0)circle(6pt);
       \draw[ fill=black] (0:5)+(185:1.0)circle(.8pt);
  
    \draw[fill=gray] (0:5)+(257:1.0)circle(6pt);
       \draw[fill=black] (0:5)+(257:1.0)circle(.8pt);
  
   \draw[fill=gray ] (0:5)+(-31:1.0)circle(6pt);
       \draw[ fill=black] (0:5)+(-31:1.0)circle(.8pt);

      \draw[fill=gray ] (0:5)+(41:1.0)circle(6pt);
       \draw[ fill=black] (0:5)+(41:1.0)circle(.8pt);

     \draw[fill=gray ] (0:5)+(-247:1.0)circle(6pt);
       \draw[fill=black] (0:5)+(-247:1.0)circle(.8pt);
    
     \node(a)at(-5:6.6){$6/5$};
          \node(a)at(1.6,-1){${\small 5/4}$};
          \draw(1.9,-1)--(4.8,0);
 
          \node(a)at(8.3,-1){${\small 7/5}$};
          \draw(8.2,-0.7)--(6.2,0);
  

   \draw[fill=black] (1,1)circle(.8pt);
           \draw[fill=black] (0.5,1)circle(.8pt);
          \draw[fill=black] (0,1)circle(.8pt);
             \draw[fill=black] (1,0)circle(.8pt);
           \draw[fill=black] (0.5,0)circle(.8pt);
         \draw[fill=black] (0,0)circle(.8pt);
       
     \node(a)at(2,1.5){${\small 1}$}; 
\end{scope}  
\end{scope}
\end{tikzpicture}  
\end{center}
\end{itemize}
\noindent This completes the recovery of the carrousel decomposition
of Example \ref{ex:very big0}.
\end{example}

  \begin{example}\label{ex:vvb} 
\red{Our final example exhibits some $B$-pieces of type $A(q,q)$.  It is
    obtained by modifying the previous example (discussed also in
    \ref{ex:very big1}, \ref{ex:very big0} and in \cite{BNP}), which
    had equation $ (zx^2+y^3)(x^3+zy^2)+z^7= 0 $. We use instead the
    equation
    $$(zx^2+y^3)(x^5+z^3y^2)+z^9= 0$$
    which replaces one the two $B(6/5)$-pieces of the geometric
    decomposition of the previous example by a $B(8/7)$-piece.  Each
    of these $B$-pieces has non-trivial topology, and by Proposition
    \ref{prop:Anne doesn't like this} each of $X_{6/5}$ and $X_{8/7}$ has a second piece which is a  $B$-pieces of type $A(6/5,6/5)$ resp.\
      $A(8/7,8/7)$.

      The first figure 
      shows the resolution graph of a general linear system of
      hyperplane sections for this example, followed by
      the graph of the resolution which resolves the base points of
      the generic polar, for which an additional blow-up was needed
      (the intersection numbers at the bottom right are $-3$ and
      $-1$). The rates associated with each vertex are also shown.
      The second figure gives a picture of the complete carrousel
      section. We use the complete carrousel here to make the
      $A$-pieces evident.}
\begin{figure}[h]
\begin{center}
\begin{tikzpicture}
\draw[thin ](-1.5,2.5)--(1.5,2.5);

\draw[thin ](-2.5,0)--(-1.5,1);
\draw[thin ](-.5,0)--(-1.5,1);
\draw[thin ](-.5,0)--(.5,.5);
\draw[thin ](.5,.5)--(1.5,1);
\draw[thin ](1.5,1)--(2.5,0);
\draw[thin ](-1.5,1)--(-1.5,2.5);
\draw[thin ](1.5,1)--(1.5,2.5);
     
\draw[fill=white ] (-1.5,1)circle(2pt);
\draw[fill=white ] (1.5,1)circle(2pt);
\draw[fill=white ] (.5,.5)circle(2pt);
\draw[fill=white ] (-.5,0)circle(2pt);
\draw[fill=white] (2.5,0)circle(2pt);
\draw[fill=white] (-2.5,0)circle(2pt);

\draw[thin] (-1.5,2.5)..controls (-0.5,3.75) and (0.5,3.75)..(1.5,2.5);
\draw[thin] (-1.5,2.5)..controls (-0.5,3.5) and (0.5,3.5)..(1.5,2.5);
\draw[thin] (-1.5,2.5)..controls (-0.5,3.25) and (0.5,3.25)..(1.5,2.5);
\draw[thin] (-1.5,2.5)..controls (-0.5,3) and (0.5,3)..(1.5,2.5);
\draw[thin] (-1.5,2.5)..controls (-0.5,2.75) and (0.5,2.75)..(1.5,2.5);
\draw[thin] (-1.5,2.5)..controls (-0.5,2.25) and (0.5,2.25)..(1.5,2.5);
\draw[thin] (-1.5,2.5)..controls (-0.5,2) and (0.5,2)..(1.5,2.5);
\draw[thin] (-1.5,2.5)..controls (-0.5,1.75) and (0.5,1.75)..(1.5,2.5);
\draw[thin] (-1.5,2.5)..controls (-0.5,1.5) and (0.5,1.5)..(1.5,2.5);
\draw[thin] (-1.5,2.5)..controls (-0.5,1.25) and (0.5,1.25)..(1.5,2.5);

\node(a)at(-2.5,-0.35){$-2$};
\node(a)at(-2.9,0){$(5)$};

\node(a)at(-.5,-0.3){$-4$};
\node(a)at(-.5,0.4){$(4)$};

\node(a)at(.5,0.2){$-3$};
\node(a)at(.5,0.75){$(6)$};

\node(a)at(2.5,-0.35){$-2$};
\node(a)at(2.9,0){$(7)$};

\node(a)at(-1.5,0.65){$-1$};
\node(a)at(-2,1){$(10)$};

\node(a)at(1.5,0.65){$-1$};
\node(a)at(2,1){$(14)$};

\node(a)at(-0.4,3.6){$-1$};
\node(a)at(-0.4,1.4){$-1$};
\node(a)at(0.4,3.6){$(2)$};
\node(a)at(0.4,1.4){$(2)$};
\node(a)at(-1.6,2.9){$(1)$};
\node(a)at(1.6,2.9){$(1)$};
\node(a)at(-1.85,2.15){$-35$};
\node(a)at(1.8,2.15){$-41$};
 
  \draw[thin,>-stealth,->](1.5,2.5)--+(1.2,0.4);
  \draw[thin,>-stealth,->](1.5,2.5)--+(1.3,0.2);
  \draw[thin,>-stealth,->](1.5,2.5)--+(1.3,0);
  \draw[thin,>-stealth,->](1.5,2.5)--+(1.3,-0.2);
  \draw[thin,>-stealth,->](1.5,2.5)--+(1.2,-0.4);
        
  \draw[thin,>-stealth,->](-1.5,2.5)--+(-1.2,0.4);
  \draw[thin,>-stealth,->](-1.5,2.5)--+(-1.3,0);
  \draw[thin,>-stealth,->](-1.5,2.5)--+(-1.2,-0.4);

   \draw[ fill=white] (0,3.44)circle(2pt);
   \draw[ fill=white] (0,3.25)circle(2pt);
   \draw[ fill=white] (0,3.05)circle(2pt);
   \draw[ fill=white] (0,2.86)circle(2pt);
   \draw[ fill=white] (0,2.68)circle(2pt);
   \draw[ fill=white] (0,2.5)circle(2pt);
   \draw[ fill=white] (0,2.31)circle(2pt);
   \draw[ fill=white] (0,2.12)circle(2pt);
   \draw[ fill=white] (0,1.94)circle(2pt);
   \draw[ fill=white] (0,1.75)circle(2pt);
   \draw[ fill=white] (0,1.56)circle(2pt);
  \draw[ fill=white ] (1.5,2.5)circle(2pt);
     \draw[ fill=white] (-1.5,2.5)circle(2pt);
  \end{tikzpicture} 
\begin{tikzpicture}
\draw[thin ](-1.5,2.5)--(1.5,2.5);

\draw[thin ](-2.5,0)--(-1.5,1);
\draw[thin ](-.5,0)--(-1.5,1);
\draw[thin ](-.5,0)--(.5,.5);
\draw[thin ](.5,.5)--(1.5,1);
\draw[thin ](1.5,1)--(2.5,0);
\draw[thin ](-1.5,1)--(-1.5,2.5);
\draw[thin ](1.5,1)--(1.5,2.5);
\draw[thin ](2.5,0)--(3.0,-0.8);

\draw[thin] (-1.5,2.5)..controls (-0.5,3.75) and (0.5,3.75)..(1.5,2.5);
\draw[thin] (-1.5,2.5)..controls (-0.5,3.5) and (0.5,3.5)..(1.5,2.5);
\draw[thin] (-1.5,2.5)..controls (-0.5,3.25) and (0.5,3.25)..(1.5,2.5);
\draw[thin] (-1.5,2.5)..controls (-0.5,3) and (0.5,3)..(1.5,2.5);
\draw[thin] (-1.5,2.5)..controls (-0.5,2.75) and (0.5,2.75)..(1.5,2.5);
\draw[thin] (-1.5,2.5)..controls (-0.5,2.25) and (0.5,2.25)..(1.5,2.5);
\draw[thin] (-1.5,2.5)..controls (-0.5,2) and (0.5,2)..(1.5,2.5);
\draw[thin] (-1.5,2.5)..controls (-0.5,1.75) and (0.5,1.75)..(1.5,2.5);
\draw[thin] (-1.5,2.5)..controls (-0.5,1.5) and (0.5,1.5)..(1.5,2.5);
\draw[thin] (-1.5,2.5)..controls (-0.5,1.25) and (0.5,1.25)..(1.5,2.5);

\node(a)at(-2.4,-0.35){$\frac75$};
\node(a)at(-2.9,0){$(39)$};

\node(a)at(-.5,-0.35){$\frac54$};
\node(a)at(-.5,0.4){$(31)$};

\node(a)at(.5,0.75){$(46)$};

\node(a)at(2.9,0){$(54)$};

\node(a)at(2.7,-1){$\frac{10}7$};
\node(a)at(3.4,-1){$(55)$};

\node(a)at(-1.5,0.65){$\frac65$};
\node(a)at(-2,1){$(77)$};

\node(a)at(1.5,0.65){$\frac87$};
\node(a)at(2,1){$(107)$};

\node(a)at(-0.4,3.7){$\frac32$};
\node(a)at(-0.4,1.3){$\frac32$};
\node(a)at(0.2,3.7){$(15)$};
\node(a)at(0.2,1.3){$(15)$};
\node(a)at(-1.6,2.9){$(7)$};
\node(a)at(1.6,2.9){$(7)$};
\node(a)at(-1.7,2.15){$1$};
\node(a)at(1.7,2.15){$1$};
 
  \draw[thin,>-stealth,->](1.5,2.5)--+(1.2,0.6);
  \draw[thin,>-stealth,->](1.5,2.5)--+(1.2,0.5);
  \draw[thin,>-stealth,->](1.5,2.5)--+(1.22,0.4);
  \draw[thin,>-stealth,->](1.5,2.5)--+(1.22,0.3);
  \draw[thin,>-stealth,->](1.5,2.5)--+(1.24,0.2);
  \draw[thin,>-stealth,->](1.5,2.5)--+(1.24,0.1);
  \draw[thin,>-stealth,->](1.5,2.5)--+(1.25,0);
  \draw[thin,>-stealth,->](1.5,2.5)--+(1.25,-0.1);
  \draw[thin,>-stealth,->](1.5,2.5)--+(1.25,-0.2);
  \draw[thin,>-stealth,->](1.5,2.5)--+(1.24,-0.3);
  \draw[thin,>-stealth,->](1.5,2.5)--+(1.24,-0.4);
  \draw[thin,>-stealth,->](1.5,2.5)--+(1.22,-0.5);
  \draw[thin,>-stealth,->](1.5,2.5)--+(1.22,-0.6);
  \draw[thin,>-stealth,->](1.5,2.5)--+(1.2,-0.7);
  \draw[thin,>-stealth,->](1.5,2.5)--+(1.2,-0.8);
        
  \draw[thin,>-stealth,->](-1.5,2.5)--+(-1.2,0.4);
  \draw[thin,>-stealth,->](-1.5,2.5)--+(-1.3,0);
  \draw[thin,>-stealth,->](-1.5,2.5)--+(-1.2,-0.4);

  \draw[thin,>-stealth,->](-2.5,0)--+(-.8,-.6);
  \draw[thin,>-stealth,->](-.5,0)--+(.9,-.6);
  \draw[thin,>-stealth,->](-1.5,2.5)--+(-1.2,-0.4);

  \draw[thin,>-stealth,->](3.0,-0.8)--+(1.0,0.4);

\draw[fill=white ] (-1.5,1)circle(2pt);
\draw[fill=white ] (1.5,1)circle(2pt);
\draw[fill=white ] (.5,.5)circle(2pt);
\draw[fill=white ] (-.5,0)circle(2pt);
\draw[fill=white] (2.5,0)circle(2pt);
\draw[fill=white] (-2.5,0)circle(2pt);
\draw[fill=white] (3.0,-.8)circle(2pt);

   \draw[thin,>-stealth,->] (0,3.44)--+(.7,0.33);
   \draw[thin,>-stealth,->] (0,3.25)--+(.7,0.27);
   \draw[thin,>-stealth,->] (0,3.05)--+(.7,0.2);
   \draw[thin,>-stealth,->] (0,2.86)--+(.7,0.14);
   \draw[thin,>-stealth,->] (0,2.68)--+(.7,0.07);
   \draw[thin,>-stealth,->] (0,2.5)--+(.7,0.03);
   \draw[thin,>-stealth,->] (0,2.31)--+(.7,-0.07);
   \draw[thin,>-stealth,->] (0,2.12)--+(.7,-0.14);
   \draw[thin,>-stealth,->] (0,1.94)--+(.7,-0.20);
   \draw[thin,>-stealth,->] (0,1.75)--+(.7,-0.27);
   \draw[thin,>-stealth,->] (0,1.56)--+(.7,-0.33);

   \draw[ fill=white] (0,3.44)circle(2pt);
   \draw[ fill=white] (0,3.25)circle(2pt);
   \draw[ fill=white] (0,3.05)circle(2pt);
   \draw[ fill=white] (0,2.86)circle(2pt);
   \draw[ fill=white] (0,2.68)circle(2pt);
   \draw[ fill=white] (0,2.5)circle(2pt);
   \draw[ fill=white] (0,2.31)circle(2pt);
   \draw[ fill=white] (0,2.12)circle(2pt);
   \draw[ fill=white] (0,1.94)circle(2pt);
   \draw[ fill=white] (0,1.75)circle(2pt);
   \draw[ fill=white] (0,1.56)circle(2pt);

  \draw[ fill=white ] (1.5,2.5)circle(2pt);
  \draw[ fill=white] (-1.5,2.5)circle(2pt);
  \end{tikzpicture} 
  \end{center}
  \caption{Example \ref{ex:vvb}: Resolutions of hyperplane sections and basepoints of the polar.}
  \label{fig:vvb2}
\end{figure}
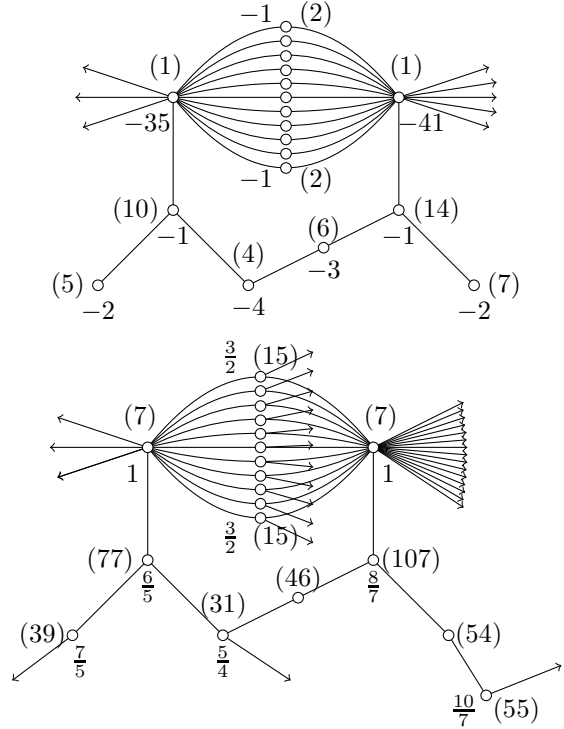

\begin{figure}[h]
\begin{center}
\begin{tikzpicture} 

\begin{scope}[xshift=7cm, yshift=-1cm]
   
 \draw[  ] (2,0)circle(10pt);
 \draw[ fill=lightgray] (2,0)circle(7pt);
 \draw[fill=white] ( 2.1,0)circle(2.3pt);
 \draw[fill=white] ( 1.9,0)circle(2.3pt);
 \draw[fill=black] ( 2.1,0)circle(.5pt);
 \draw[fill=black] ( 1.9,0)circle(.5pt);
 \node(a)at(2,-0.55){\small$3/2$};
    
 \begin{scope} [xshift=1cm] 
 \draw[  ] (2,0)circle(10pt);
 \draw[ fill=lightgray] (2,0)circle(7pt);
 \draw[fill=white] ( 2.1,0)circle(2.3pt);
 \draw[fill=white] ( 1.9,0)circle(2.3pt);
 \draw[fill=black] ( 2.1,0)circle(.5pt);
 \draw[fill=black] ( 1.9,0)circle(.5pt);
 \node(a)at(2,-0.55){\small$3/2$};
 \end{scope}
  
 \begin{scope} [xshift=2cm] 
 \draw[ ] (2,0)circle(10pt);
 \draw[ fill=lightgray] (2,0)circle(7pt);
 \draw[fill=white] ( 2.1,0)circle(2.3pt);
 \draw[fill=white] ( 1.9,0)circle(2.3pt);
 \draw[fill=black] ( 2.1,0)circle(.5pt);
 \draw[fill=black] ( 1.9,0)circle(.5pt);
 \node(a)at(2,-0.55){\small$3/2$};
 \end{scope}

 \begin{scope} [yshift=-1.2cm] 
 \draw[  ] (2,0)circle(10pt);
 \draw[ fill=lightgray] (2,0)circle(7pt);
 \draw[fill=white] ( 2.1,0)circle(2.3pt);
 \draw[fill=white] ( 1.9,0)circle(2.3pt);
 \draw[fill=black] ( 2.1,0)circle(.5pt);
 \draw[fill=black] ( 1.9,0)circle(.5pt);
 \node(a)at(2,-0.55){\small$3/2$};

  \begin{scope} [xshift=1cm] 
  \draw[  ] (2,0)circle(10pt);
  \draw[ fill=lightgray] (2,0)circle(7pt);
  \draw[fill=white] ( 2.1,0)circle(2.3pt);
  \draw[fill=white] ( 1.9,0)circle(2.3pt);
  \draw[fill=black] ( 2.1,0)circle(.5pt);
  \draw[fill=black] ( 1.9,0)circle(.5pt);
  \node(a)at(2,-0.55){\small$3/2$};
  \end{scope}

  \begin{scope} [xshift=2cm] 
  \draw[  ] (2,0)circle(10pt);
  \draw[ fill=lightgray] (2,0)circle(7pt);
  \draw[fill=white] ( 2.1,0)circle(2.3pt);
  \draw[fill=white] ( 1.9,0)circle(2.3pt);
  \draw[fill=black] ( 2.1,0)circle(.5pt);
  \draw[fill=black] ( 1.9,0)circle(.5pt);
\node(a)at(2,-0.55){\small$3/2$};
  \end{scope}
 \end{scope}
 \begin{scope} [yshift=-1.2cm] 
 \draw[  ] (2,0)circle(10pt);
 \draw[ fill=lightgray] (2,0)circle(7pt);
 \draw[fill=white] ( 2.1,0)circle(2.3pt);
 \draw[fill=white] ( 1.9,0)circle(2.3pt);
 \draw[fill=black] ( 2.1,0)circle(.5pt);
 \draw[fill=black] ( 1.9,0)circle(.5pt);
\node(a)at(2,-0.55){\small$3/2$};
   
  \begin{scope} [xshift=1cm] 
  \draw[  ] (2,0)circle(10pt);
  \draw[ fill=lightgray] (2,0)circle(7pt);
  \draw[fill=white] ( 2.1,0)circle(2.3pt);
  \draw[fill=white] ( 1.9,0)circle(2.3pt);
  \draw[fill=black] ( 2.1,0)circle(.5pt);
  \draw[fill=black] ( 1.9,0)circle(.5pt);
  \node(a)at(2,-0.55){\small$3/2$};
  \end{scope}

  \begin{scope} [xshift=2cm] 
  \draw[  ] (2,0)circle(10pt);
  \draw[ fill=lightgray] (2,0)circle(7pt);
  \draw[fill=white] ( 2.1,0)circle(2.3pt);
  \draw[fill=white] ( 1.9,0)circle(2.3pt);
  \draw[fill=black] ( 2.1,0)circle(.5pt);
  \draw[fill=black] ( 1.9,0)circle(.5pt);
  \node(a)at(2,-0.55){\small$3/2$};
  \end{scope}
 \end{scope}
 \begin{scope} [yshift=-2.4cm] 
 \draw[  ] (2,0)circle(10pt);
 \draw[ fill=lightgray] (2,0)circle(7pt);
 \draw[fill=white] ( 2.1,0)circle(2.3pt);
 \draw[fill=white] ( 1.9,0)circle(2.3pt);
 \draw[fill=black] ( 2.1,0)circle(.5pt);
 \draw[fill=black] ( 1.9,0)circle(.5pt);
 \node(a)at(2,-0.55){\small$3/2$};
   
  \begin{scope} [xshift=1cm] 
  \draw[  ] (2,0)circle(10pt);
  \draw[ fill=lightgray] (2,0)circle(7pt);
  \draw[fill=white] ( 2.1,0)circle(2.3pt);
  \draw[fill=white] ( 1.9,0)circle(2.3pt);
  \draw[fill=black] ( 2.1,0)circle(.5pt);
  \draw[fill=black] ( 1.9,0)circle(.5pt);
  \node(a)at(2,-0.55){\small$3/2$};
  \end{scope}

  \begin{scope} [xshift=2cm] 
  \draw[  ] (2,0)circle(10pt);
  \draw[ fill=lightgray] (2,0)circle(7pt);
  \draw[fill=white] ( 2.1,0)circle(2.3pt);
  \draw[fill=white] ( 1.9,0)circle(2.3pt);
  \draw[fill=black] ( 2.1,0)circle(.5pt);
  \draw[fill=black] ( 1.9,0)circle(.5pt);
  \node(a)at(2,-0.55){\small$3/2$};
  \end{scope}
 \end{scope}
 \begin{scope} [xshift=.5cm,yshift=-3.6cm] 
 \draw[  ] (2,0)circle(10pt);
 \draw[ fill=lightgray] (2,0)circle(7pt);
 \draw[fill=white] ( 2.1,0)circle(2.3pt);
 \draw[fill=white] ( 1.9,0)circle(2.3pt);
 \draw[fill=black] ( 2.1,0)circle(.5pt);
 \draw[fill=black] ( 1.9,0)circle(.5pt);
\node(a)at(2,-0.55){\small$3/2$};
   
  \begin{scope} [xshift=1cm] 
  \draw[  ] (2,0)circle(10pt);
  \draw[ fill=lightgray] (2,0)circle(7pt);
  \draw[fill=white] ( 2.1,0)circle(2.3pt);
  \draw[fill=white] ( 1.9,0)circle(2.3pt);
  \draw[fill=black] ( 2.1,0)circle(.5pt);
  \draw[fill=black] ( 1.9,0)circle(.5pt);
  \node(a)at(2,-0.55){\small$3/2$};
  \end{scope}
 \end{scope}
\end{scope}

\begin{scope}[yshift=-2.7cm]
\draw[ ] (0:5) circle(1.6cm);   
\draw[ ] (0:5) circle(2.2cm);
\draw[ ] (0:5) circle(2.8cm);  
\draw[ ] (0:5) circle(3cm);  
\draw[ ] (5,0)circle(24pt); 

    
\draw[ fill=lightgray] (5,0)circle(11pt);
     
\draw[fill=white] ( 4.85,-0.15)circle(2.5pt);
\draw[fill=white] ( 4.85,0.15)circle(2.5pt);
\draw[fill=white] ( 5.15,-0.15)circle(2.5pt);
\draw[fill=white] ( 5.15,0.15)circle(2.5pt);

\draw[fill=black] ( 4.85,-0.15)circle(.5pt);
\draw[fill=black] ( 4.85,0.15)circle(.5pt);
\draw[fill=black] ( 5.15,-0.15)circle(.5pt);
\draw[fill=black] ( 5.15,0.15)circle(.5pt);
   
\draw[fill=lightgray] (0:5)+(0:1.2)circle(4pt);
\draw[fill=white] (0:5)+(0:1.2)circle(2pt);
\draw[fill=black] (0:5)+(0:1.2)circle(.5pt);
\draw[  ] (0:5)+(0:1.2)circle(6.5pt);

\draw[fill=lightgray] (0:5)+(72:1.2)circle(4pt);
\draw[fill=white] (0:5)+(72:1.2)circle(2pt);
\draw[fill=black] (0:5)+(72:1.2)circle(.5pt);
\draw[ ] (0:5)+(72:1.2)circle(6.5pt);
     
\draw[fill=lightgray] (0:5)+(144:1.2)circle(4pt);
\draw[fill=white] (0:5)+(144:1.2)circle(2pt);
\draw[fill=black] (0:5)+(144:1.2)circle(.5pt);
\draw[ ] (0:5)+(144:1.2)circle(6.5pt);

\draw[fill=lightgray] (0:5)+(216:1.2)circle(4pt);
\draw[fill=white] (0:5)+(216:1.2)circle(2pt);
\draw[fill=black] (0:5)+(216:1.2)circle(.5pt);
\draw[] (0:5)+(216:1.2)circle(6.5pt);

\draw[fill=lightgray] (0:5)+(-72:1.2)circle(4pt);
\draw[fill=white] (0:5)+(-72:1.2)circle(2pt);
\draw[fill=black] (0:5)+(-72:1.2)circle(.5pt);
\draw[] (0:5)+(-72:1.2)circle(6.5pt);

\draw[fill=lightgray] (0:5)+(339.5:2.5)circle(4pt);
\draw[fill=white] (0:5)+(339.5:2.5)circle(2pt);
\draw[fill=black] (0:5)+(339.5:2.5)circle(.5pt);
\draw[  ] (0:5)+(339.5:2.5)circle(6.5pt);
             
\draw[fill=lightgray] (0:5)+(288:2.5)circle(4pt);
\draw[fill=white] (0:5)+(288:2.5)circle(2pt);
\draw[fill=black] (0:5)+(288:2.5)circle(.5pt);
\draw[ ] (0:5)+(288:2.5)circle(6.5pt);

\draw[] (0:5)+(236.5:2.5)circle(6.5pt);
\draw[fill=lightgray] (0:5)+(236.5:2.5)circle(4pt);
\draw[fill=white] (0:5)+(236.5:2.5)circle(2pt);
\draw[fill=black] (0:5)+(236.5:2.5)circle(.5pt);
       
\draw[ ] (0:5)+(185:2.5)circle(6.5pt);
\draw[fill=lightgray] (0:5)+(185:2.5)circle(4pt);
\draw[fill=white] (0:5)+(185:2.5)circle(2pt);
\draw[fill=black] (0:5)+(185:2.5)circle(.5pt);
    
\draw[fill=lightgray] (0:5)+(133.5:2.5)circle(4pt);
\draw[fill=white] (0:5)+(133.5:2.5)circle(2pt);
\draw[fill=black] (0:5)+(133.5:2.5)circle(.5pt);
\draw[ ] (0:5)+(133.5:2.5)circle(6.5pt);

\draw[fill=lightgray] (0:5)+(82:2.5)circle(4pt);
\draw[fill=white] (0:5)+(82:2.5)circle(2pt);
\draw[fill=black] (0:5)+(82:2.5)circle(.5pt);
\draw[ ] (0:5)+(82:2.5)circle(6.5pt);

\draw[fill=lightgray] (0:5)+(30.75:2.5)circle(4pt);
\draw[fill=white] (0:5)+(30.75:2.5)circle(2pt);
\draw[fill=black] (0:5)+(30.75:2.5)circle(.5pt);
\draw[ ] (0:5)+(30.75:2.5)circle(6.5pt);

\node(a)at(-5:6.2){\small$6/5$};
\node(a)at(-3:7.5){\small$8/7$};
        
\node(a)at(2,1.5){\small${ 1}$};
         
\node(a)at(8.3,-2.5){\small${ 7/5}$};
\node(a)at(7.3,-2.5){\small${10/7}$};
\draw(8.2,-2.2)--(5.4,-1.04);
\draw(6.9,-2.44)--(5.8,-2.28);

\node(a)at(1.6,-1){\small${5/4}$};
\draw(1.9,-1)--(4.8,0);
 


\draw[fill=white] (1,1.5)circle(3pt);
\draw[fill=white] (0.5,1.5)circle(3pt);
\draw[fill=white] (0,1.5)circle(3pt);
\draw[fill=white] (1,1)circle(3pt);
\draw[fill=white] (0.5,1)circle(3pt);
\draw[fill=white] (0,1)circle(3pt);
\draw[fill=white] (1,.5)circle(3pt);
\draw[fill=white] (0.5,.5)circle(3pt);
\draw[fill=white] (0,.5)circle(3pt);
\draw[fill=white] (1,0)circle(3pt);
\draw[fill=white] (0.5,0)circle(3pt);
\draw[fill=white] (0,0)circle(3pt);
\draw[fill=white] (1,-0.5)circle(3pt);
\draw[fill=white] (0.5,-0.5)circle(3pt);
\draw[fill=white] (0,-0.5)circle(3pt);
\draw[fill=white] (1,-1)circle(3pt);
\draw[fill=white] (0.5,-1)circle(3pt);
\draw[fill=white] (0,-1)circle(3pt);

\draw[fill=black] (1,1.5)circle(.5pt);
\draw[fill=black] (0.5,1.5)circle(.5pt);
\draw[fill=black] (0,1.5)circle(.5pt);
\draw[fill=black] (1,1)circle(.5pt);
\draw[fill=black] (0.5,1)circle(.5pt);
\draw[fill=black] (0,1)circle(.5pt);
\draw[fill=black] (1,.5)circle(.5pt);
\draw[fill=black] (0.5,.5)circle(.5pt);
\draw[fill=black] (0,.5)circle(.5pt);
\draw[fill=black] (1,0)circle(.5pt);
\draw[fill=black] (0.5,0)circle(.5pt);
\draw[fill=black] (0,0)circle(.5pt);
\draw[fill=black] (1,-0.5)circle(.5pt);
\draw[fill=black] (0.5,-0.5)circle(.5pt);
\draw[fill=black] (0,-0.5)circle(.5pt);
\draw[fill=black] (1,-1)circle(.5pt);
\draw[fill=black] (0.5,-1)circle(.5pt);
\draw[fill=black] (0,-1)circle(.5pt);
      
             \node(a)at(2,1.5){${\small 1}$};       
\end{scope}  

\end{tikzpicture}  
\end{center}\caption{Example \ref{ex:vvb}: Complete carrousel section.
}
\end{figure}
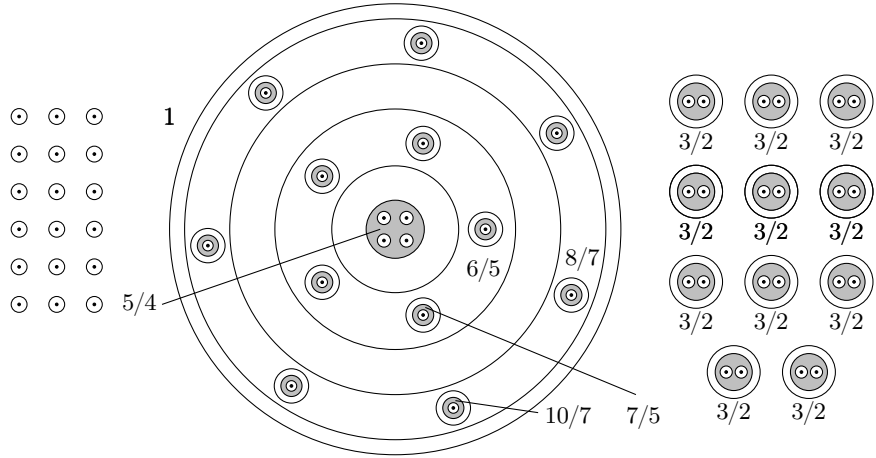
\end{example}

\newpage\hbox{~}

     \section*{\bf Part 4: Zariski equisingularity implies semi-analytic
   Lipschitz triviality}\label{Part 4}

\section{Equisingularity}\label{sec:notations}

  In this section, we define Zariski equisingularity as in Speder \cite{Sp} and we specify it in codimension $2$. For a family of hypersurface germs in $(\C^3,0)$ with isolated singularities, the definition is equivalent to the last definition of equisingularity stated by Zariski in \cite{Z3} using equidimensionality type.  We also define
Lipschitz triviality in this codimension and we fix notations for the rest of the paper.

\begin{definition} \label{def:zariski equisingularity} Let
  $(\mathfrak{X},0)\subset (\C^{n},0)$ be a reduced hypersurface germ
  and $(Y,0) \subset (\mathfrak{X},0)$ a nonsingular germ.
  $\mathfrak{X}$ is \emph{Zariski equisingular} along $Y$ near $0$ if
for a generic linear projection $\mathscr L \colon \C^n \to
  \C^{n-1}$ for the pair $(\mathfrak{X},Y)$ at $0$, the branch locus
  $\Delta \subset \Bbb C^{n-1}$ of the restriction $\mathscr L
  |_{\mathfrak{X}} \colon \mathfrak{X} \to \Bbb C^{n-1}$ is
  equisingular along $\mathscr L(Y)$ near $0$ (in particular,
  $(\mathscr L(Y),0) \subset (\Delta,0)$).  When $Y$ has codimension
  one in $\mathfrak{X}$, Zariski equisingularity means that $\Delta $
  is nonsingular.
\end{definition}

The notion of a generic linear projection is defined in  \cite[definition 4]{Sp}. 

When $Y$ has codimension one in $\mathfrak{X}$, Zariski equisingularity is
equivalent to Whitney conditions for the pair $(\mathfrak{X} \setminus Y,Y)$ and
also to topological triviality of $\mathfrak{X}$ along $Y$.

From now on we consider the case where $(Y,0)$ is the singular locus of $\mathfrak{X}$  and $Y$ has codimension
$2$ in $\mathfrak{X}$ (i.e., dimension $n-3$).  Then any slice of $\mathfrak{X}$ by a
smooth $3$-space transversal to $Y$ is a normal surface
singularity. If  $\mathscr L $ is a generic projection for $(\mathfrak{X},Y)$ at $0$ and $H$ a hyperplane through  $y \in Y$ close to $0$ which contains the line $\ker \mathscr L $, then the restriction of  $\mathscr L $ to $H \cap \mathfrak{X}$    is  a generic projection of the normal surface germ  $(H \cap \mathfrak{X} ,0)$ in the sense of Definition \ref{def:generic linear   projection}. 

We have the following
characterization of Zariski equisingularity for families of isolated
singularities in $(\Bbb C^3,0)$ which will be used throughout the rest of the paper:

\begin{proposition} Let $(\mathfrak{X},0)\subset (\C^{n},0)$ be a reduced
  hypersurface germ at the origin of $\C^{n}$ with $2$-codimension
  smooth singular locus $(Y,0)$. Then $\mathfrak{X}$ is Zariski equisingular along $Y$ near
  $0$ if for a generic linear projection $\mathscr L \colon \C^{n} \to
  \C^{n-1}$, the branch locus $\Delta  \subset \Bbb
  C^{n-1}$ of $\mathscr L |_\mathfrak{X}\colon \mathfrak{X} \to \C^{n-1}$ is topologically equisingular
  along $\mathscr L(Y)$ near $0$.\qed
 \end{proposition}

\begin{remark} In \cite{Z3}, Zariski introduced an alternative definition of equisingularity  based on the concept of dimensionality. The definition  is  also by induction on the codimension and involves the discriminant locus of successive projections. At each step, a number called dimensionality is defined by induction, and it  is the same for ``almost all" projections. This defines a notion of genericity for such projections. The definition says   $\mathfrak{X}$ is \emph{equisingular} along $Y$ if  the dimensionality equals $1$ at the last step of the induction i.e., when $Y$ has codimension $1$ in $ \mathfrak{X}$.  This is equivalent to saying that the family of discriminants is equisingular as a family of curves. For details see \cite{Z3} or \cite{L}. In \cite{BH}, Brian\c{c}on and Henry proved that when $Y$ has codimension $2$ in $\mathfrak{X}$, i.e., for a family of complex
hypersurfaces in $\C^3$, the dimensionality can be computed using only linear projections.  As a consequence, in codimension $2$, Zariski  equisingularity as defined in \ref{def:zariski equisingularity} coincides with Zariski's equisingularity concept using dimensionality. 
\end{remark}

 \begin{definition} 
\begin{enumerate}\item 
$(\mathfrak{X},0)$   \emph{has constant (semi-analytic)  Lipschitz geometry} along $Y$ if   there exists a  smooth  (semi-analytic) retraction $r \colon (\mathfrak{X},0)  \to (Y,0)$  whose fibers are transverse to $Y$  and a neighbourhood $U$ of $0$ in $Y$ such that for all  $y \in U$, there exists a (semi-analytic) bilipschitz    homeomorphism  $h_y \colon (r^{-1}(y),y) \to (r^{-1}(0) \cap \mathfrak{X},0)$. 
\item  
The germ $(\mathfrak{X},0)$ is (semi-analytic)
  \emph{Lipschitz trivial} along $Y$ if there exists a germ at 0 of a
  (semi-analytic) bilipschitz homeomorphism $\Phi \colon(\mathfrak
  X,Y)\to (X,0)\times Y$ with $\Phi|_Y=id_Y$, where $(X,0)$ is a
  normal complex surface germ.
\end{enumerate}
 \end{definition}
 
 The aim of this last part is to prove Theorem \ref{th:zariski vs bilipschitz} stated in the introduction:
 
\begin{theorem*}
 The following are
  equivalent:
     \begin{enumerate}
  \item \label{it:zariski-}  $(\mathfrak{X},0)$ is Zariski equisingular along $Y$;  
  \item \label{it:constancy-} $(\mathfrak{X},0)$ has constant Lipschitz geometry along $Y$;
  \item \label{it:constancySA-} $(\mathfrak{X},0)$ has constant semi-analytic Lipschitz geometry along $Y$;
      \item \label{it:lipschitz-} $(\mathfrak{X},0)$ is semi-analytic Lipschitz trivial along $Y$  \end{enumerate}
\end{theorem*}

The implications \eqref{it:lipschitz-} $\Rightarrow$ \eqref{it:constancySA-} and \eqref{it:constancySA-} $\Rightarrow$ \eqref{it:constancy-} are trivial.  
 \eqref{it:constancy-} $\Rightarrow$ \eqref{it:zariski-} is an easy  consequence
of part \eqref{it:discriminant} of Theorem \ref{th:invariants from
  geometry}: 

\begin{proof}[Proof of \eqref{it:constancy-} $\Rightarrow$
  \eqref{it:zariski-}] Assume $(\mathfrak{X},0)$ has constant
  semi-analytic Lipschitz geometry along $Y$. Let $\mathscr L \colon
  \C^n \to \C^{n-1}$ be a generic linear projection for
  $\mathfrak{X}$. Let $r\colon \mathscr L(\mathfrak{X}) \to \mathscr L
  (Y)$ be a smooth semi-analytic retraction whose fibers are
  transversal to $\mathscr L (Y)$. Its lift by $\mathscr L$ is a
  retraction $\widetilde{r} \colon \mathfrak{X} \to Y$ whose fibers
  are transversal to $Y$. For any $t \in Y$ sufficiently close to $0$,
  $X_t = (\tilde{r})^{-1}(t)$ is semi-analytically bilipschitz
  equivalent to $X_0 = (\tilde{r})^{-1}(0)$. Then, according to part
  \eqref{it:discriminant} of Theorem \ref{th:invariants from
    geometry}, the discriminants $\Delta_t$ of the restrictions
  $\mathscr L |_{ X_t}$ have same embedded topology.  This proves that
  $\mathfrak{X}$ is Zariski is equisingular along $Y$.
\end{proof}  
  The last sections of the paper are devoted to the proof of    \eqref{it:zariski-}  $\Rightarrow$ \eqref{it:lipschitz-}. 
 
\subsection*{Notations} Assume $\mathfrak{X}$ is Zariski equisingular along $Y$ at $0$.   Since Zariski equisingularity is a stable property under analytic
 isomorphism $(\C^{n},0) \rightarrow (\C^{n},0)$, we will assume
 without lost of generality that $Y = \{ {0}\} \times \Bbb C^{n-3}
 \subset \Bbb C^3 \times \Bbb C^{n-3} = \Bbb C^n$ where $ {0}$ is the
 origin in $\Bbb C^3$.  We will denote by $ (\underbar{x} ,t)$ the
 coordinates in $\C^{3}\times \Bbb C^{n-3}$, with $ \underbar{x}=
 (x,y,z) \in \Bbb C^3$ and $t\in \Bbb C^{n-3}$.  For each $t$, we set 
 $X_t = \mathfrak{X} \cap (\C^3 \times \{t\})$ and consider $\mathfrak{X}$ as the
 $n-3$-parameter family of isolated hypersurface singularities $(X_t,(
 {0},t)) \subset (\Bbb C^3 \times \{t\},(0,t))$. For simplicity, we
 will write $(X_t,0)$ for $(X_t,( {0},t))$.

 Let $\mathscr L \colon \C^n \rightarrow \C^{n-1}$ be a generic linear
 projection for $\mathfrak{X}$. We choose the coordinates
 $(\underbar{x},t)$ in $\C^3 \times \Bbb C^{n-3}$ so that $\mathscr L$
 is given by $\mathscr L(x,y,z,t)=(x,y,t)$.  For each $t$, the
 restriction $\mathscr L |_{\C^3 \times \{t\}} \colon \C^3 \times
 \{t\} \to \C^2 \times \{t\}$ is a generic linear projection for
 $X_t$.  We denote by $\Pi_t \subset X_t$ the polar curve of the
 restriction $ \mathscr L |_{X_t} \colon X_t \rightarrow \C^2\times
 \{t\}$ and by $\Delta_t $ its discriminant curve $\Delta_t = \mathscr
 L(\Pi_t)$.
 
 Throughout Part 4, we will use the Milnor balls defined as
 follows. By Speder \cite{Sp}, Zariski equisingularity implies Whitney
 conditions. Therefore one can choose a constant size $\epsilon>0$ for
 Milnor balls of the $(X_t,0)$ as $t$ varies in some small ball
 $D_\delta$ about $0\in \C^{n-3}$ and the same is true for the family
 of discriminants $\Delta_t$.  So we can actually use a
 ``rectangular'' Milnor ball as introduced in Section
 \ref{sec:carrousel1}:
  $$B^6_\epsilon=\{(x,y,z):|x|\le \epsilon, |(y,z)|\le R\epsilon\}\,,$$ and its projection 
   $$ B^4_\epsilon=\{(x,y):|x|\le \epsilon, |y|\le R\epsilon,
   \}\,,$$ where $R$ and $\epsilon_0$ are chosen so that for $\epsilon
   \leq \epsilon_0$ and $t$ in a small ball $D_\delta $,
  \begin{enumerate}
\item $B^6_\epsilon \times \{t\}$ is a Milnor ball for $(X_t,0)$;
\item  for each $\alpha$ such that $|\alpha| \leq \epsilon$,  $h_t^{-1}(\alpha)$ intersects the standard sphere ${\Bbb S}^6_{R \epsilon}$ transversely, where $h_t $ denotes the restriction $h_t = x |_{X_t}$;
\item the curve $\Pi_t$ and its tangent cone $T_0\Pi_t$
  meet $\partial B^6_\epsilon \times \{t\}$ only in the part $|x|=\epsilon$.
\end{enumerate}

\section{Proof outline that Zariski implies
  bilipschitz equisingularity}\label{sec:Ztob}

We start by outlining the proof that Zariski equisingularity implies
semi-analytic Lipschitz triviality.  We assume therefore that we have
a Zariski equisingular family $\mathfrak{X}$ as above.

We choose $B^6_\epsilon\subset \C^3$ and $D_\delta\subset \C^{n-3}$ as
in the previous section.  We want to construct a semi-analytic
bilipschitz homeomorphism $\Phi\colon
\mathfrak{X}\cap(B^6_\epsilon\times D_\delta) \to (X_0\times
\C^{n-3})\cap (B^6_\epsilon\times D_\delta)$ which preserves the
$t$-parameter.  Our homeomorphism will be a piecewise diffeomorphism.

In Section \ref{sec:carrousel1} we used a carrousel decomposition of
$B^4_\epsilon$ for the discriminant curve $\Delta_0$ of the
generic linear projection $ \mathscr L |_{X_0}\colon X_0\to
B^4_\epsilon$ and we lifted it to obtain   the geometric decomposition of
$(X,0)$ (Definition \ref{def:geometric decomposition}). Here, we will consider this construction for each $ \mathscr
L |_{X_t}\colon X_t \to B^4_\epsilon \times \{t\}$. 

Next, we construct a semi-analytic bilipschitz map $\phi\colon
B^4_\epsilon\times D_\delta\to B^4_\epsilon\times D_\delta$ which
restricts for each $t$ to a map $\phi_t\colon B^4_\epsilon\times
\{t\}\to B^4_\epsilon\times \{t\}$ which takes the carrousel
decomposition of $B^4_\epsilon$ for $X_t$ to the corresponding
carrousel decomposition for $X_0$. We also arrange that $\phi_t$
preserves a foliation of $B^4_\epsilon$ adapted to the \carrousels. A
first approximation to the desired map $\Phi\colon
\mathfrak{X}\cap(B^6_\epsilon\times D_\delta) \to (X_0\times
\C^{n-3})\cap (B^6_\epsilon\times D_\delta)$ is then obtained by
simply lifting the map $\phi$ via the branched covers $\mathscr L$ and
$ \mathscr L |_{X_0} \times id_{D_\delta}$:
$$
\begin{CD}
\mathfrak{X}\cap(B^6_\epsilon \times D_\delta) @>\Phi>> (X_0\cap B^6_\epsilon)\times D_\delta\\
@V \mathscr L VV @V \mathscr L |_{X_0} \times  id_{D_\delta} VV\\
B^4_\epsilon\times D_\delta @>\phi>>  B^4_\epsilon\times D_\delta
\end{CD}
$$

Now, fix $t$ in $D_{\delta}$. For $\underbar{x} \in X_t \cap
B^6_{\epsilon}$, denote by $K(\underbar{x},t)$ the local bilipschitz
constant of the projection $ \mathscr L |_{X_t} \colon X_t \to
B^4_\epsilon \times\{t\}$ (see Section \ref{sec:polarwedges}). Let us
fix a large $K_0>0$ and consider the set ${\mathcal B}_{K_0,t} = \{
\underbar{x} \in X_t \cap (B^6_{\epsilon}\times D_{\delta}) \ \colon \
K(\underbar{x},t) \geq K_0\}$.

We then show that $\Phi$ is bilipschitz for the inner metric except
possibly in the zone $C_{K_0} = \bigcup_{t \in D_{\delta}} {\mathcal
  B}_{K_0,t} $, where the bilipschitz constant for the linear
projection $ \mathscr L$ becomes large.  This is all done in Section
\ref{sec:foliated carrousel}.

Using the  approximation Proposition \ref{prop:thinzones} of $\mathcal B_{K_0,t}$ by   polar wedges, we  then prove that  $\Phi$ can be adjusted if necessary to be
bilipschitz on $C_{K_0}$.

 The construction just explained is based on the key Lemma \ref{prop:polar constancy} which states  constancy of the polar  rates of the components of $(\Delta_t,0)$ as $t$ varies. The proof of this lemma  is based on the use of families of plane curves  $(\gamma_t,0)$ such that $(\Delta_t \cup \gamma_t,0)$ has constant topological type as $t$ varies, and their {\it liftings} by $\mathscr L$, {\it spreadings} and {\it projected liftings} defined in Section \ref{sec:test curves}. The key argument is  the \emph{restriction formula } of Teissier \cite{T0}  that we recall in section \ref{sec:formula}. Similar techniques were used by Casas-Alvero in \cite{C} in his study of equisingularity of inverse images of plane curves by an analytic morphism $(\C^2,0) \to (\C^2,0)$.

Once $\Phi$ is inner bilipschitz, we must
show that the outer Lipschitz geometry is also preserved. This is done in Section  \ref{sec:constancy} using again  families of plane curves  $(\gamma_t,0)$ as just defined and the restriction formula. Namely we prove the invariance of bilipschitz type of the liftings by $\mathscr L$ of such families, which  enables one to test invariance of the
outer Lipschitz geometry.
 
 The pieces
of the proof are finally put together in Section \ref{sec:proofZtoL}.

\section{Foliated Carrousel Decomposition}\label{sec:foliated carrousel}

We use again the notations of Section \ref{sec:notations}.  In this
section, we construct a self-map $\phi\colon B^4_\epsilon\times
D_\delta \to B^4_\epsilon\times D_\delta$ of the image of the generic
linear projection which removes the dependence on $t$ of the carrousel
decompositions and we introduce a foliation of the carrousels by $1$-dimensional complex  leaves which will be preserved by $\phi$.  We first present a
parametrized version of the carrousel construction of section
\ref{sec:carrousel} for any choice of truncation of the Puiseux series
expansions of $\Delta$.

\subsection*{Carrousel depending on the parameter $\mathbf t$} 

  \red{Since the family of discriminants $\Delta_t$  is Zariski equisingular, then by \cite[Theorem 7]{Z0} (see also \cite{P}),  it admits a Puiseux parametrization with parameter, i.e.,  the branches of $\Delta_t$ admits Puiseux parametrizations of the form $$y = \sum_{i \geq 1} a_i(t)x^{p_i},$$ with $a_i(t) \in \C\{t\}$.
}
  

For each $t$, the tangent cone 
of $\Delta_t$ is a union of tangent lines $L^{(1)}_t,\dots, L^{(m)}_t$
and these are distinct lines for each $t$. 
In Section \ref{sec:carrousel} we described   carrousel
decompositions when there is no parameter $t$. They now decompose 
  a conical neighbourhood of each line $L^{(j)}_t$. These conical
neighbourhoods are as follows. Let the equation of the $j$-th line be
$y=a_1^{(j)}(t)x$.  We choose a small enough $\eta>0$ and $\delta>0$
such that cones
$$V^{(j)}_t:=\{(x,y):|y-a_1^{(j)}(t)x|\le \eta |x|, |x|\le
\epsilon\}\subset \C^2_t$$ are disjoint for all $t\in D_\delta$, and
then shrink $\epsilon$ if necessary so $\Delta_t^{(j)}\cap
\{|x|\le\epsilon\}$ will lie completely in $V^{(j)}_t$ for all $t$.

We now describe how   a carrousel decomposition  is modified to the parametrized case.  We 
fix $j=1$ for the moment and therefore drop the superscripts, so our
tangent line $L$ has equation $y=a_1(t) x$.

We first choose a  truncation of the Puiseux series for each component of
$\Delta_t$.   Then for each pair $\kappa=(f, p_k)$ consisting of a
Puiseux polynomial $f=\sum_{i=1}^{k-1}a_i(t)x^{p_i}$ and an exponent
$p_k$ for which there is a Puiseux series
$y=\sum_{i=1}^{k}a_i(t)x^{p_i}+\dots$ describing some component of
$\Delta_t$, we consider all components of $\Delta_t$ which fit this
data. If $a_{k1}(t),\dots,a_{k{m_\kappa}}(t)$ are the coefficients of
$x^{p_k}$ which occur in these Puiseux polynomials we define
\begin{align*}
  B_{\kappa,t}:=\Bigl\{(x,y):~&\alpha_\kappa|x^{p_k}|\le
  |y-\sum_{i=1}^{k-1}a_i(t)x^{p_i}|\le
  \beta_\kappa|x^{p_k}|\,,\\
  & |y-(\sum_{i=1}^{k-1}a_i(t)x^{p_i}+a_{kj}(t)x^{p_k})|\ge
  \gamma_\kappa|x^{p_k}|\text{ for }j=1,\dots,{m_\kappa}\Bigr\}\,.
\end{align*}
Again, $\alpha_\kappa,\beta_\kappa,\gamma_\kappa$ are chosen so that
$\alpha_\kappa<|a_{kj}(t)|-\gamma_\kappa<|a_{kj}(t)|+\gamma_\kappa<\beta_\kappa$
for each $j=1,\dots,{m_\kappa}$ and all small $t$.  If $\epsilon$ is
small enough, the sets $ B_{\kappa,t}$ will be disjoint for different
$\kappa$.  The closure of the complement in $V_t$ of the union of the
$B_{\kappa,t}$'s is a union of $A$- and $D$-pieces.

 \subsection*{Foliated carrousel} We now refine our carrousel decomposition by adding a 
 piecewise smooth foliation of $V_t$ by complex curves compatible with
 the carrousel. We do this as follows:

 A piece $B_{\kappa,t}$ as above is foliated with closed leaves given
 by curves of the form
$$C_\alpha:=\{(x,y): y=\sum_{i=1}^{k-1}a_i(t)x^{p_i}+\alpha x^{p_k}\}$$
for $\alpha\in \C$ satisfying $\alpha_\kappa\le|\alpha|\le
\beta_\kappa $ and $|a_{kj}(t)-\alpha|\ge \gamma_\kappa$ for
$j=1,\dots,{m_\kappa}$ and all small $t$. We foliate $D$-pieces
similarly, including the pieces which correspond to $\Delta$-wedges
about $\Delta_t$.

An $A$-piece has the form 
$$A=\{(x,y):\beta_1|x^{p_{k+1}}|\le |y-(\sum_{i=1}^k a_i(t)x^{p_i})|\le
\beta_0|x^{p_k}|\}\,,$$ where $a_k(t)$ may be $0$, $p_{k+1}>p_k$
and $\beta_0,\beta_1$ are $>0$.  We foliate with leaves the immersed
curves of the following form 
$$C_{r,\theta}:=\{(x,y): y=\sum_{i=1}^{k}a_i(t)x^{p_i}+\beta(r) e^{i\theta}x^{r}\}$$
with $p_k\le r\le p_{k+1}$, $\theta \in \Bbb R$ and
$\beta(r)=\beta_1+\frac{p_{k+1}-r}{p_{k+1}-p_k}(\beta_0-\beta_1)$.
Note that these leaves may not be
closed; for irrational $r$ the topological closure is
homeomorphic to the cone on a torus. 

\begin{definition} We call a carrousel decomposition  equipped with
  such a foliation by curves a \emph{foliated carrousel decomposition}. 
\end{definition}

  \subsection*{Trivialization of the family of foliated carrousels} We set 
$$ V:=\bigcup_{t\in D_\delta}V_t\times\{t\}\,.$$
\begin{proposition}\label{prop:phi on V} 
  If\/ $\delta$ is sufficiently small,
  there exists a semi-analytic  bilipschitz map $\phi_V\colon V \to V_0\times
  D_\delta$ such that: 
  
  \begin{enumerate}
\item $\phi_V$ preserves the $x$ and $t$-coordinates,
\item  for each $t\in D_\delta$,  $\phi_V$ preserves the foliated carrousel decomposition of $V_t$ i.e., it maps the carrousel decomposition of $V_t$ to that of 
    $V_0$, preserving the respective foliations.
      \item $\phi_V$ maps complex lines to complex lines on
  the portion $|x|<\epsilon$ of $\partial V$.
\end{enumerate}

\end{proposition}
\begin{proof}
  We first construct the map on the slice $V\cap \{x=\epsilon\}$.  We
  start by extending the identity map $ V_0 \cap\{x=\epsilon\}\to
  V_0\cap\{x=\epsilon\}$ to a family of piecewise smooth maps
  $V_t\cap\{x=\epsilon\}\to V_0\cap\{x=\epsilon\}$ which map carrousel
sections to carrousel sections.

For fixed $t$ we are looking at a carrousel section $V_t\cap\{x=\epsilon\}$
as exemplified in Figure \ref{fig:carrousel}. The various regions are
bounded by circles. Each disk or annulus in $V_t\cap\{x=\epsilon\}$ is
isometric to its counterpart in $V_0\cap\{x=\epsilon\}$ so we map it by a
translation. To extend over a piece of the form
$B_{\kappa,t}\cap\{x=\epsilon\}$ we will subdivide $B_{\kappa,t}$
further.  Assume first, for simplicity, that the coefficients
$a_{k1}(t),\dots, a_{km_\kappa}(t)$ in the description of $B_{\kappa,t}$
satisfy $|a_{k1}(t)|<|a_{k2}(t)<\dots<|a_{km_{\kappa}}(t)|$ for small
$t$. For each $j=1,\dots,m_{\kappa}$ we define
\begin{align*}
  B_{\kappa j,t}:=\Bigl\{(x,y):~&\alpha'_\kappa|a_{kj}(t)||x^{p_k}|\le
  |y-\sum_{i=1}^{k-1}a_i(t)x^{p_i}|\le
  \beta'_\kappa|a_{kj}(t)|x^{p_k}|\,,\\
  & |y-(\sum_{i=1}^{k-1}a_i(t)x^{p_i}+a_{kj}(t)x^{p_k})|\ge
  \gamma'_\kappa|a_{kj}(t)||x^{p_k}|\Bigr\}\,,
\end{align*}
with $\alpha'_\kappa<|1-\gamma'_\kappa|<|1+\gamma'_\kappa|$,
$\beta'_\kappa|a_{kj}(t)|<\alpha'_\kappa|a_{k,j+1}|$ for
$j=1,\dots,m_{\kappa}-1$, $\gamma'_\kappa|a_{kj}(t)\le \gamma_k$,
$\alpha_\kappa<\alpha'_\kappa|a_{k1}(t)|<\beta'_\kappa|a_{km_{\kappa}}(t)|<\beta_\kappa$.
This subdivides $B_{\kappa,t}$ into pieces $B_{\kappa j,t}$ with
$A(p_k,p_k)$-pieces between each $B_{\kappa j,t}$ and $B_{\kappa\, j+1,t}$
and between the $B_{\kappa j,t}$'s and the boundary components of
$B_{\kappa,t}$.  We can map $B_{\kappa j,t}\cap \{x=\epsilon\}$ to
$B_{\kappa j,0}\cap \{x=\epsilon\}$ by a translation by
$(\sum_{i=1}^{k-1}a_i(0)\epsilon^{p_i} -
\sum_{i=1}^{k-1}a_i(t)\epsilon^{p_i})$ followed by a similarity which
multiplies by $a_{kj}(0)/a_{kj}(t)$ centered at
$\sum_{i=1}^{k-1}a_i(0)\epsilon^{p_i}$.  We map the annuli of
$\overline {B_{\kappa,0}\setminus \bigcup B_{\kappa j,0}}\cap
\{x=\epsilon\}$ to the corresponding annuli of $\overline
{B_{\kappa,t}\setminus \bigcup B_{\kappa j,t}}\cap \{x=\epsilon\}$ by
maps which agree with the already constructed maps on their
boundaries and which   extend over each annulus by a ``twist map'' of
the form in polar coordinates: $(r,\theta)\mapsto (ar+b,\theta+cr+d)$
for some $a,b,c,d$.

We assumed above that the $|a_{kj}(t)|$ are pairwise unequal, but if
$|a_{kj}(t)|=|a_{kj'}(t)|$ as $t$ varies, then
$a_{kj}(t)/a_{kj'}(t)$ is a complex analytic function of $t$ with
values in the unit circle, so it must be constant. So in this
situation we use a single $B_{\kappa j}$ piece for all $j'$ with
$|a_{kj}(t)|=|a_{kj'}(t)|$ and the construction still works.

This defines our map $\phi_V$ on $\{x=\epsilon\}\times D_\delta$ for
$\delta$ sufficiently small, and the requirement that $\phi_V$
preserve $x$-coordinate and foliation extends it uniquely to all of
$V\times D_\delta$.

The map $\phi_V$ is constructed to be semi-analytic, and it remains
to prove that it is bilipschitz. It is bilipschitz on $(V\cap
\{|x|=\epsilon\})\times D_\delta$ since it is piecewise smooth on a
compact set. Depending what leaf of the foliation one is on, in a
section $|x|=\epsilon'$ with $0<\epsilon'\le \epsilon$, a
neighbourhood of a point $\underbar p$ scales from a neighbourhood of
a point $\underbar p'$ with $|x|=\epsilon$ by a factor of
$(\epsilon'/\epsilon)^r$ in the $y$-direction (and $1$ in the $x$
direction) as one moves in along a leaf, and the same for
$\phi_V(\underbar p)$. So to high order the local bilipschitz constant
of $\phi_V$ at $\underbar p$ is the same as for $\underbar p'$ and
hence bounded. Thus the bilipschitz constant is globally bounded on
$(V\setminus \{0\})\times D_\delta$, and hence on $V \times D_\delta$.
\end{proof}

The $V$ of the above proposition was any one of the $m$ sets
$V^{(i)}:=\bigcup_{t\in D_\delta}V^{(i)}_t$, so the proposition
extends $\phi$ to all these sets.  We extend the carrousel foliation
on the union of these sets to all of $B^4_\epsilon\times \{t\}$ by
foliating the complement of the union of $V^{(i)}_t$'s by complex
lines.  We denote $B^4_\epsilon $ with this carrousel decomposition
and foliation structure by $B_{\epsilon,t}$. We finally extend $\phi$
to the whole of $\bigcup_{t\in D_\delta}B_{\epsilon, t}$ by a
diffeomorphism which takes the complex lines of the foliation linearly
to complex lines of the foliation on $B_{\epsilon,0}\times D_\delta$,
preserving the $x$ coordinate. The resulting map remains obviously
semi-analytic and bilipschitz.

 We then have shown:
\begin{proposition}\label{prop:phi}
  There exists a  map
$\phi\colon \bigcup_{t\in D_\delta}B_{\epsilon,t} \to B_{\epsilon,0}\times D_\delta$
which is semi-analytic and bilipschitz and such that each $\phi_t\colon
B_{\epsilon,t}\to B_{\epsilon,0}\times\{t\}$ preserves the 
carrousel decompositions and foliations.\qed
\end{proposition}

\section{Liftings and spreadings}\label{sec:test curves}

 Let $(X,0) \subset (\C^3,0)$ be an isolated hypersurface singularity
and let $\ell \colon \C^3 \rightarrow
\C^2$ be a generic linear projection for $(X,0)$. Consider an  irreducible plane curve $(\gamma,0)\subset (\C^2,0)$. 

\begin{definition} \label{def:projection}
    The \emph{lifting} of $\gamma$ is the curve   $$L_\gamma:=(\ell |_X)^{-1}(\gamma)\,.$$
     Let $\ell'\colon \C^3\to \C^2$ be another generic linear projection  for $(X,0)$ which is also 
generic for $L_\gamma$. We call the plane curve
      $$P_{\gamma} =
   \ell'(L_{\gamma})$$
    a \emph{projected lifting} of $\gamma$. 
          \end{definition}
     
Notice that the topological type of $(P_{\gamma},0)$ does not depend
on the choice of $\ell'$.

Let us choose the coordinates $(x,y,z)$ of $\C^3$ in such a way that
$\ell=(x,y)$ and  $\gamma$ is tangent to the $x$--axis. 
We consider a Puiseux expansion of $\gamma$:
\begin{align*}
  x(w)&=w^q\\
  y(w)&=a_1w^{q_1} + a_2 w^{q_2} + \ldots + a_{n-1} w^{q_{n-1}} + a_n
  w^{q_n} + \dots
\end{align*}
Let $F(x,y,z)=0$ be an equation for $X\subset \C^3$.
\begin{definition} We call the plane curve $(F_{\gamma},0) \subset
  (\C^2,0)$ with equation
$$F(x(w),y(w),z)=0$$ a \emph{spreading } of the  lifting $L_\gamma$. 
\end{definition}

Notice that the topological type of $(F_{\gamma},0)$ does not depend
on the choice of the parametrization.
 
\begin{lemma} \label{lem:top} The topological types of the spreading $(F_{\gamma},0)$ and
  of the projected lifting $(P_{\gamma},0)$  determine each other.
\end{lemma}

\begin{proof} Assume that the coordinates of $\C^3$ are chosen in such
  a way that $\ell'=(x,z)$. Then $P_{\gamma} = \rho (F_{\gamma})$ where
  $\rho\colon\C^2 \rightarrow \C^2$ denotes the morphism $\rho(w,z) =
  (w^q,z)$, which is a cyclic $q$-fold cover branched on the line
  $w=0$. Thus $P_\gamma$ determines $F_\gamma$. The link of
  $P_\gamma\cup\{z\text{-axis}\}$ consists of an iterated torus link
  braided around an axis and it has a $\Z/q$-action fixing the
  axis. Such a $\Z/q$-action is unique up to isotopy. Thus the link of
  $F_\gamma$ can be recovered from the link of $P_\gamma$ by
  quotienting by this $\Z/q$--action.
\end{proof}

\section{Restriction formula}\label{sec:formula}

In this section, we apply the restriction formula   to compute
the Milnor number of a spreading $F_{\gamma}$. Let us first
recall the formula.

\begin{proposition}{\rm(\cite[1.2]{T0})} \label{magicTeissier} Let
  $(Y,0) \subset (\C^{n+1},0)$ be a germ of hypersurface with isolated
  singularity, let $H$ be a hyperplane of $\C^{n+1}$ such that $H
  \cap Y$ has isolated singularity.  Assume that $(z_0,\ldots,z_n)$ is
  a system of coordinates of $\C^{n+1}$ such that $H$ is given by
  $z_0=0$. Let $\operatorname{proj} \colon (Y,0) \rightarrow (\C,0)$
  be the restriction of the function $z_0$. Then
  $$m_H = \mu(Y) + \mu(Y \cap H)\,,$$
  where $\mu (Y)$ and $\mu (Y \cap H)$ denote the Milnor numbers
  respectively of $(Y,0) \subset (\C^{n+1},0)$ and $(Y \cap H,0)
  \subset (H,0)$, and where $m_H $ is the multiplicity of the origin
  $0$ as the discriminant of the morphism $\operatorname{proj}$.\qed
\end{proposition}

 The multiplicity $m_H$ is defined as follows (\cite{T0}). Assume
that an equation of $(Y,0)$ is   $f(z_0,z_1,\ldots,z_n)=0$  and let
$(\Gamma,0) $ be the curve in $(\C^{n+1},0)$ defined by
$\frac{\partial f}{\partial z_1} = \cdots = \frac{\partial f}{\partial
  z_n} = 0$.  Then $m_H = (Y,\Gamma)_0$, the intersection at the
origin of the two germs $(Y,0)$ and $(\Gamma,0)$.
  
\begin{remark} Applying this when $n=1$, i.e., in the case of a plane
  curve $(Y,0) \subset (\C^2,0)$, we obtain:
$$m_H = \mu(Y) + \operatorname{mult} (Y) - 1\,,$$
where $\operatorname{mult} (Y)$ denotes the multiplicity of  $(Y,0)$.
\end{remark}

\begin{proposition}\label{prop:mu-spreading} Let $(X,0) \subset (\C^3,0)$ be an isolated hypersurface singularity. Let $\ell\colon \C^3 \rightarrow
\C^2$ be a generic linear projection for  $(X,0)$ and   let $(\gamma,0) \subset
  (\C^2,0)$  be an irreducible plane curve  which is not a branch of the discriminant curve of $\ell |_X$.  The Milnor number at $0$ of the
  spreading $(F_{\gamma},0)$ can be computed in terms of the following
  data:
  \begin{enumerate}
  \item the multiplicity of the surface $(X,0)$;
  \item the topological type of the triple $(X,\Pi,L_{\gamma})$ where
    $(\Pi,0) \subset (X,0) $ denotes the polar curve of $\ell |_X$.
  \end{enumerate}
\end{proposition}

\begin{proof} Choose coordinates in $\C^3$ such that $\ell=(x,y)$, and
  consider a Puiseux expansion of $\gamma$ as in Section \ref{sec:test
    curves}.  Applying the restriction formula \ref{magicTeissier} to
  the projection $\operatorname{proj}\colon (F_{\gamma},0) \rightarrow
  (\C,0)$ defined as the restriction of the function $w$, we
  obtain: 
  $$\mu(F_{\gamma}) =
  (F_{\gamma},\Gamma)_0-\operatorname{mult}(F_{\gamma}) +1\,,$$ where
  $(\Gamma,0)$ has equation
  $$\frac{\partial F}{\partial z}(x(w),y(w),z)=0\,.$$

  We have easily: $\operatorname{mult} (F_{\gamma}) =
  \operatorname{mult}(X,0)$.
 
  Moreover, the polar curve $\Pi$  is
  the set $\{(x,y,z)\in X:\frac{\partial F}{\partial
    z}(x,y,z)=0\}$. Since the map $\theta\colon F_\gamma\to L_\gamma$
  defined by $\theta (w,z)=(x(w),y(w),z)$ is a bijection, the
  intersection multiplicity $(F_{\gamma},\Gamma)_0$ in $\C^2$ equals
  the intersection multiplicity at $0$ on the surface $(X,0)$ of the
  two curves $L_{\gamma}$ and $\Pi$.
 \end{proof}

 \section{Constancy of  projected liftings and of polar rates}
 \label{sec:constancy}
We use again the notations of Section \ref{sec:Ztob} and we consider a Zariski equisingular family $(X_t,0)$   of $2$-dimensional complex hypersurfaces.

 \begin{proposition}\label{prop:constancy}   
   The following data is constant through the family $(X_t,0)$:
\begin{enumerate}
\item\label{it:constancy1} the multiplicity  of the surface $(X_t,0)$,
\item\label{it:constancy2} the  topological type of the pair $(X_t,\Pi_t)$.
\end{enumerate}
\end{proposition}
\begin{proof}
  \eqref{it:constancy1} follows (in all codimension) from the chain of
  implications: Zariski equisingularity $\Rightarrow$ Whitney
  conditions (Speder \cite{Sp}) $\Rightarrow$ $\mu^*$-constancy (e.g.,
  \cite{BS3}; see also \cite{Va1} in codimension $2$).
  \eqref{it:constancy2} follows immediately since $(X_t,\Pi_t)$ is a
  continuous family of branched covers of $\C^2$ of constant covering
  degree and with constant topology of the branch locus.
\end{proof}

\begin{corollary} \label{Z-test-curves}  Consider a   family of irreducible  plane curves 
  $(\gamma_t,0) \subset (\C^2,0)$  such that   $\gamma_t$ is not a branch of $\Delta_t$ and  the topological type of   $(\gamma_t\cup\Delta_t,0)$ is constant as $t$ varies. Let
  $\mathscr L' \colon \C^3 \times \C^{n-3}\to\C^2\times \C^{n-3}$ be a generic linear projection for $(\mathfrak{X},0)$ which is also generic for
  the family of  liftings   $(L_{\gamma_t})_t$.
 Then the  family of  projected liftings
  $P_{\gamma_t}=\mathscr L' (L_{\gamma_t})$ has constant topological type.
\end{corollary}

\begin{proof} According to Propositions \ref{prop:mu-spreading} and
  \ref{prop:constancy}, the Milnor number $\mu(F_{\gamma_t})$ is
  constant. The result follows by the L\^e-Ramanujan theorem 
  for a family of plane curves (\cite{LR}) and Lemma \ref{lem:top}. 
\end{proof}

 Let us now consider  a branch  $(\Delta'_{t})_t$  of the family of discriminant curves $(\Delta_t)_t$. For
 each $t$, let $s(\Delta'_{t})$ denote the  {polar rate} of
 $\Delta'_{t}$.

\begin{lemma} \label{prop:polar constancy}
 The polar rate $s(\Delta'_{t})$ does not depend on $t \in D_{\delta}$.
  \end{lemma}

\begin{proof}
   Let $r \in \frac{1}{N}\N$ and $\lambda \in \C^*$ and let $(\gamma_t,0)\subset(\C^2,0)$ be the family of  irreducible curves defined  by 
$$  y=\sum_{i\geq N} a_i(t) x^{i/N} + \lambda x^{r}\,,$$
where $ y=\sum_{i\geq N} a_i(t) x^{i/N}$ is  a Puiseux expansion of
$\Delta'_{t}$.  Since $r \in \frac{1}{N}\N$,  the curves $(\Delta_t \cup \gamma_t,0)$ have constant topological type as $t$ varies. Then, by Corollary \ref{Z-test-curves},     the projected liftings
$P_{\gamma_t}$ have constant topological type. We take $r$ big enough
to be sure that  the   curve $\gamma_t$ is in a $\Delta$-wedge about 
  $\Delta'_t$. Let   $P'_{\gamma_t}$ be the union of components of $P_{\gamma_t}$ which are in a polar wedge about $\Pi'_t$.  Then the family $(P'_{\gamma_t})_t$ also has constant topological type. By  Lemma
\ref{lemma:constant}, this
implies the constancy of the rates $q_t(r)=  \frac{s(\Delta'_{t} )+ r}{2}$.  Obviously $q_t(s(\Delta'_t))=s(\Delta'_t)$ and $q_t(r)<r$ for each $r>s(\Delta'_t)$. Since $q_t(r)$ is constant through the family,   the inequality $q_t(r)< r$ holds for each $t$
or for none. This proves the constancy
of the polar rate $s(\Delta_t')$ as $t$ varies in $D_{\delta}$. 
\end{proof}

\section{Proof that Zariski  equisingularity implies
  Lipschitz triviality}\label{sec:proofZtoL}

 We use  the notations of Sections \ref{sec:notations} and  \ref{sec:Ztob}. For $t$ fixed in $D_{\delta}$ and  $K_0>0$ sufficiently large, recall   (Section \ref{sec:Ztob}) that ${\mathcal B}_{K_0,t} $ denotes the neighbourhood of $\Pi_t$ in $X_t$ where the local bilipschitz constant $K(\underbar{x},t)$ of $ \mathscr L |_{X_t}   $ is bigger than $K_0$.

The aim of this section is to complete the proof that Zariski equisingularity implies Lipschitz
triviality.  

  According to Lemma \ref{prop:polar constancy}, the polar rate of each branch of $(\Delta_t,0)$ is constant as $t$ varies. Then, we can consider the carrousel decomposition for $(\Delta_t,0)$ obtained by truncating the Puiseux expansion of  each branch at the first term which has exponent greater or equal to the corresponding polar rate and where truncation does not affect the topology of $(\Delta_t,0)$. In other words, we consider the complete carrousel decomposition for $(\Delta_t,0)$ as defined in Section \ref{sec:intermediate}, obtaining a carrousel decomposition depending on the parameter $t$, as introduced in Section \ref{sec:foliated carrousel}. In particular, each component of $(\Delta_t,0)$ is contained in a $D$-piece which is a  $\Delta$-wedge about it. 

We consider the corresponding foliated carrousels as described in Section \ref{sec:foliated carrousel} and we denote by $B_{\epsilon,t}$ the ball $B^4\times\{t\}$ equipped with this foliated carrousel.   Then, we consider a trivialization of this family of carrousels as in Propositions \ref{prop:phi on V} and \ref{prop:phi}. 
 The following result is a corollary of Proposition \ref{prop:phi}:

\begin{lemma}\label{cor:Phi}
There exists a commutative diagram  
$$
\begin{CD}
  (\mathfrak{X},0)\cap (\C^3\times D_\delta) @>\Phi>> (X_0,0)\times D_\delta\\
@V \mathscr L |_{\mathfrak{X}} VV @V \mathscr L |_{X_0} \times id VV\\
(\C^2,0)\times D_\delta @>\phi>>  (\C^2,0)\times D_\delta
\end{CD}
$$
such that $\phi$ is the map of the above proposition and $\Phi$ is semi-analytic, and
inner bilipschitz except possibly in the set $C_{K_0}:= \bigcup_{t \in D_{\delta}} {\mathcal B}_{K_0,t} $.
\end{lemma}
\begin{proof}
  $\Phi$ is simply the lift of $\phi$ over the branched cover $ \mathscr L |_{\mathfrak{X}}$. The
  map $ \mathscr L |_{\mathfrak{X}}$ is a local diffeomorphism with Lipschitz constant bounded
  above by $K_0$ outside $C_{K_0}$ so $\Phi$ has Lipschitz coefficient
  bounded by $K_0^2$ times the Lipschitz bound for $\phi$ outside
  $C_{K_0}$. The semi-analyticity of $\Phi$ is because $\phi$, $ \mathscr L |_{\mathfrak{X}}$
  and $ \mathscr L |_{X_0} \times id$ are semi-analytic.
\end{proof}

We will show that the map $\Phi$ of Lemma \ref{cor:Phi}
is bilipschitz with respect to the outer metric after modifying it as
necessary within the   set $C_{K_0}$.

\begin{proof}
 We first make the modification in ${\mathcal B}_{K_0,0}$. Recall  that according to Proposition \ref{prop:thinzones},  ${\mathcal B}_{K_0,t} $ can be approximated for each $t$ by a   polar wedge about $\Pi_t$, and then its image ${\mathcal N}_{K_0,t}=  \mathscr L |_{X_t} ({\mathcal B}_{K_0,t})$ can be approximated by a   $\Delta$-wedge about the polar curve $\Delta_t$, i.e., by a   union of $D$-pieces of the foliated \carrousel\ $B_{\epsilon,t}$. 

  Let ${\mathcal N}'_{K_0,0}$ be one of these $D$-pieces, and let  $q$ be its rate.   
   By     \cite[Lemma 12.1.3]{BNP},   any component ${\mathcal B}'_{K_0,0}$ of ${\mathcal B}_{K_0,0}$ over  ${\mathcal N}'_{K_0,0}$ is also a  D(q)-piece in $X_0$. 
       
   By Proposition \ref{prop:phi}, the inverse image under the map
   $\phi_t$ of ${\mathcal N}'_{K_0,0}$ is ${\mathcal N}'_{K_0,t}$.
   Then the inverse image under the lifting $\Phi_t$ of ${\mathcal
     B}'_{K_0,0} \subset X_0$ will be the   polar wedge  ${\mathcal
     B}'_{K_0,t}$ in $X_t$ with the same   polar rate. So we can adjust
   $\Phi_t$ as necessary in this zone to be a bilipschitz equivalence
   for the inner metric and to remain semi-analytic.

  We can thus assume now that $\Phi$ is semi-analytic and bilipschitz
  for the inner metric, with Lipschitz bound $K$ say. We will now show
  it is also bilipschitz for the outer metric.

  We first consider a pair of points $\underbar p_1,\underbar p_2$ of
  $X_t\cap B^6_\epsilon$ which lie over the same point $\underbar p\in
  B_{\epsilon,t}$ (we will say they are ``vertically aligned'').  We may
  assume, by moving them slightly if necessary, that $\underbar p$ is
  on a closed curve $\gamma$ of our foliation of   $B_{\epsilon,t}$ with
  Puiseux expansion
\begin{align*}
  x(w)&=w^q\\
  y(w)&=a_1w^{q_1} + a_2 w^{q_2} + \ldots + a_{n-1} w^{q_{n-1}} + a_n
  w^{q_n} + \dots
\end{align*}
Let $\mathscr L'$ be a different generic linear projection as in Corollary
\ref{Z-test-curves}.  If $w=w_0$ is the parameter of the point
$\underbar p\in \gamma$, consider the arc $C(s)=(x(sw_0),y(sw_0))$,
$s\in[0,1)$ and lift it to $X_t$ to obtain a pair of arcs $\underbar
p_1(s),\underbar p_2(s)$, $s\in[0,1)$ with $\underbar p_1(1)=\underbar
p_1$ and $\underbar p_2(1)=\underbar p_2$.  We assume that the
distance $d(\underbar p_1(s),\underbar p_2(s))$ shrinks faster than
linearly as $s \to 0$, since otherwise the pair is uninteresting from
the point of view of bilipschitz geometry.

Note that the distance $d(\underbar p_1(s),\underbar p_2(s))$ is a
multiple $kd(\mathscr L'(\underbar p_1(s)),\mathscr L'(\underbar p_2(s)))$ (with $k$
depending only on the projections $\mathscr L$ and $\mathscr L'$). Now $d(\mathscr L'(\underbar
p_1(s)),\mathscr L'(\underbar p_2(s)))$ is to high order $s^rd(\mathscr L'(\underbar
p_1(1)),\mathscr L'(\underbar p_2(1)))$ for some rational $r>1$. Moreover, by
Corollary \ref{Z-test-curves}, if we consider the corresponding
picture in $X_0$, starting with $\Phi(\underbar p_1)$ and
$\Phi(\underbar p_2)$ we get the same situation with the same exponent
$r$. So to high order we see that $d(\Phi(\underbar
p_1(s)),\Phi(\underbar p_2(s)))/d(\underbar p_1(s),\underbar p_2(s))$
is constant.

There is certainly an overall bound on the factor by which $\Phi$
scales distance between vertically aligned points $\underbar p_1$ and
$\underbar p_2$ so long as we restrict ourselves to the compact
complement in $\mathfrak{X}\cap (B^6_\epsilon\times D_\delta)$ of an open
neighbourhood of $\{0\}\times D_\delta$. The above argument then shows
that this bound continues to hold as we move towards $0$. Thus $\Phi$
distorts distance by at most a constant factor for all vertically
aligned pairs of points.

Finally, consider any two points $\underbar p_1,\underbar p_2\in X_0
\cap B^6_\epsilon $ and their images $\underbar q_1,\underbar q_2$ in
$\C^2 $. We will assume for the moment that neither $\underbar p_1$ or
$\underbar p_2$ is in the approximated polar wedge ${\mathcal
  B}_{K_0,0}$. The outer distance between $\underbar p_1$ and
$\underbar p_2$ is the length of the straight segment joining
them. Let $\gamma$ be the image of this segment, so $\gamma$ connects
$\underbar q_1$ to $\underbar q_2$. If $\gamma$ intersects the
approximated $\Delta$-wedge ${\mathcal N}_{K_0,0}$, we can modify it
to a curve which avoids ${\mathcal N}_{K_0,0}$ and which has length
less than $\pi$ times the original length of $\gamma$. Lift $\gamma$
to a curve $\gamma'$ starting at $\underbar p_1$. Then $\gamma'$ ends
at a point $\underbar p'_2$ which is vertically aligned with
$\underbar p_2$.  Let $\gamma''$ be the curve obtained by appending
the vertical segment $\underbar p'_2\underbar p_2$ to $\gamma'$. We
then have:
$$len(\underbar p_1 \underbar p_2)\leq len( \gamma'')\,,$$
where $len$ denotes length. 
On the other hand,  $len(\gamma') \leq K_0  len(\gamma)$, 
$len(\gamma)\le \pi len(\underbar q_1\underbar q_2)$, and since the segment $\underbar q_1\underbar q_2$ is the projection of $\underbar p_1\underbar p_2$, $len(\underbar q_1\underbar q_2)\le len(\underbar p_1\underbar p_2)$. Thus $len(\gamma')\le \pi K_0 len(\underbar p_1 \underbar p_2)$. If we connect the segment $\underbar p_1\underbar p_2$ to
$\gamma'$ at $\underbar p_1$ we get a curve from $\underbar p_2$ to
$\underbar p'_2$, so   $len(\underbar p'_2\underbar p_2) \leq (1+\pi K_0)len(\underbar p_1\underbar p_2)$ and as $len(\gamma'') = len(\gamma')+ len(\underbar p'_2\underbar p_2) $, we then obtain:  
$$ len( \gamma'') \leq (1+2 K_0 \pi) len(\underbar p_1 \underbar p_2)\,.$$

 We have thus shown that up to
a bounded constant the outer distance between two points can be
achieved by following a path in $X$ followed by a vertical segment.

We have proved this under the assumption that we
  do not start or end in the  approximated polar wedge ${\mathcal B}_{K_0,0}$. Now, take two other projections $\mathscr L'$ and $\mathscr L''$ such that for   $K_0$  sufficiently large, the corresponding  approximated polar wedges $\mathcal B_{K_0,0}={\mathcal B}_{K_0,0}(\mathscr L)$, ${\mathcal B}_{K_0,0}(\mathscr L')$ and ${\mathcal B}_{K_0,0}(\mathscr L'')$   are pairwise disjoint outside the origin. Then  if  $\underbar p_1$ and $\underbar p_2$  are any two points in $X_0 \cap B^6_{\epsilon}$, they are outside the   polar wedge  for at least one of $\mathscr L$, $\mathscr L'$ or $\mathscr L''$ and we conclude  as before. 
 
The same argument applies to $X_0\times D_\delta$. Now paths in $\mathfrak{X}$
are modified by at most a bounded factor by $\Phi$ since $\Phi$ is
inner bilipschitz, while vertical segments are also modified by at
most a bounded factor by the previous argument.
Thus outer metric is modified by at most a bounded factor.
\end{proof}

\end{document}